%% file: Diagrams.tex
\documentclass{article}

\usepackage{intro}
\usepackage[margin=1in]{geometry}
\usepackage{cite}
\usepackage{hyperref}
\usepackage{url}

\title{Diagrams in the mod~$p$ cohomology of Shimura curves}
\author{Andrea Dotto and Daniel Le}
\date{}

\begin{document}

\maketitle

\begin{abstract}
We prove a local-global compatibility result in the mod~$p$ Langlands program for~$\GL_2(\bQ_{p^f})$. Namely, given a global residual representation~$\rbar$ appearing in the mod $p$ cohomology of a Shimura curve that is sufficiently generic at~$p$ and satisfies a Taylor--Wiles hypothesis, we prove that the diagram occurring in the corresponding Hecke eigenspace of mod $p$ completed cohomology is determined by the restrictions of~$\rbar$ to decomposition groups at~$p$. If these restrictions are moreover semisimple, we show that the $(\varphi,\Gamma)$-modules attached to this diagram by Breuil give, under Fontaine's equivalence, the tensor inductions of the duals of the restrictions of $\rbar$ to decomposition groups at~$p$.
\end{abstract}

\input{introduction}

\tableofcontents

\section{Preliminaries in representation theory.}\label{repprelim}

\subsection{$S$-operators and Jacobi sums.}
Throughout, for $0 < i \leq q-1$ we will write
\begin{gather*}
S_i = \sum_{\lambda \in \bF_q^\times} [\lambda]^i \fourmatrix{\lambda}{1}{1}{0}\\
S_i^+ = \sum_{\lambda \in \bF_q^\times} [\lambda]^i \fourmatrix{1}{0}{\lambda}{1}
\end{gather*}
and we put $S_0 = \fourmatrix{0}{1}{1}{0} + S_{q-1}, S_0^+ = \fourmatrix{1}{0}{0}{1} + S_{q-1}^+$. These are elements of the group algebra~$\mO_E[\GL_2(k_L)]$, where the Teichm\"uller lift is taken in~$W(k_L) \subseteq W(k_E)$. 

The operators~$S_i$ play a prominent role in~\cite{BPmodp}, whereas it is sometimes simpler to compute with the~$S_i^+$. If~$r$ is an integer, we will write $S_r, S_r^+$ for the operators indexed by a number in $\{0, \ldots, q-1\}$ congruent to~$r$ mod~$q-1$, and we will explicitly say which one we mean for $0, q-1$ on a case-by-case basis. 

For $0 \leq a, b \leq q-1$, introduce the Jacobi sum
\begin{displaymath}
\bJ(a, b) = \sum_{\alpha \in \bF_q} [\alpha]^a [1-\alpha]^b.
\end{displaymath}
Usually, the convention is that all powers of~$[0]$ are equal to~$0$ except the $0$-th power, which is equal to~$1$. This is reflected in the definition of the $S$-operators, and it results in $\bJ(a, 0) = \sum_{\alpha \in \bF_q} [\alpha]^a$, which vanishes unless $a = 0, q-1$. Our computations will give rise to nonzero constants, and so will never contain Jacobi sums of this type. For this reason and not to disrupt convention, we introduce the sum~$\bJ_0(a, b)$ that is defined in the same way but with the assumption that~$[0]^0 = 0$. Recall the following results.
\begin{thm}[Stickelberger]\label{Stickelbergertheorem}
Let~$0 < a, b < q-1$ be two integers such that $a + b \not = q-1$. Write their $p$-adic expansions as $a = \sum_i a_i p^i, b = \sum_j b_j p^j$, and similarly for $a+b = \sum_{k = 0}^{f-1} (a+b)_k p^k + (q-1)Q$. Then in~$\mO_E$ we have the equality
\begin{displaymath}
\bJ(a, b) = U(a, b)p^{u(a, b)} + C
\end{displaymath}
with $v_p(C) > u(a,b)$ and
\begin{gather*}
u(a, b) = \frac{1}{p-1}\left ( \sum_{j=0}^{f-1} p-1- \left ( a_j + b_j -(a+b)_j \right ) \right )\\
U(a, b) = (-1)^{f-1+u(a, b)} \frac{\prod_j a_j! b_j!}{\prod_j (a+b)_j!} \in \bZ_p^\times.
\end{gather*}
\end{thm}
\begin{lemma}\label{sumdivide}
In the situation of the theorem above, assume $a+b = q-1$. Then $\bJ(a, b) = (-1)^{a+1}$.
\end{lemma}
\begin{proof}
Write $\bJ(a, b) = \sum_{\alpha \in \bF_q \backslash \{0, 1 \}} \left [ \frac{\alpha}{1-\alpha} \right ]^a = \sum_{x \in \bF_q \backslash \{0, -1 \}} [x]^a = -(-1)^a$ since $\alpha \mapsto [\alpha]^a$ is a nontrivial character of~$\bF_q^\times$ (see~\cite[Theorem 8.1]{IrelandRosen}).
\end{proof}

\begin{lemma}\label{relations2}
Let~$0 < i, j < q-1$, so that $\bJ_0(i, j) = \bJ(i, j)$ is not zero by Theorem \ref{Stickelbergertheorem} and Lemma \ref{sumdivide}. Then
\[
S_i^+ S_j^+ =
\begin{dcases}
\bJ(i, j) S_{i+j}^+ &\mbox{if } i+j \neq q-1 \\
(-1)^{i+1}S_0^+ + (-1)^iq \fourmatrix{1}{0}{0}{1} &\mbox{if } i+j = q-1.
\end{dcases}
\]
\end{lemma}
\begin{proof}
We work in this proof with $[0]^n = 0$ for all $n \in \bZ$. Assume that $q-1$ does not divide~$i+j$. We expand 
\begin{align*}
S_i^+S_j^+ & = \sum_{\lambda, \mu \in \bF_q^\times}[\lambda]^i[\mu]^j\fourmatrix{1}{0}{\lambda + \mu}{1} \\
& = \sum_{\lambda \in \bF_q^\times, \nu \in \bF_q \backslash \{\lambda \}} [\lambda]^i [\nu - \lambda]^j \fourmatrix{1}{0}{\nu}{1} = \sum_{\lambda \in \bF_q^\times, \nu \in \bF_q} [\lambda]^i [\nu - \lambda]^j \fourmatrix{1}{0}{\nu}{1}\\
& = \sum_{\lambda, \nu \in \bF_q^\times}[\nu^{-1}\lambda]^i[1-\nu^{-1}\lambda]^j[\nu]^{i+j}\fourmatrix{1}{0}{\nu}{1} + (-1)^j \left ( \sum_{\lambda \in \bF_q^\times}[\lambda]^{i+j} \right )\fourmatrix{1}{0}{0}{1}\\
& =  \sum_{\nu \in \bF_q^\times} \left ( \sum_{\alpha \in \bF_q^\times}[\alpha]^i[1-\alpha]^j \right )[\nu]^{i+j} \fourmatrix{1}{0}{\nu}{1} + 0 = \left ( \sum_{\alpha \in \bF_q^\times}[\alpha]^i[1-\alpha]^j \right )S_{i+j}^+.
\end{align*}
The vanishing of~$\left ( \sum_{\lambda \in \bF_q^\times}[\lambda]^{i+j} \right )$ occurs because, by assumption, $[\lambda]^{i+j}$ is a nontrivial character of~$\bF_q^\times$. 

Otherwise, $i+j = q-1$ and the second summand equals~$q-1$, so
\begin{displaymath}
S_i^+ S_{q-1-i}^+ = \left ( \sum_{\alpha \in \bF_q^\times}[\alpha]^i[1-\alpha]^{q-1-i} \right )S_{q-1}^+ + (-1)^{q-1-i}(q-1)\fourmatrix{1}{0}{0}{1}
\end{displaymath}
and the claim follows from Lemma~\ref{sumdivide} since $\bJ(i, q-1-i) = \bJ_0(i, q-1-i)$.
\end{proof}

\begin{lemma}\label{relations3}
Let $i$ and $j$ be as in Lemma \ref{relations2}. 
Then 
\[
S_i S_j^+ =
\begin{dcases}
\bJ(i, j) S_{i+j} &\mbox{if } i+j \neq q-1 \\
(-1)^{i+1}S_0 + (-1)^iq \fourmatrix{0}{1}{1}{0} &\mbox{if } i+j = q-1.
\end{dcases}
\]
\end{lemma}
\begin{proof}
This follows from Lemma \ref{relations2} and the identities $S_i = \fourmatrix{0}{1}{1}{0} S_i^+$ for all $0 \leq i \leq q-1$.
\end{proof}

Now we extend Lemma \ref{relations2} to more general products. 
The next proposition will be applied together with Lemma~\ref{S_0^+vanishing}, and the reduction mod~$p$ of~$\beta$ will turn out to  be closely related to Breuil's constants.

\begin{pp}\label{pp:contraction}
Let $k \geq 2$ and suppose that $0< i_1,\ldots, i_k < q-1$ are integers such that $q-1$ divides $\mI_t = \sum_{m=1}^t i_m$ if and only if $t=k$.
Then there exist nonzero $\alpha, \beta \in W(k_L)$ such that
\begin{gather*}
S^+_{i_1} S^+_{i_2} \cdots S^+_{i_k} = \alpha S_0^+ + \beta \fourmatrix{1}{0}{0}{1},\\
v_p(\beta) = \frac{1}{p-1} \left ( \sum_{\substack{1 \leq t \leq k \\ 0 \leq j \leq f-1}}p - 1 - i_{t, j}  \right ),\\ 
v_p(\alpha) = v_p(\beta) - f,
\end{gather*}
the leading term of $\beta$ is the Teichm\"uller lift of
\begin{equation}\label{eqn:cexpressions}
c(\beta) = (-1)^{v_p(\beta) + (k-2)(f-1)} \prod_{\substack{1 \leq t \leq k\\0 \leq j \leq f-1}} i_{t, j}!
\end{equation}
and the leading term of~$\alpha$ is the Teichm\"uller lift of~$-c(\beta)$.
\end{pp}
\begin{proof}
First we record a lemma.
\begin{lemma}\label{sums}
For all~$j$, the sum $\mI_{k-1, j} + i_{k, j}$ is equal to~$p-1$.
\end{lemma}
\begin{proof}
By definition, $\mI_{k-1} = \sum_{j=0}^{f-1} \mI_{k-1, j}p^j + Q(q-1)$ for some~$Q \geq 0$. 
We are assuming that $q-1$ divides $\mI_{k-1} + i_{k}$, hence $i_k + \sum_{j=0}^{f-1} \mI_{k-1, j}p^j $ is divisible by~$q-1$. However, both summands are contained in~$[1, q-2]$, hence their sum must be equal to~$q-1$. The lemma follows since $q -1 = \sum_{j=0}^{f-1} (p-1)p^j$ and $i_{k, j}, \mI_{k, j} \in [0, p-2]$ for all~$j$.
\end{proof}
Now we induct on~$k$. When~$k =2$, Lemma~\ref{relations2} states
\begin{equation}\label{eqn:basecase}
S_{i_1}^+S_{i_2}^+ = (-1)^{i_2+1}S_0^+ + (-1)^{i_2}q \fourmatrix 1 0 0 1.
\end{equation}
By Lemma~\ref{sums} we have $i_{1, j} + i_{2, j} = p-1$, which implies the claim on~$v_p(\alpha)$ and~$v_p(\beta)$. For the rest, we recall the identity $x!(p-1-x)! \equiv (-1)^{x+1} \mod p$, valid for integers $x \in [0, p-2]$. 
It implies that the product term in~(\ref{eqn:cexpressions}) is $(-1)^{i_{2, 0} + \cdots + i_{2, j-1} + f} $, which equals~$(-1)^{i_2 + f}$ since~$p$ is odd, and we are done since~$k = 2$ is even.

Assuming the proposition for~$k$, we apply Lemma~\ref{relations2} and we deduce
\[
S^+_{i_1} S^+_{i_2} \cdots S^+_{i_{k+1}} = \bJ(i_1, i_2) S^+_{\mI_2}S^+_{i_3}\cdots S^+_{i_{k+1}}.
\]
The inductive assumption applies to the product at the right-hand side (where~$\mI_2$ is identified with its unique representative in~$[1, q-2]$), and since~$\bJ(i_1, i_2) \ne 0$ we see that the existence of~$\alpha, \beta$ already follows by induction.

To compute their valuation and leading term, we apply Theorem~\ref{Stickelbergertheorem} and write
\[
v_p\,\bJ(i_1, i_2) = \frac 1 {p-1} \left ( \sum_{j=0}^{f-1} p-1-i_{1, j} - i_{2, j} + \mI_{2, j} \right )
\]
and the leading term of~$\bJ(i_1, i_2)$ as
\[
(-1)^{f-1 + v_p \bJ(i_1, i_2)} \frac{\prod_j i_{1, j}!i_{2, j}!}{\prod_j \mI_{2, j}!}.
\]
The claim follows by induction.
\end{proof}

We now study the action of these $S^+$-operators on $\mO_E[\GL_2(k_L)]$-modules with the goal of proving analogous results involving $S$-operators.
They will apply with less generally than Lemma~\ref{relations2}, in that we will not obtain an identity in the group algebra $\mO_E[\GL_2(k_L)]$.
In the following section, we will introduce the special case of $\mO_E[\GL_2(k_L)]$-modules coming from $\mO_E$-lattices in tame types.

\begin{lemma}\label{shift}
Let~$v$ be an $H$-eigenvector in an $\mO_E[\GL_2(k_L)]$-module, of eigencharacter $\chi: \fourmatrix{a}{0}{0}{d} \mapsto [a]^r \eta(ad)$. Then~$S_i v$ has eigencharacter $\chi^s \alpha^{-i} = \chi \alpha^{-r-i}$ and $S_i^+ v$ has eigencharacter $\chi \alpha^{i}$.
\end{lemma}
\begin{proof}
This follows from an explicit computation based on the equalities
\begin{gather*}
\fourmatrix{a}{0}{0}{d}\fourmatrix{\lambda}{1}{1}{0} = \fourmatrix{ad^{-1}\lambda}{1}{1}{0} \fourmatrix{d}{0}{0}{a}\\
\fourmatrix{a}{0}{0}{d}\fourmatrix{1}{0}{\lambda}{1} = \fourmatrix{1}{0}{a^{-1}d\lambda}{1}\fourmatrix{a}{0}{0}{d}.
\end{gather*}
\end{proof}
If~$v$ is as in Lemma~\ref{shift}, it follows that $S_i S_j v$ has $H$-eigencharacter $\chi^s \alpha^{-(i-j-r)}$. The following lemma is proved in~\cite[Th\'eor\`eme~2.5.2]{BD}. 

\begin{lemma}\label{relations1}
Let~$v$ be an $I$-eigenvector in an $\mO_E[\GL_2(k_L)]$-module such that $I$ acts on $v$ through the $H$-eigencharacter~$\chi: \fourmatrix{a}{0}{0}{d} \mapsto [a]^r\eta(ad)$. 
Assume that $i \not = 0$ and that $q-1$ does not divide neither $i-j-r$ nor $i-j$. Then
\begin{displaymath}
S_i S_j v = \eta(-1)\bJ_0(i, -j - r) S_{i-j-r} v.
\end{displaymath}
\end{lemma}
\begin{proof}
Throughout this proof, we will use the convention that $[0]^n = 0$ for all $n \in \bZ$. First we notice that for $s, t \in \bF_q$ we have
\[
\fourmatrix{s}{1}{1}{0}\fourmatrix{t}{1}{1}{0} = 
\begin{dcases} 
\hspace{4mm} \fourmatrix{1}{s}{0}{1} &\mbox{if } t=0. \\
\\
\hspace{4mm} \fourmatrix{s+t^{-1}}{1}{1}{0}\fourmatrix{t}{1}{0}{-t^{-1}} &\mbox{if } t \not = 0. \\
\end{dcases}
\]
Assuming in addition that $j \not = 0$, we compute $S_i S_j v$ as follows:
\begin{align*}
S_i S_j v  & = \sum_{s, t \in \bF_q^\times}[s]^i[t]^{-j}\fourmatrix{s+t}{1}{1}{0}\fourmatrix{t^{-1}}{1}{0}{-t}v = \eta(-1) \sum_{s, t \in \bF_q^\times}[s]^i[t]^{-j-r}\fourmatrix{s+t}{1}{1}{0}v\\
& = \eta(-1)\sum_{s \in \bF_q^\times} \sum_{x \in \bF_q \backslash \{s \}}[s]^i[x-s]^{-j-r}\fourmatrix{x}{1}{1}{0} v = \eta(-1)\sum_{s, x \in \bF_q}[s]^i[x-s]^{-j-r}\fourmatrix{x}{1}{1}{0} v =\\
& = \eta(-1)\sum_{x \in \bF_q^\times}[x]^{i-j-r}\left ( \sum_{s \in \bF_q} [x^{-1}s]^i[1-x^{-1}s]^{-j-r} \right ) \fourmatrix{x}{1}{1}{0}v + \eta(-1)(-1)^{-j-r}\sum_{s \in \bF_q}[s]^{i-j-r}\fourmatrix{0}{1}{1}{0}v.
\end{align*}
Since $q-1$ does not divide $i-j-r$, the sum $\sum_{s \in \bF_q}[s]^{i-j-r} = 0$ and
\begin{displaymath}
S_i S_j v = \eta(-1) \bJ_0(i, -j-r) S_{i-j-r} v.
\end{displaymath}
When $j = 0$, we have an extra term
\begin{displaymath}
S_i S_0 v   = \sum_{s, t \in \bF_q^\times}[s]^i\fourmatrix{s+t}{1}{1}{0}\fourmatrix{t^{-1}}{1}{0}{-t}v + \sum_{s \in \bF_q} [s]^i \fourmatrix{1}{s}{0}{1}v
\end{displaymath}
but the second sum vanishes, because $I_1$ fixes~$v$ and $q-1$ does not divide $i-j = i$.
\end{proof}

To explain the hypothesis in the above lemma, we remark that we will apply it to an $I$-eigenvector~$v$ with eigencharacter~$\chi$ in the principal series type $\Ind_I^{\GL_2(\mO_L)} \chi$. Then the assumptions of the lemma are satisfied when the $H$-eigencharacter of $S_i S_j v$ has multiplicity one in the type, because the only characters with multiplicity two are~$\chi^s$ and~$\chi = \chi^s \alpha^r$ (see Lemma \ref{basistype}), and $S_i S_j v$ has eigencharacter $\chi^s \alpha^{-(i-j-r)}$.

\subsection{Lattices in tame $\GL_2(L)$-types.}\label{lattices}

We will only consider \emph{nonscalar} tame $K$-types for~$\GL_2(L)$.
By the tameness assumption, we can also consider these types as $\GL_2(k_L)$-modules.
Note that these types satisfy mod~$p$ multiplicity one.
The symbol~$\theta$ will usually denote an $\mO_E$-lattice in a tame type~$\theta[1/p]$. 

We are going to study the action of $S$-operators on tame types, recalling and generalizing results of~\cite{Breuilmodp} and~\cite{BD}. 
If~$\theta$ is an $\mO_E$-lattice in a tame type, then the restriction $\theta|_H$ splits into a direct sum of $H$-eigenspaces (because the order of the diagonal torus~$H$ is coprime to~$p$), and we begin with some general results about the action of $S$-operators on $H$-eigenvectors.

\begin{lemma}\label{S_0^+vanishing}
Let~$\theta$ be an $\mO_E$-lattice in a tame type for~$\GL_2(L)$ over~$E$. 
If~$\chi: H \to \mO_E^\times$ occurs in~$\theta$ with multiplicity one and~$x_\chi$ is an $H$-eigenvector of eigenvalue~$\chi$, then~$S_0^+ x_{\chi} =~0$. 
\end{lemma}
\begin{proof}
The restriction of~$\theta[1/p]$ to the lower mirabolic subgroup is the induction of a regular character of the lower unipotent (if $\theta[1/p]$ is cuspidal) or the direct sum of such an induction and two characters (if~$\theta[1/p]$ is principal series). 
Since the centre is acting by a character, this decomposition is stable under the lower-triangular Borel subgroup, and since~$\chi$ has multiplicity one we see that~$x_{\chi}$ is contained in the direct summand that is nontrivial under the mirabolic subgroup. 
Then the claim follows from the definition of $S_0^+$ since the restriction of this summand to the lower-triangular unipotent subgroup is a direct sum of \emph{nontrivial} characters.
\end{proof}

Now we construct two kinds of relations between $H$-eigenvectors in $\theta[1/p]$ depending on whether $\theta[1/p]$ is a principal series or cuspidal representation. 
The reason we need to treat these cases separately is that the action of~$S_i^+$-operators is not transitive on a basis of $H$-eigencharacters of a principal series type.

\begin{lemma}\label{basistype}
There are $q-1$ different $H$-eigencharacters in~$\theta$. If~$\theta[1/p]$ is cuspidal they all have multiplicity one, and if $\theta[1/p] \cong \Ind_{I}^{\GL_2(\mO_L)}\chi$ they all have multiplicity one except $\chi$ and~$\chi^s$ which have each multiplicity two. 
There is an $\mO_E$-basis~$\{ x_\xi \}_\xi$ of~$\theta$ such that
\begin{enumerate}
\item $x_\xi$ is an eigenvector of~$H$ with eigencharacter~$\xi$.
\item if~$\theta[1/p]$ is a cuspidal type and $\xi_1 \not = \xi_2$, then there exist an integer $0<i(\xi_1, \xi_2)^+ < q-1$ and a scalar $\alpha_{\xi_1, \xi_2}^+ \in E^\times$ such that
\begin{displaymath}
x_{\xi_2} = \alpha_{\xi_1, \xi_2}^+ S_{i(\xi_1, \xi_2)^+}^+ x_{\xi_1}.
\end{displaymath}
\item \label{item:multone} if $\theta[1/p]$ is a principal series type, $\xi_1 \not = \xi_2$, and either~$\xi_2$ has multiplicity one or both $x_{\xi_1}, x_{\xi_2}$ are $I_1$-fixed, then there exist an integer $0\leq i(\xi_1, \xi_2) \leq q-1$ and a scalar $\alpha_{\xi_1, \xi_2} \in E^\times$ such that
\begin{displaymath}
x_{\xi_2} = \alpha_{\xi_1, \xi_2} S_{i(\xi_1, \xi_2)} x_{\xi_1}.
\end{displaymath}
\end{enumerate}
\end{lemma}
\begin{proof}
In the cuspidal case, begin with an arbitrary $H$-eigenvector~$x_\xi$. The span of the vectors $\fourmatrix{1}{0}{\lambda}{1}x_\xi$ for~$\lambda \in \bF_q$ is stable under~$H$ and under the lower unipotent subgroup, hence under the lower mirabolic subgroup of~$\GL_2(k_L)$. 
Since~$\theta$ is a lattice in a cuspidal type, the proof of Lemma~\ref{S_0^+vanishing} shows that these vectors span $\theta[1/p]$ (because the induction to the mirabolic of a regular character is irreducible).
By Lemma~\ref{S_0^+vanishing}, $S_0^+$ acts by~$0$ on~$\theta[1/p]$. 
Hence the vectors $\fourmatrix{1}{0}{\lambda}{1} x_{\xi}$ for $\lambda \in \bF_q^\times$ already span $\theta[1/p]$, and since $\dim_E \theta[1/p] = q-1$ they form a basis of~$\theta[1/p]$ over~$E$. 
Applying a Vandermonde matrix to this basis, we find that the $S_i^+ x_\xi$ for $1 \leq i \leq q-1$ form a basis of~$\theta[1/p]$ over~$E$.
The result follows by scaling since~$\theta$ has a basis of $H$-eigenvectors. 
We remark that $S_{q-1}^+ x_{\xi} = -x_{\xi}$, and then the existence of~$\alpha_{\xi_1, \xi_2}^+$ follows from Lemma~\ref{relations2}.

In the principal series case, let~$\varphi$ be a generator of an $I_1$-stable $\mO_E$-line in~$\theta$ and write the $H$-character of~$\varphi$ as $\fourmatrix{a}{0}{0}{d} \mapsto a^r\eta(ad)$. 
By the discussion in~\cite[\S 2]{Breuilmodp} certain scalar multiples of the $S_i \varphi$ for $0 \leq i \leq q-1$ form a basis of~$\theta$ over~$\mO_E$ together with~$\varphi$. Let~$\{x_\xi \}$ be this basis, so that for each~$\xi$ either $x_\xi = \varphi$ or there exists~$i(\xi)$ such that $x_{\xi}$ is a multiple of~$S_{i(\xi)}\varphi$. 

We now consider the claim in~\eqref{item:multone}. 
If~$x_{\xi_1} = \varphi$ then the claim holds by construction. 
Now notice that if~$\xi_2$ has multiplicity one then we have an equation 
\begin{displaymath}
S_{i(\xi_2) + i(\xi_1) + r} S_{i(\xi_1)} \varphi = \eta(-1)\bJ_0(i(\xi_2) + i(\xi_1) + r, -i(\xi_1)-r) S_{i(\xi_2)} \varphi.
\end{displaymath}
by Lemma~\ref{relations1}, which applies because $q-1$ does not divide any of~$i(\xi_2)$ and $i(\xi_2) + r$. (If it did, then $x_{\xi_2}$ would have eigencharacter $\chi^s\alpha^{-i(\xi_2)} = \chi^s \text{ or } \chi$ respectively, contradicting multiplicity one.)
In the case $i(\xi_2) + i(\xi_1) + r$ is divisible by~$q-1$, we take~$i(\xi_1,\xi_2) = q-1$.

Finally, in the case where both $x_{\xi_1}$ and $x_{\xi_2}$ are fixed by~$I_1$ the claim follows from the fact that $S_0$ sends a nonzero $I_1$-fixed vector to a nonzero $I_1$-fixed vector of the opposite $H$-eigencharacter.
\end{proof}

\begin{lemma}\label{relationsbasistype}
Let~$\{x_\xi \}$ be a basis as in Lemma~\ref{basistype}. If~$\xi_1 \not = \xi_2$, and $\xi_2$ appears with multiplicity one in~$\theta$, then one can find a relation
\begin{displaymath}
x_{\xi_2} = \alpha_{\xi_1, \xi_2}^+ S_{i(\xi_1, \xi_2)^+}^+ x_{\xi_1}
\end{displaymath}
with $\alpha_{\xi_1, \xi_2}^+ \in E^\times$ and $0<i(\xi_1, \xi_2)^+ < q-1$, even if~$\theta[1/p]$ is a principal series type.
\end{lemma}
\begin{proof}
By lemma \ref{basistype}, we only need to consider the case when $\theta[1/p]$ is a principal series representation.
By Lemma~\ref{shift}, such an integer $i(\xi_1, \xi_2)^+$ is uniquely determined since $\xi_1 \not = \xi_2$ and the operator~$S_j^+$ multiplies the character by~$\alpha^{j}$. 
By the multiplicity one assumption, it suffices to prove that $S_{i(\xi_1, \xi_2)^+}^+ x_{\xi_1}$ is nonzero. 
For this, the matrix identity
\begin{displaymath}
\fourmatrix{\lambda}{1}{1}{0}\fourmatrix{1}{0}{\mu}{1} = \fourmatrix{\lambda+\mu}{1}{1}{0}
\end{displaymath}
and the same computation as Lemma~\ref{relations2} prove that there is a relation 
\[
S_j S_{i(\xi_1, \xi_2)^+}^+ = \bJ(j, i(\xi_1, \xi_2)^+) S_{i(\xi_1, \xi_2)^++j}
\]
whenever $0 < j < q-1$ and $i(\xi_1, \xi_2)^+ + j \not = q-1$. 

The type~$\theta[1/p]$ contains $q-3$ eigencharacters with multiplicity one. If $S_{i(\xi_1, \xi_2)^+}x_{\xi_1} \not = 0$, then we left multiply by $\fourmatrix 0 1 1 0$ and deduce the claim that $S_{i(\xi_1, \xi_2)^+}^+ x_{\xi_1}$ is nonzero. 
Otherwise, since $p \geq 5$, we can take a~$0<j<q-1$ not equal to $q-1-i(\xi_1, \xi_2)^+$ so that the (unambiguously defined) vector~$S_{i(\xi_1, \xi_2)^++j}x_{\xi_1}$ does not vanish, by Lemma~\ref{basistype}.
The claim follows as the Jacobi sum $\bJ(j, i(\xi_1, \xi_2)^+)$ does not vanish.
\end{proof}

We say that an inclusion of $\mO_E$-lattices $\theta^1 \to \theta^2$ in a tame type $\theta[1/p]$ is \emph{saturated} if the induced map $\theta^1 \to \lbar{\theta}^2$ is nonzero. 
By~\cite[Lemma~4.1.1]{EGS}, if~$\sigma \in \JH(\thetabar)$ then there exists a unique homothety class of lattices in~$\theta[1/p]$ with irreducible cosocle~$\sigma$. 
We will denote a representative for this class by $\theta^\sigma$.
We often choose representatives so that certain inclusions are saturated.
If we fix two representatives~$\theta^{\sigma_1}$ and~$\theta^{\sigma_2}$ and a saturated inclusion $\theta^{\sigma_1} \subset \theta^{\sigma_2}$ as above, then there exists a unique $n \in \bN$ such that $p^n \theta^{\sigma_2} \subset \theta^{\sigma_1}$ and the inclusion is saturated (since $\theta[1/p]$ and thus the $\mO_E$-lattices $\theta^{\sigma_1}$ and~$\theta^{\sigma_2}$ are defined over $W(k_E)[1/p]$ and $W(k_E)$, respectively). 
When we speak of a saturated inclusion $\theta^{\sigma_2} \to \theta^{\sigma_1}$, we will mean the map
\[
\theta^{\sigma_2} \xrightarrow{p^n} p^n\theta^{\sigma_2} \subset \theta^{\sigma_1}.
\]

\begin{defn}\label{defn:PS}
Let $\chi$ be an $\mO_E^\times$-valued character of $H$.
Let $\theta(\chi)$ be the $\mO_E$-lattice $\Ind_I^K \chi$ in a principal series tame type, realized as a space of functions on~$K$.
Let $\varphi^\chi$ be the unique element of $\theta(\chi)$ supported on $I$ such that $\varphi^\chi(1) = 1$.
Let $\sigma(\chi)$ be the (irreducible) cosocle of $\theta(\chi)$, and denote the image of $\varphi^\chi$ in $\sigma(\chi)$ by $\varphi^\chi$ as well.
For a type $\theta$ with Jordan--H\"older factor isomorphic to $\sigma(\chi)$, we write $\theta^\chi$ for short rather than $\theta^{\sigma(\chi)}$.
\end{defn}

\begin{lemma}\label{generators1}
The representation 
\[ 
\langle \mO_E[\GL_2(k_L)] \cdot S_0\varphi^{\chi^s} \rangle \subset \theta(\chi^s)
\] 
is isomorphic to~$\theta(\chi^s)^\chi$, and the inclusion is saturated.
If~$\xi$ is neither $\chi$ nor $\chi^s$, then both 
\[
\langle \mO_E[\GL_2(k_L)]\cdot S_{i(\chi^s,\xi)} \varphi^{\chi^s} \rangle,\, \langle \mO_E[\GL_2(k_L)]\cdot S^+_{i(\chi^s,\xi)^+} \varphi^{\chi^s} \rangle \subset \theta(\chi^s)
\] 
are isomorphic to~$\theta(\chi^s)^\xi$, and the inclusion is saturated.
\end{lemma}
\begin{proof}
For the first part, since $S_0 \varphi^{\chi^s}$ is an $I$-eigenvector with character $\chi$, it generates a sublattice isomorphic to~$\theta(\chi)$. 
Moreover, $S_0 \varphi^{\chi^s}$ is nonzero in~$\theta(\chi^s)/\varpi_E$.

Let $\theta(\chi^s)^\xi\subset \theta(\chi^s)$ be a saturated inclusion.
The image of~$\theta(\chi^s)^\xi$ in $\theta(\chi^s)/\varpi_E$ is nonzero, and is characterized as being the only $k_E[\GL_2(k_L)]$-subrepresentation of $\thetabar(\chi^s)$ with irreducible cosocle isomorphic to~$\sigma(\xi)$. 
Equivalently, the image is characterized as being the minimal subrepresentation of $\overline{\theta}(\chi^s)$ containing~$\sigma(\xi)$ as a Jordan--H\"older factor, or as the minimal subrepresentation containing $\xi$ as an $H$-eigencharacter (using that the $H$-character $\xi$ has multiplicity one in~$\theta(\chi^s)$ by Lemma \ref{basistype}).
In particular, this image is generated over~$k_E[\GL_2(k_L)]$ by the image of both~$S_{i(\chi^s,\xi)}\varphi^{\chi^s}$ and $S^+_{i(\chi^s,\xi)^+} \varphi^{\chi^s}$ by Lemmas \ref{basistype} and \ref{relationsbasistype}. 
Since~$H$ has order coprime to~$p$, there exists an $H$-eigenvectors in~$\theta(\chi^s)^\xi$ lifting the images of~$S_{i(\chi^s,\xi)}\varphi^{\chi^s}$ and $S^+_{i(\chi^s,\xi)^+} \varphi^{\chi^s}$. 
Since~$\xi$ has multiplicity one, these lifts are multiples of~$S_{i(\chi^s,\xi)}\varphi^{\chi^s}$ and $S^+_{i(\chi^s,\xi)^+} \varphi^{\chi^s}$, which therefore are contained in~$\theta(\chi^s)^\xi$.
We conclude by noting that both~$S_{i(\chi^s,\xi)}\varphi^{\chi^s}$ and~$S^+_{i(\chi^s,\xi)^+} \varphi^{\chi^s}$ generate~$\theta(\chi^s)^\xi$ over~$\mO_E[\GL_2(k_L)]$ because their images each generate~$\lbar \theta(\chi^s)^\xi$ over $k_E[\GL_2(k_L)]$. 
\end{proof}

We end this section by recalling a relation between $S_0$, $\Pi = \fourmatrix 0 1 p 0$, and the Hecke operator $U_p$ on a principal series type.
\begin{lemma}\label{U_peigenvalues}
Let~$\pi = \Ind_{B(L)}^{\GL_2(L)}(\chi_1| \cdot | \otimes \chi_2)$ for tamely ramified characters~$\chi_i: L^\times \to E^\times$ (this is a smooth unnormalized parabolic induction) with different restrictions to~$\mO_L^\times$. Let~$\varphi \in \pi$ be an $I_1$-fixed vector with $H$-eigencharacter $\chi_1|_{\mO_L^\times} \otimes \chi_2|_{\mO_L^\times}$. Then
\begin{displaymath}
\Pi \varphi = q^{-1} \chi_1(p) S_0 \varphi.
\end{displaymath}
\end{lemma}
\begin{proof}
Compare with~\cite[Th\'eor\`eme~2.5.2]{BD}. 
The absolute value~$|\cdot|$ is normalized for~$L$, so that $|p| = q^{-1}$.
By the Iwasawa decomposition, the space of $I_1$-fixed vectors in~$\pi$ is two-dimensional and contains two different $H$-eigencharacters. An explicit computation gives
\begin{displaymath}
S_0 \Pi = U_p = \sum_{\lambda \in \bF_q} \fourmatrix{p}{[\lambda]}{0}{1}.
\end{displaymath}
This preserves the $H$-eigencharacter, hence $U_p \varphi = \alpha \varphi$ for some scalar~$\alpha$. Evaluating the function~$\alpha\varphi$ at the identity of~$\GL_2(L)$, we obtain $\alpha \varphi(1) = q \cdot \chi_1(p)|p| \varphi(1)$ hence $\alpha = \chi_1(p)$ since $\varphi(1) \not = 0$ as $\varphi$ is supported in $B I_1$. Now $\pi$ is isomorphic to $\Ind_{B(L)}^{\GL_2(L)}(\chi_2| \cdot | \otimes \chi_1)$, as it is the normalized induction that is Weyl-invariant, and~$\pi$ is the normalized induction of~$\chi_1|\cdot|^{1/2} \otimes \chi_2|\cdot|^{1/2}$.
Hence it is also true that $U_p \Pi \varphi = \chi_2(p) \Pi \varphi$. Substituting $S_0 \Pi$ for $U_p$, we obtain
\begin{displaymath}
q^{-1}\chi_1(p)\chi_2(p) S_0 \varphi = \chi_2(p) \Pi \varphi
\end{displaymath}
and the claim follows.
\end{proof}

\subsection{Tame types and inertial local Langlands.}\label{sec:ill}

In this section, we parametrize tame $\GL_n(L)$-types and describe the generic tame inertial local Langlands correspondence.

\subsubsection{Deligne--Lusztig induction.}\label{sec:DL}

We recall a description of Deligne--Lusztig induction for the group $\bG_0 = \Res_{k_L / \bF_p} \GL_n$. 
We will eventually take $n$ to be $2$, but we take $n$ to be general in this section and in parts of \S \ref{sec:defring}.
Let $F$ be the $p$-th power, relative Frobenius endomorphism of $\bG_0$ (relative and absolute Frobenius morphisms coincide here).
The induced endomorphism of $\bG_0(\overline{\bF}_p)$ coincides with the endomorphism coming from $p$-power Frobenius endomorphism on $\cbF_p$.
Let $\bG$ be $\bG_0 \times_{\bF_p} \cbF_p$.
Under the isomorphism 
\begin{equation} \label{eqn:product}
\bG \cong \prod_{\iota:k_L \hookrightarrow \cbF_p} \GL_{n, \cbF_p} \cong\prod_{i=0}^{f-1} \GL_{n, \cbF_p}
\end{equation}
given by the bijection $i \leftrightarrow \iota_{-i}$ for $i\in \bZ/f$, the Frobenius endomorphism $F$ of the group~acts on points by $(A_0, \ldots, A_{f-1}) \mapsto (A_{f-1}^{(p)}, A_0^{(p)}, \ldots, A_{f-2}^{(p)})$. 

There is a diagonal torus~$\bT_0$ isomorphic to $\Res_{k_L / \bF_p}(\mathbb{G}_m)^n$, and an upper-triangular Borel subgroup.
The Weyl group~$W = W(\bG_0, \bT_0)$ is isomorphic to $S_n^{\times f}$, and the $F$-conjugacy classes in~$W$ are represented by $s= (s, 1, \ldots, 1)$ where~$s$ represents conjugacy classes of the symmetric group~$S_n$.
Hence the conjugacy classes of rational maximal tori are represented by $\Res_{k_L/\bF_p}(\bT_0')$, where $\bT_0'$ runs over conjugacy classes of rational maximal tori of ${\GL_n}_{/k_L}$. 
We write~$\bT = \bT_0 \times_{\bF_p} \cbF_p$ and we identify the character lattice~$X(\bT)$ with~$(\bZ^n)^{\oplus f}$ in the usual way using (\ref{eqn:product}). 
The group~$\Gal(\cbF_p/\bF_p)$ acts on the character lattice of~$\bT$ and we denote by~$\pi$ the action of the arithmetic Frobenius element $x \mapsto x^p$. There is also an action of~$F$ by $F(\chi) = \chi \circ F$. One has the equality $(\pi \chi)(F(t)) = \chi(t)^p$, hence $\pi(\chi_0, \ldots, \chi_{f-1})= (\chi_{f-1}, \chi_0, \ldots, \chi_{f-1})$.
Let $\Lambda_R \subset X(\bT)$ be the root lattice.

We let $\bG_0^\der \subset \bG_0$ be the derived subgroup $\Res_{k_L / \bF_p} \SL_n$. 
We similarly define $\bG^\der$, $\bT_0^\der$, and $\bT^\der$.
We write $\Lambda_W$ for $X(\bT^\der)$.
The root lattice $\Lambda_R$ in $X(\bT)$ is canonically identified with the root lattice in $\Lambda_W$.
Let $X^0(\bT)$ denote the kernel of the natural restriction map $X(\bT) \ra X(\bT^\der) = \Lambda_W$.
Under the identification $X(\bT) \cong (\bZ^n)^{\oplus f}$, $X^0(\bT)$ is identified with $\bZ^{\oplus f}$ where $\bZ \hookrightarrow \bZ^n$ is embedded diagonally.

If~$\mu$ is a character of~$\bT$ and $w$ is an element of $W$, then there is an associated virtual representation $R_w(\mu)$ of~$\GL_n(k_L)$ over~$E$. 
There is an action of $W \ltimes X(\bT)$ (the semidirect product defined by the action $w \cdot \mu = F(w)\mu$ for $w\in W$ and $\mu\in X(\bT)$) on~$W \times X(\bT)$ by
\begin{displaymath}
(w_1, \mu_1) \cdot (w_2, \mu_2) = (w_1 w_2 F(w_1)^{-1}, w_1(\mu_2) + (F - w_1 w_2 F(w_1)^{-1})\mu_1 ).
\end{displaymath} 
The fibers of the map $(w, \mu) \mapsto R_w(\mu)$ are precisely the orbits of this action, by \cite[Lemma 4.2]{herzig}.

Rather than giving the definition of~$R_w(\mu)$ in full, we spell it out in the main case of interest when $n=2$.
Let $1, s \in W$ be the identity and the element $s = (s, 1, \ldots, 1)$ (nontrivial in the first component).
We identify~$\mu$ with an $f$-tuple $(\mu_1, \mu_2) = ((\mu_{1, 0}, \mu_{2, 0}), \ldots, (\mu_{1, f-1}, \mu_{2, f-1})) \in (\bZ^2)^{\oplus f}$. Then 
\begin{displaymath}
R_1(\mu) \cong \Ind_{B_2(k_L)}^{\GL_2(k_L)}\left ( [-]^{\sum_{t=0}^{f-1}\mu_{1, t}p^t} \oplus [-]^{\sum_{t=0}^{f-1}\mu_{2, t}p^t} \right )
\end{displaymath}
whereas $R_s(\mu)$ is the negative of the Deligne--Lusztig induction of the character $[-]^{\sum_{t=0}^{f-1}\mu_{1, t}p^t + p^f\sum_{t=0}^{f-1}\mu_{2, t}p^{t}}$ of~$k_{L, 2}^\times$ to~$\GL_2(k_L)$.

\subsubsection{Tame inertial types.}

Let $d$ be a natural number, and set $e'$ to be $p^{df}-1$. 
Let $\pi' \in \overline{E}$ be a root of $u^{e'}+p$.
We define the character 
\begin{align*}
[\omega_{df}]: I_{\bQ_p} &\ra \overline{\mO}_E^\times \\
g &\mapsto \frac{g(\pi')}{\pi'}.
\end{align*}
This definition does not depend on the choice of root $\pi'$.
We assume that $E$ is sufficiently large so that this character is valued in $\mO_E^\times$, and let $\omega_{df}$ be its reduction modulo $\varpi_E$.

For $(w,\mu) \in W\times X(\bT)$, there is an $s \in W$ of the form $(s,1,\ldots,1)$ such that $R_w(\mu) \cong R_s(\ba)$ for some $\ba \in X(\bT)$, as explained in \S \ref{sec:DL}.
Let $d$ be the order of $s$ and suppose that $\ba$ maps to $(a_{k,j})_{k,j}$, with $1\leq k\leq n$ and $0 \leq j\leq f-1$, under the identification $X(\bT) \cong (\bZ^n)^{\oplus f}$.
Let
\[
\ba_k = \sum_{i=0}^{f-1} a_{k,i}p^i.
\]
Then we define $\tau(w,\mu):I_L \ra \GL_n(\mO_E)$ to be 
\[
\bigoplus_{k=1}^n [\omega_{df}]^{\sum_{i=0}^{d-1} \ba_{s^i(k)}}p^{if}.
\]
Note that this does not depend on the choice of $s$ and $\ba$ above.
Moreover, $\tau(w,\mu)$ is an inertial type for the Weil group~$W_L$, i.e.~it admits an extension to $W_L$.

If $(w,\mu)$ is a \emph{good pair} in the sense of~\cite[\S 2.2]{LLL}, then $\tau(w,\mu) \cong \tau(w',\mu')$ implies that $R_w(\mu) \cong R_{w'}(\mu')$, and moreover this Deligne--Lusztig representation is irreducible (see \cite[Proposition 2.2.4]{LLL}).
If $(w,\mu)$ is a good pair, we set $\sigma(\tau(w,\mu))$ to be $R_w(\mu)$ and note that this gives a well-defined injective map from a subset of tame inertial types to tame types for $K$ or, equivalently, irreducible representations of $\GL_n(\bF_q)$.
This defines a tame inertial local Langlands (see \cite[Proposition 2.4.1(i)]{EGH}).

Noting, for any tame inertial type $\tau$, the reduction~$\taubar$ is a sum of characters whose Teichm\"uller lift is $\tau$, we let $V(\taubar(w,\mu))$ be $R_w(\mu)$ for $(w,\mu)$ a good pair as above.
The map $V$ is closely related to the maps $V$ and $V_\phi$ that appear in \cite{herzig,LLL} and \cite{GHS}, respectively, but we will not need this in the sequel.

Again, we make the above explicit in the case $n=2$.
Let $\rhobar:G_L \ra \GL_2(k_E)$ be a continuous Galois representation. There are two possibilities for the semisimplification of $\rhobar|_{I_L}$, namely
\begin{gather}
\fourmatrix{\omega_f^{\sum_{t=0}^{f-1}\mu_{1, t}p^t}}{0}{0}{\omega_f^{\sum_{t=0}^{f-1}\mu_{2, t}p^t}} \text{ in the reducible case}\\
\fourmatrix{\omega_{2f}^{\sum_{t=0}^{f-1}\mu_{1, t}p^t + p^f\sum_{t=0}^{f-1}\mu_{2, t}p^t}}{0}{0}{\omega_{2f}^{\sum_{t=0}^{f-1}\mu_{2, t}p^t+ p^f\sum_{t=0}^{f-1}\mu_{1, t}p^t}} \text{in the irreducible case},
\end{gather}
where $\mu_{k,t} \in\bZ$ for each $k \in \{1, 2\}$ and $0 \leq t \leq f-1$.
We let~$s_{\rhobar} \in S_2$ be~$1$ in the first case and the nontrivial element in the second, and we write~$\mu = (\mu_1, \mu_2)$. Then the definitions above specialize to $V(\rhobar^\semis|_{I_L}) = R_{s_{\rhobar}}(\mu)$.

\subsection{Serre weights.}\label{sec:extgph}

Recall that a \emph{Serre weight} for $\GL_2(k_L)$ is (the isomorphism class of) an irreducible $\GL_2(k_L)$-representation over $k_E$, or equivalently an irreducible $\GL_2(\mO_L)$-representation over $k_E$.
Every Serre weight is isomorphic to a representation of the form
\begin{equation} \label{eqn:serrewt} 
(\eta \circ \det) \otimes \bigotimes_{j\in \bZ/f} \Sym^{r_j} (k_L^2 \otimes_{\sigma_j,k_L} k_E)
\end{equation}
where $0 \leq r_j \leq p-1$ for all $j\in \bZ/f$ and $\eta: k_L^\times \ra k_E^\times$ is a character.
We will use the shorthand $(r_0,\ldots,r_{f-1})\otimes \eta$ for the Serre weight in (\ref{eqn:serrewt}). 

We will also use the following description of Serre weights.
We use the notation of \S \ref{sec:DL} with $n=2$.
Then $\bG_0(\bF_p) = \GL_2(k_L)$.
If $\mu\in X(\bT)$ is dominant and $0 \leq \langle \alpha,\mu \rangle \leq p-1$ for all positive roots $\alpha$ of $\bG$, then we let $L(\mu)$ be the (unique up to isomorphism) irreducible representation of $\bG$ with highest weight $\mu$.
Let $F(\mu)$ be the restriction of $L(\mu)$ to $\GL_2(k_L)$.
Then every Serre weight is isomorphic to $F(\mu)$ for some $\mu$ as above. 
Moreover, $F(\mu)$ and $F(\lambda)$ are isomorphic if and only if $\mu-\lambda\in (F-1)X^0(\bT) = (p-\pi)X^0(\bT)$ (see \cite[Lemma 9.2.4]{GHS}).
If $\mu$ is $(\mu_{1,j},\mu_{2,j})_{j\in \bZ/f}$, then $F(\mu)$ is isomorphic to the representation 
\[
(\mu_{1,j} - \mu_{2,j})_{j\in \bZ/f} \otimes \sigma_0^{\sum_{j = 0}^{f-1} p^j \mu_{2,j}}
\]
in the notation introduced above. (Recall that $\sigma_0: k_L \to k_E$ is our fixed choice of embedding.)

\paragraph{Extension graph.}

In this section, $n=2$ so that in the notation of \S \ref{sec:DL}, $\bG_0$ is the group $\Res_{k_L / \bF_p} \GL_2$.
Let $\mu\in X(\bT)$ be dominant such that $0 \leq \langle \alpha,\mu \rangle \leq p-1$ for all positive roots $\alpha$ of $\bG$.
Recall that $\Lambda_W$ is $X(\bG_0^\der)$ where $\bG_0^\der\subset \bG_0$ is the subgroup $\Res_{k_L / \bF_p} \SL_2$.
Let $\Lambda_W^\mu\subset \Lambda_W$ be the subset
\[
\Lambda_W^\mu = \{\omega \in \Lambda_W: 0 \leq \langle \alpha, \mu + \omega' \rangle \leq p-1 \text{ for all positive roots~$\alpha$ of~$\bG$}\}
\]
where $\omega' \in X(\bT)$ is any lift of $\omega$.
We now give an equivalent definition, based on \cite[\S 2]{LLLM2}, of an injective map 
\[
\ft_\mu: \Lambda_W^\mu \ra X(\bT)/(F-1)X^0(\bT)
\]
from \cite[\S 2]{LMS}.

We first define a map $\ft_\mu': X(\bT) \ra X(\bT)/(F-1)X^0(\bT)$ as follows.
Let $W_a$ and~$\tld W$ be the affine Weyl group and the extended affine Weyl group of~$\bG$, respectively.
The natural inclusion $X(\bT) \ra \tld{W}$ induces an isomorphism $X(\bT)/\Lambda_R \ra \tld{W}/W_a$.
Let $\Omega\subset \tld{W}$ be the stabilizer of the dominant base ($1$-)alcove for $\bG$.
Then we have a splitting $\tld{W} = \Omega \ltimes W_a$ of the inclusion $W_a \subset \tld{W}$.

From these considerations, for $\omega'\in X(\bT)$ there exists a unique $\tld{w}'\in \Omega$ such that $\tld{w}'t_{\pi^{-1}(\omega')} \in W_a$.
We then define $\ft_\mu'(\omega')$ to be 
\[
\tld{w}'\cdot (\mu- \eta + \omega') \pmod{(F-1)X^0(\bT)},\]
where~$\eta_j = (1, 0)_j$ and the action of $\tld{w}'$ is by the $p$-dot action.
It is easy to check that $\ft_\mu'$ is well-defined and factors through the natural quotient map $X(\bT) \ra \Lambda_W$.
Then $\ft_\mu$ is defined to be the restriction of this map to $\Lambda_W^\mu\subset \Lambda_W$.
For $\omega\in \Lambda_W^\mu$, the isomorphism class of $F(\ft_\mu(\omega))$ is well-defined.

Now let $\omega_j \in \Lambda_W$ be the fundamental dominant weight which is nonzero (only) in embedding $j$.
For a subset $J\subset \bZ/f$, let $\omega_J$ be $\sum_{j\in J} \omega_j$.
There is a unique element $w_J \in W$ such that $w_Jt_{-\pi^{-1}(\omega_J)}$ fixes the dominant base ($1$-)alcove for $\bG^{\der}$.
Explicitly, $w_J \in W$ is the element that is nontrivial on exactly~$\delta_\red(J)$.
Let $\tld{w}_J$ be $w_Jt_{-\pi^{-1}(\omega_J)}$.
(Recall that~$\pi^{-1}$ is a left shift on~$X(\bT)$.)

\begin{pp}\label{pp:jh}
Let $w \in W$ and $\nu \in \Lambda_R$ such that $-\eta + \nu+ w\omega_J \in \Lambda_W^\mu$ for all subsets $J \subset \bZ/f$.
Then the set $\JH(\overline{R_w(\mu-\eta+\nu)})$ is given by
\[
\big\{F(\ft_\mu(-\eta + \nu+ w\omega_J))\big\}_{J\subset \bZ/f}
\]
(where by an abuse of notation $\eta$ denotes the image of $\eta$ in $\Lambda_W$). 
\end{pp}
\begin{proof}
This follows from \cite[\S A, Theorem 3.4]{herzig} applied to the $\bF_p$-form $\Res_{k_L / \bF_p} \GL_2$ of $\GL_2^{\times f}$.
Note that in the notation of \emph{loc.~cit.}, $\gamma'_{w_1,w_2}$ is the Kronecker symbol $\delta_{w_1,w_2}$ for $w_1$ and $w_2\in W$.
Moreover, the highest weight of each character that appears in the formula for $X_w'(1,\mu-\eta+\nu)$ is in the unique $p$-restricted dominant ($p$-)alcove, and is thus the highest weight of a Serre weight.
The proposition then follows from checking that the $2^f$ highest weights coincide with $\ft_\mu(-\eta+\nu+w\omega_J)$ for $J\subset \bZ/f$ (modulo $(p-\pi)X^0(\bT)$).

Indeed, let $s\in W$ and $\epsilon \in X(\bT)$ be such that $\tld{s}=st_\epsilon$ is in $\Omega$.
Let $\tld{w}_h$ be $w_0 t_{-\eta}$.
Then $\tld{w}_h^{-1} \tld{s} t_{\pi^{-1}(-\eta+\nu-w\pi\epsilon)} \in W_a$ so that 
\begin{equation}\label{eqn:tr}
\ft_\mu(-\eta+\nu-w\pi\epsilon) = \tld{w}_h^{-1}\tld{s}\cdot(\mu-2\eta+\nu-w\pi\epsilon) \pmod{(p-\pi)X^0(\bT)}.
\end{equation}
If we let $w_1$ be $w_0 s$, then $\epsilon'_{w_0w_1}$ can be taken to be $\epsilon$ and $\rho'_{w_1}$ can be taken so that $\tld{w}_h^{-1} \tld{s} = t_{\rho'_{w_1}}w_1$ in the notation of \cite[\S 5]{herzig}.
Then (\ref{eqn:tr}) becomes $w_1\cdot (\mu-\eta+\nu -w\pi\epsilon'_{w_0w_1} - \eta) + p \rho'_{w_1}$ as in \cite[\S A, Theorem 3.4]{herzig}.
We conclude by noting that the maps $s \mapsto -\pi\epsilon \pmod{X^0(\bT)}$ and $s \mapsto w_1 = w_0 s$ defined above give bijections $W \ra \{\omega_J\}_{J\subset \bZ/f}$ and $W \ra W$, respectively. 
\end{proof}

\input{kisinwithquestions}

\section{Diagrams, patching and groupoids.} \label{sec:diagrams}

\subsection{Some generalities on diagrams.}

We define a \emph{$0$-diagram} to be a triple $(D_0,D_1,r)$ where $D_0$ is a $L^\times\GL_2(\mO_L)$-representation with central character, $D_1$ is an $N(I)$-representation (where $N(I)$ is the normalizer in $\GL_2(L)$ of $I$), and $r: D_1|_{L^\times I} \hookrightarrow D_0|_{L^\times I}$ is an injection that induces an isomorphism $D_1 \cong D_0^{I_1}$ (we have weaker hypotheses on the central character than \cite[\S 9]{BPmodp}).
We will consider only diagrams over $k_E$.
We assume throughout that the $K_1$-action on $D_0$ is trivial. 

Let $\rhobar: G_L \to \GL_2(k_E)$ be generic and let $D_0 = D_0(\rhobar)$ be the $\GL_2(k_L)$-representation over $k_E$ defined in \cite[\S 13]{BPmodp} (where it is defined over $\lbar{\bF}_p$, but descends to $k_L$ and in particular $k_E$ via $\sigma_0$). 
Let $D_1^\circ$ be $D_0^{I_1}$.
Recall that the Jordan--H\"older factors of $D_0$ appear with multiplicity one by \cite[Theorem 1.1(ii)]{BPmodp} so that in particular the isotypic space $D_1^\circ[\chi]$ has dimension at most one.

To give a diagram $(D_0,D_1,\can)$ such that $D_1|_I = D_1^\circ$ and~$\can$ is the inclusion $D_1^\circ \hookrightarrow D_0$ is equivalent to extending the action of $I$ on $D_1^\circ$ to an action of the normalizer $N(I)$ of $I$. 
Such an extension is equivalent to giving a map $\Pi: D_1^\circ \ra D_1^\circ$ taking $D_1^\circ[\chi]$ to $D_1^\circ[\chi^s]$ for all $\chi$, and such that $\Pi^2: D_1^\circ \ra D_1^\circ$ acts by a nonzero scalar.

By definition, a morphism of diagrams consists of a pair~$(\psi_0, \psi_1)$ of morphisms, equivariant for~$L^\times \GL_2(\mO_L)$ and~$N(I)$ respectively, that commute with the inclusion~$r$.
It follows that every diagram is isomorphic to one for which~$D_1 = D_1^\circ$ and~$r = \can$.
We introduce an invariant of a diagram that classifies it up to isomorphism and will be accessible to computation.

\begin{defn}\label{defn:R}
Let $R: D_0^{I_1} \ra (\soc_K D_0)^{I_1}$ be the unique map so that if $D_0^{I_1}[\chi]$ is nonzero then $R|_{D_0^{I_1}[\chi]}$ is nonzero and given by $S_{i(\chi)}$ for some $0 \leq i(\chi) \leq q-1$, except when~$\chi$ appears in~$(\soc_K D_0)^{I_1}$, in which case $R|_{(\soc_K D_0)^{I_1}[\chi]}$ is the identity. 
(The existence and uniqueness of $R$ follow from~\cite[Lemma~2.7]{BPmodp}, compare also Lemma~\ref{basistype}.) 
Given~$\chi$, we write $R\chi$ for the character such that $R(D_0^{I_1}[\chi]) = (\soc_K D_0)^{I_1}[R\chi]$.
\end{defn}

\begin{rk}\label{welldefinedinteger}
In the previous definition, the operator~$S_{i(\chi)}$ is unique whenever~$\chi$ is not contained in~$(\soc_K D_0)^{I_1}$. This is because in most cases $R\chi$ has multiplicity one in the $\GL_2(\mO_L)$-representation generated by~$\chi$, and the only exception is when $R\chi = \chi$ or $R\chi = \chi^s$. 
The first case is excluded since we are assuming~$(\soc_K D_0)^{I_1}[\chi] = 0$. 
In the second case, only $S_0$ preserves $I_1$-invariants. 
So we can attach a uniquely determined integer~$i(\chi)$ to any character~$\chi \in D_0^{I_1}$ not appearing in $(\soc_K D_0)^{I_1}$. 
This integer has appeared before: relative to the tame principal series type $\Ind_I^K[\chi]$, it is what we called $i(\chi, R\chi)$ in Lemma~\ref{basistype}.
To shorten notation, we will continue writing~$i(\chi)$ for this integer, and when~$\chi = R\chi$ we will write~$S_{i(\chi)}$ for the identity map.
\end{rk}

\begin{defn}
We write~$\mathcal{G}$ for the groupoid with an object~$\bx_\xi$ for each character such that $(\soc_K D_0)^{I_1}[\xi]$ is nonzero, and morphisms freely generated by $g_\chi: \bx_{R\chi} \to \bx_{R\chi^s}$ for each nonzero $D_ 0^{I_1}[\chi]$.
\end{defn}

Let $D = (D_0,D_1,\can)$ be a diagram such that $D_1|_I = D_1^\circ$. 
Then we can construct a $1$-dimensional representation (i.e.~a character) $\chi_D$ of $\mathcal{G}$ as follows: the objects $\bx_\xi$ go to $(\soc_K D_0)^{I_1}[\xi]$ and the morphisms $g_\chi$ go to maps 
\[
(\soc_K D_0)^{I_1}[R\chi] \ra (\soc_K D_0)^{I_1}[R\chi^s] 
\]
defined by $g_\chi(R(v)) = R(\Pi v)$.

\begin{pp}
The diagram $D$ can be recovered from $\chi_D$.
\end{pp}
\begin{proof}
This is clear from the fact that $R: D_1^\circ[\chi] \ra D_1^\circ[R\chi]$ is an isomorphism for all $\chi$, hence~$\Pi$ is the only map that makes all the diagrams
\[
\begin{tikzcd}
D_1^\circ[\chi] \arrow{r}{\Pi}\arrow{d}{R} & D_1^\circ[\chi^s] \arrow{d}{R}\\
(\soc_K D_0)^{I_1}[R\chi] \arrow{r}{g_\chi} & (\soc_K D_0)^{I_1}[R\chi^s].
\end{tikzcd}
\]
commute.
\end{proof}

From a representation of a groupoid $\mG$, one obtains by restriction a representation of the group $\mG_\bx$ of automorphisms of any object~$\bx$ of~$\mG$. 
The following result is standard.

\begin{pp}\label{pp:groupoidrep}
Suppose that $X$ is a set of representatives for $\pi_0(\mG)$. 
The product of restriction functors
\[
\Rep_{\mG} \ra \prod_{\bx \in X} \Rep_{\mG_\bx}
\]
is an equivalence of categories.
\end{pp}
\begin{proof}
Consider the category~$G$ whose object set is~$X$ and $\Hom_G(\bx, \bx) = \mG_{\bx}$ for all~$\bx \in X$, while $\Hom_G(\bx_1, \bx_2)$ is empty if $\bx_1 \ne \bx_2$.
The inclusion functor $G \to \mG$ is fully faithful and essentially surjective, hence it is an equivalence of categories.
The proposition follows since a representation of~$\mG$ is a functor from~$\mG$ to vector spaces.
\end{proof}

Let $\chi_0, \ldots, \chi_{k-1}$ be characters such that $R\chi_i^s = R\chi_{i+1}$ for each $i$ (where we set $\chi_k$ to be $\chi_0$). Then there is a map 
\[
g_{\chi_{k-1}} \circ \cdots \circ  g_{\chi_0}: (\soc_K D_0)^{I_1}[R\chi_0] \ra (\soc_K D_0)^{I_1}[R\chi_0].
\]
By Proposition \ref{pp:groupoidrep}, if one determines the constants giving all such maps, one determines the isomorphism class of $D$.

\subsection{Diamond diagrams from patched modules.}\label{sec:diamond}
There is a special class of diagrams, called Diamond diagrams, attached to many suitable globalizations of~$\rhobar$: for example, representations of~$\Gal(\overline{F}/F)$, for $F/\bQ$ a totally real field, that are modular for a totally definite (or split at a precisely one real place) quaternion algebra over~$F$, and that recover~$\rhobar$ by restricting to decomposition groups (see \S \ref{sec:modpaut}). 
In \S \ref{sec:diagrams} and \ref{sec:constants}, we will work in the more general context of $\mO_E[\GL_2(L)]$-modules with an arithmetic action, in the sense of~\cite[\S 5.2]{GN} (see also \cite[\S 3.1]{CEGGPS2}), that recover these examples through Taylor--Wiles--Kisin patching (see \S \ref{sec:TW}). 
In this section, we recall the definition of $\mO_E[\GL_2(L)]$-modules with an arithmetic action and state our first main theorem.

\paragraph{}
Let $\rhobar: G_L \ra \GL_2(k_E)$ be a local Galois representation.
Recall from \S \ref{sec:defrings} the following notation.
The ring $R_{\rhobar}$ is the framed unrestricted $\mO_E$-deformation ring for $\rhobar$.
If $\psi: G_L \ra \mO_E^\times$ is a lift of $\det \rhobar\overline{\varepsilon}^{-1}$, then $R^\psi_\rhobar$ is the quotient of $R_\rhobar$ parameterizing lifts of $\rhobar$ with determinant $\psi\varepsilon$. 
We will assume throughout that $\psi$ factors through $W(k_E)^\times$.
For an inertial type $\tau$, recall that $R_{\rhobar}^{\tau}$ is the quotient of $R_{\rhobar}$ corresponding to potentially crystalline lifts of $\rhobar$ of Galois type $\tau$ and Hodge--Tate weight $(0,1)$. 
Fix a natural number $h$ and let $R_\infty$ be $R_{\rhobar}\widehat{\otimes}_{\mO_E} \mO_E[\![x_1,x_2,\ldots, x_h]\!]$ and $R_\infty^\psi$ be $R^\psi_{\rhobar}\widehat{\otimes}_{\mO_E} \mO_E[\![x_1,x_2,\ldots, x_h]\!]$.
Let $R_\infty(\tau)$ be $R_{\rhobar}^{\tau} \otimes_{R_{\rhobar}} R_\infty$ and  $R_\infty^\psi(\tau)$ be $R_{\rhobar}^{\psi,\tau} \otimes_{R_{\rhobar}^\psi} R_\infty^\psi$.

We consider the case where the determinant is not fixed (the necessary modifications for the fixed determinant case are straightforward).
By an $\mO_E[\GL_2(L)]$-module $M_\infty$ with an arithmetic action of $R_\infty$ we will mean a nonzero $\mO_E$-module with commuting actions of~$R_\infty$ and a right action of~$\GL_2(L)$ satisfying the following properties.
\begin{enumerate}
\item $M_\infty$ is a finitely generated right $R_\infty[\![\GL_2(\mO_L)]\!]$-module.
\item $M_\infty$ is projective in the category of pseudocompact right $\mO_E[\![\GL_2(\mO_L)]\!]$-modules.
\end{enumerate}
For a finite $\mO_E[\![\GL_2(\mO_L)]\!]$-module $\sigma$, we let $M_\infty(\sigma)$ be 
\[
M_\infty\otimes_{\mO_E[\![\GL_2(\mO_L)]\!]} \sigma.
\]
\begin{enumerate}
\setcounter{enumi}{2}
\item If $\theta$ is an $\mO_E$-lattice in a locally algebraic type $\sigma(\tau)$, then $M_\infty(\theta)$ is maximal Cohen--Macaulay over $R_\infty(\tau)$; and moreover
\item the action of $\mathcal{H}(\theta)$ on $M_\infty(\theta)[1/p]$ is given by the composite 
\[
\mathcal{H}(\theta) \overset{\eta}{\ra} R_{\rhobar}^{\tau}[1/p] \ra R_\infty(\tau)[1/p]
\]
with $\eta$ as in \S \ref{sec:LLC}. 
\end{enumerate}
In the fixed determinant case, we furthermore require that
\begin{enumerate}
\setcounter{enumi}{4}
\item the center $L^\times$ of $\GL_2(L)$ acts on $M_\infty$ through the fixed determinant character $\psi: L^\times \ra \mO_E^\times$.
\end{enumerate}
Then in particular, $M_\infty(-)$ is a (fixed determinant) patching functor in the sense of \cite[Definition 6.1.3]{EGS}. 
This gives access to results in \cite[\S 8-10]{EGS}.
For example, if $\sigma$ is a Serre weight, then $M_\infty(\sigma) \neq 0$ if and only if $\sigma \in W(\rhobar)$ by \cite[Theorem 9.1.1]{EGS}.
We say further that an $\mO_E[\GL_2(L)]$-module $M_\infty$ with an arithmetic action of $R_\infty$ is \emph{minimal} if
\begin{enumerate}
\setcounter{enumi}{5}
\item for $\theta$ as above, $\dim_{\lbar E}M_\infty(\theta) \otimes_{R_\infty(\tau)} \lbar E$ is at most one for any $\lbar E$-point of $R_\infty(\tau)$.
\end{enumerate}
In this case, $M_\infty(-)$ is a \emph{minimal} (fixed determinant) patching functor in the sense of \cite[\S 6.1]{EGS}. 

Let us elaborate on the action of $\mH(\theta)$ on $M_\infty(\theta)[1/p]$.
In fact, $\mH(\theta)$ acts on $M_\infty(\theta)$ by the identification of the latter with $\Hom_{\GL_2(\mO_L)}(\theta, M_\infty^\vee)^\vee$ in the following lemma and Frobenius reciprocity.
Note that $M_\infty^\vee$ inherits a natural left action of $\GL_2(L)$. 
\begin{lemma}\label{duals}
Let~$\theta$ be a finite $\mO_E[\![\GL_2(\mO_L)]\!]$-module. Then $M_\infty(\theta)$ is Pontrjagin dual to $\Hom_{\GL_2(\mO_L)}(\theta, M_\infty^\vee)$.
\end{lemma}
\begin{proof}
This follows from \cite[Lemma B.3]{GN}.
\end{proof}

The following theorem, which essentially follows from \cite[Theorem~10.1.1]{EGS}, will be fundamental to our method.

\begin{thm}\label{thm:free}
Suppose that $M_\infty$ is a minimal $\mO_E[\GL_2(L)]$-module with an arithmetic action of $R_\infty$.
Let $\tau$ be a nonscalar tame inertial type such that $R_\rhobar^\tau$ is nonzero.
If~$\sigma(\tau)^\circ$ is an $\mO_E$-lattice in $\sigma(\tau)$ with irreducible cosocle, then~$M_\infty(\sigma(\tau)^\circ)$ is free of rank one over~$R_\infty(\tau)$.
If $\psi$ factors through $W(k_E)^\times$, then the analogous fixed determinant result holds as well.
\end{thm}
\begin{proof}
If $E$ is an unramified extension of $\bQ_p$, then this follows from \cite[Theorem~10.1.1]{EGS}.
We begin with the following lemma.
Let $R_{\infty,0}$ and $R_{\infty,0}(\tau)$ be defined similarly to $R_\infty$ and $R_\infty(\tau)$, respectively, with the Witt vectors $W(k_E)$ in place of $\mO_E$. 
\begin{lemma}\label{lem:restscal}
Suppose that $M_\infty$ is an $\mO_E[\GL_2(L)]$-module with an arithmetic action of $R_\infty$.
Then the inclusion $W(k_E) \subset \mO_E$ makes $M_\infty$ a $W(k_E)[\GL_2(L)]$-module with an arithmetic action of $R_{\infty,0}$.
If $\psi$ factors through $W(k_E)^\times$, then the analogous fixed determinant result holds as well.
\end{lemma}
\begin{proof}
This is straightforward with the exception of item 3.
Let $\fm_0\subset R_{\infty,0}$ be the maximal ideal, and notice that~$R_{\infty}$ is finite free over~$R_{\infty, 0}$.
Then
\begin{align*}
\depth_{R_{\infty,0}} M_\infty(\theta) &= \depth_{R_{\infty,0}} M_\infty(\theta)+\depth_{k_E} R_\infty/\fm_0 R_\infty \\
&= \depth_{R_\infty} M_\infty(\theta) \otimes_{R_{\infty,0}} R_\infty\\
&= \depth_{R_\infty} M_\infty(\theta)\\
&= \dim \Supp_{R_\infty} M_\infty(\theta)\\
&= \dim \Supp_{R_{\infty,0}} M_\infty(\theta)
\end{align*}
where the first equality is true since~$\dim k_E = 0$, the second follows from \cite[Proposition 6.3.1]{EGA}, and the third and fifth follow from the fact that $R_\infty$ is finite free over $R_{\infty,0}$.
\end{proof}

The only place where the proof of \cite[Theorem~10.1.1]{EGS} requires that $E$ is an unramified extension of $\bQ_p$ is in the proof of \cite[Lemma 10.1.12]{EGS} ($\mO_E[\![X,Y]\!]/(XY-p)$ is a regular ring if and only if $E$ is unramified).
However, using Lemma \ref{lem:restscal}, the same proof shows that $M_\infty(\sigma(\tau)^\circ)$ is free over $R_{\infty,0}(\tau)$.
Now we deduce that
\[
M_\infty(\sigma(\tau)^\circ) \otimes_{R_{\infty, 0}(\tau)} R_\infty(\tau) \cong M_\infty(\sigma(\tau)^\circ) \otimes_{R_\infty(\tau)} \left ( R_\infty(\tau) \otimes_{R_{\infty, 0}(\tau)} R_\infty(\tau) \right )
\]
is $R_\infty(\tau)$-free.
Since $M_\infty(\sigma(\tau)^\circ)$ is a direct summand of the right-hand side (as~$R_\infty(\tau)$-modules), it is finite projective over~$R_\infty(\tau)$, hence finite free since~$R_\infty(\tau)$ is a local ring.
The rest of the proof of \cite[Theorem~10.1.1]{EGS}, where key computations are done modulo $\varpi_E$, works without modification.
\end{proof}

Let $\mathfrak{m}$ denote the maximal ideal of $R_\infty$.

\begin{defn}\label{defn:piglob}
We let $\pi(\rhobar^\vee)$ be the $\GL_2(L)$-representation $M_\infty^\vee[\mathfrak{m}]$ (where $(-)^\vee$ denotes the Pontrjagin dual). 
Let $D_0$ be $\pi(\rhobar^\vee)^{K_1}$ and $D_1$ be $\pi(\rhobar^\vee)^{I_1}$ and set $D = \mathcal{D}(\pi(\rhobar^\vee)) = (D_0,D_1,\can)$ for the inclusion $\can: D_1 \to D_0$. 
\end{defn}

\begin{rk}
We denote $M_\infty^\vee[\mathfrak{m}]$ by $\pi(\rhobar^\vee)$ and not $\pi(\rhobar)$ to emphasize that our conventions in \S \ref{sec:global} are dual to those in \cite[\S 9]{Breuilmodp}.
\end{rk}

We now suppose that $\rhobar:G_L\ra \GL_2(k_E)$ is generic as in \S \ref{Galoisreps}.
Then $D_0$ is isomorphic to $D_0(\rhobar)$ by (the proof of) \cite[Corollary 5.2]{Le}.

\begin{thm}\label{patchingdiagram}
If $\rhobar$ is generic, then the diagram~$D$ is independent of the choice of~$M_\infty$ satisfying properties~$(1)$ to~$(6)$ above.
\end{thm}

The proof of Theorem~\ref{patchingdiagram} will occupy the rest of \S 4, and will proceed by determining the compositions~$g_{\chi_{k-1}}\circ\cdots\circ g_{\chi_0}$ for every cycle of characters $\chi_0, \ldots, \chi_{k-1}$ such that $R\chi_i^s = R\chi_{i+1}$ for every $i\in \bZ/k$. 

\begin{example}
We describe the groupoid~$\mG$ when~$L = \bQ_{p^2}$ and $\rhobar$ is semisimple generic.
There are always four weights in the socle, parametrized by the subsets $J \subseteq \{0, 1\}$ through a map $J \mapsto \sigma(\lambda_J, \rhobar)$ (see sections~\ref{formalweights} and~\ref{weightdefinition}).
Hence there are four objects in the groupoid~$\mG$, corresponding to the decomposition 
\[
(\soc_K D_0(\rhobar))^{I_1} = \chi_1 \oplus \chi_2 \oplus \chi_3 \oplus \chi_4.
\]

If~$\rhobar$ is irreducible, by~\cite[Proposition~14.7]{BPmodp} the representation~$D_0(\rhobar)^{I_1}$ has dimension~$8$, and the morphisms of~$\mG$ are generated by
\begin{align*}
g_{\chi_i} &: (\soc_K D_0(\rhobar))^{I_1}[\chi_i] \ra (\soc_K D_0(\rhobar))^{I_1}[\chi_{i+1}] \\ g_{\chi_i^s} &: (\soc_K D_0(\rhobar))^{I_1}[\chi_{i+1}] \ra (\soc_K D_0(\rhobar))^{I_1}[\chi_i]
\end{align*}
(with some ordering of the characters $\chi_i$).
The diagram~$D$ is determined by the central character and a single parameter in~$k_E^\times$, corresponding to the composition $g_{\chi_3} \circ g_{\chi_2} \circ g_{\chi_1} \circ g_{\chi_0}$.

For a split reducible $\rhobar$, the dimension of~$D_0(\rhobar)^{I_1}$ is $10$, and hence there will be $10$ generating morphisms for~$\mG$.
In this case, the diagram~$D$ is determined by the central character and $5$ parameters in~$k_E^\times$ (though really one of these is determined by the central character).
\end{example}

\subsection{Lifting~$g_\chi$.}
The first step in the proof of Theorem~\ref{patchingdiagram} is to construct specific lifts of the maps~$g_\chi$ (or rather, their duals) to patched modules in characteristic zero, via composition of saturated inclusions and~$U_p$-operators. In this section, we fix a character $\chi: I/I_1 \to \mO_E^\times$ that appears in~$D_0(\rhobar)^{I_1}$.
If~$\xi$ is another character, recall from Definition \ref{defn:PS} the definitions of $\theta(\chi)$ and $\theta(\chi)^\xi$ if $\sigma(\xi)$ is a Jordan--H\"older factor of $\lbar \theta(\chi)$.

\begin{lemma}\label{generators}
The vector~$S_0\varphi^{\chi^s}$ generates~$\theta(\chi^s)^\chi$ over~$\mO_E[\GL_2(k_L)]$. 
If~$\chi \not = R\chi^s$, then the vector $S_{i(\chi^s)} \varphi^{\chi^s}$ generates~$\theta(\chi^s)^{R\chi^s}$ over~$\mO_E[\GL_2(k_L)]$. (Compare Remark~\ref{welldefinedinteger} for the definition of~$i(\chi^s)$.)
\end{lemma}
\begin{proof}
This follows from Lemma \ref{generators1}, since by Remark~\ref{welldefinedinteger} we have $i(\chi^s) = i(\chi^s, R\chi^s)$.
\end{proof}

We are going to need an isomorphism
\[
\Pi: M_\infty(\theta(\chi^s)) \to M_\infty(\theta(\chi))
\]
arising as follows.
By Lemma~\ref{duals}, it is equivalent to construct an isomorphism 
\[
\Hom_{\mO_E[\GL_2(\mO_L)]}(\theta(\chi), M_\infty^\vee) \to \Hom_{\mO_E[\GL_2(\mO_L)]}(\theta(\chi^s), M_\infty^\vee).
\]
But by Frobenius reciprocity, these $\Hom$-spaces are isomorphic to $(M_\infty^\vee)^{I_1}[\chi]$ and $(M_\infty^\vee)^{I_1}[\chi^s]$, respectively.
Hence multiplication by the matrix~$\Pi$ induces an isomorphism $(M_\infty^\vee)^{I_1}[\chi]\ra (M_\infty^\vee)^{I_1}[\chi^s]$.
We summarize this construction in the following lemma, for ease of reference.

\begin{lemma}\label{lemma:piiso}
Multiplication by~$\Pi$ on~$M_\infty^\vee$ induces an isomorphism $\Pi: M_\infty(\theta(\chi^s)) \to M_\infty(\theta(\chi))$.
\end{lemma}

We are also going to need the following related construction.
As in Lemma~\ref{generators}, there is an isomorphism $\theta(\chi) \ra \theta(\chi^s)^{\chi}$ sending $\varphi^{\chi}$ to $S_0 \varphi^{\chi^s}$. Precomposing with the isomorphism from Lemma \ref{lemma:piiso}, we get an isomorphism $M_\infty(\theta(\chi^s)^{\chi^s}) \ra M_\infty(\theta(\chi^s)^{\chi})$, which we also denote by $\Pi$.
Now consider a sequence
\[
\theta(\chi^s)^{R\chi} \to \theta(\chi^s)^{\chi} \to \theta(\chi^s)^{R\chi^s} \to \theta(\chi^s)^{\chi^s} = \theta(\chi^s)
\]
of saturated inclusions, as defined in~\S \ref{lattices}.
In this context we have the following commutative diagram, where the unlabelled arrows are induced by the saturated inclusions above.
\begin{equation}\label{diag:Up}
\begin{tikzcd}
M_\infty(\theta(\chi^s)^{R\chi}) \arrow{r} \arrow{rrddd}[swap]{U_p(\chi)} & M_\infty(\theta(\chi^s)^{\chi}) \arrow{r} \arrow{rdd}[swap]{U_p(\chi)} & M_\infty(\theta(\chi^s)^{R\chi^s}) \arrow{d} \\
& & M_\infty(\theta(\chi^s)^{\chi^s}) \arrow{d}{\Pi} \\
& & M_\infty(\theta(\chi^s)^{\chi}) \\
& & M_\infty(\theta(\chi^s)^{R\chi}) \arrow{u}
\end{tikzcd}
\end{equation}

\begin{lemma}\label{lemma:Upchi}
In the diagram~(\ref{diag:Up}) the diagonal arrows are multiplication by the image through~$\eta$ of the $U_p$-element of~$\mH(\theta(\chi)[1/p])$, and the vertical arrows are isomorphisms.
\end{lemma}
\begin{proof}
The inclusions are dual to restrictions, hence the short diagonal arrow is dual to
\begin{align*}
\Hom_{\mO_E[\GL_2(k_L)]}(\theta(\chi^s)^\chi, M_\infty^\vee) &\to \Hom_{\mO_E[\GL_2(k_L)]}(\theta(\chi^s)^\chi, M_\infty^\vee)\\
f &\mapsto f^+|_{\theta(\chi^s)^\chi}
\end{align*}
where~$f^+: \theta(\chi^s) \to M_\infty^\vee$ is characterized by the equation~$f^+(\varphi^{\chi^s}) = \Pi f(S_0 \varphi^{\chi^s})$. 
Then $f^+|_{\theta(\chi^s)^\chi}$ is characterized by the equation
\[
f^+(S_0 \varphi^{\chi^s}) = S_0\Pi f(S_0 \varphi^{\chi^s}) = U_p f(S_0\varphi^{\chi^s}),
\]
(recall from the proof of Lemma~\ref{U_peigenvalues} that $S_0 \Pi = U_p$). 
Since~$f(S_0 \varphi^{\chi^s})$ is an $I/I_1$-eigenvector with eigencharacter~$\chi$, $U_p$ is an element in $\mH(\theta(\chi)[1/p])$. 
The result now follows from \cite[Lemma 4.17(2)]{CEGGPS}.
The long diagonal arrow is the restriction of this arrow (recall that~$\sigma \mapsto M_\infty(\sigma)$ is an exact functor).

For the second claim, we only need to consider the unlabelled vertical arrows in~(\ref{diag:Up}) by Lemma \ref{lemma:piiso}.
These arrows are injective since~$\sigma \mapsto M_\infty(\sigma)$ is an exact functor.
By Nakayama's lemma, it suffices to show that $M_\infty(\theta(\chi^s)^{R\chi^s})/\fm \ra M_\infty(\theta(\chi^s))/\fm$ is surjective (and similarly for~$\chi$).
Dualizing and using that~$M_\infty^\vee[\fm]^{K(1)} \cong D_0(\rhobar)$, it suffices to show that $\Hom_K(\theta(\chi^s)/\theta(\chi^s)^{R\chi^s},D_0(\rhobar))$ is zero.
Any map in $\Hom_K(\theta(\chi^s),D_0(\rhobar))$ factors through $D_{0,\sigma(R\chi^s)}(\rhobar)$ in the notation of \cite[\S 14]{BPmodp}.
The image of a nonzero such map would then contain $\sigma(R\chi^s)$. 
Since $\theta(\chi^s)/\theta(\chi^s)^{R\chi^s}$ does not contain $\sigma(R\chi^s)$ as a Jordan--H\"older factor by construction, we are done.
Alternatively, one could show that $\JH(\theta(\chi^s)/\theta(\chi^s)^{R\chi^s}) \cap W(\rhobar) = \emptyset$, which immediately implies that $M_\infty(\theta(\chi^s)/\theta(\chi^s)^{R\chi^s})$ is zero. 
\end{proof}

We let $U_p(\chi)$ be $\eta(U_p)$ where $U_p$ is the Hecke operator in $\mH(\theta(\chi)[1/p])$.
We denote by $h_\chi: M_\infty(\theta(\chi^s)^{R\chi^s}) \to M_\infty(\theta(\chi^s)^{R\chi})$ the composition of the three vertical arrows (the last one inverted). Its mod~$\fm$ reduction is closely related to~$g_\chi$. To see this, we deduce from Lemma~\ref{generators} that there exist surjections
\begin{align}
\label{definingprI}\pr_{\chi^s}: \theta(\chi^s)^{R\chi^s} &\to \sigma(R\chi^s), \\
\nonumber S_{i(\chi^s)} \varphi^{\chi^s} &\mapsto \varphi^{R\chi^s}\\
\nonumber\\
\label{definingprII}\pr_{\chi}: \theta(\chi^s)^{R\chi} &\to \sigma(R\chi), \\
\nonumber S_{i(\chi)}S_0 \varphi^{\chi^s} &\mapsto \varphi^{R\chi}
\end{align}
where in the second case we use the isomorphism 
\begin{align*}
\theta(\chi^s)^\chi &\to \theta(\chi), \\
S_0\varphi^{\chi^s} &\mapsto \varphi^\chi.
\end{align*} 
(Recall that by definition $\sigma(\xi)$ is a quotient of~$\theta(\xi)$ with generator~$\varphi^{\xi}$. 
While the existence of the surjections is automatic from the fact that, for example, $\sigma(R\chi) \cong \cosoc \theta(\chi^s)^{R\chi}$, we characterize them using, for example, the generator $S_{i(\chi)}S_0\varphi^{\chi^s}$ of $\theta(\chi^s)^{R\chi}$; see Lemma~\ref{generators}.) 
In the special case that $\chi = R\chi^s$, so that $S_{i(\chi^s)} = S_0$ and $S_{i(\chi)} = \id$, we have that $\pr_{\chi^s} = \pr_{\chi}: S_0\varphi^{\chi^s} \mapsto \varphi^\chi$.

\begin{pp}\label{pp:dualmap}
There is a commutative diagram
\begin{equation}\label{eqn:hbar}
\begin{tikzcd}
M_\infty(\theta(\chi^s)^{\chi^s})/\fm \arrow{r}{\Pi} & M_{\infty}(\theta(\chi^s)^{\chi})/\fm  \\
M_\infty(\theta(\chi^s)^{R\chi^s})/\fm\arrow{u} \arrow{r}{h_\chi}\arrow{d}{\pr_{\chi^s}} & M_{\infty}(\theta(\chi^s)^{R\chi})/\fm \arrow{u} \arrow{d}{\pr_\chi} \\
M_{\infty}(\sigma(R\chi^s))/\fm \arrow{r}{\lbar h_\chi} & M_{\infty}(\sigma(R\chi))/\fm,
\end{tikzcd}
\end{equation}
where the unlabelled arrows are induced from saturated inclusions, the maps labelled $\pr_\chi$ and $\pr_{\chi^s}$ are induced from these surjections, and all vertical arrows are isomorphisms.
Moreover, the map $\lbar h_\chi$ is dual to $g_\chi$.
\end{pp}
\begin{proof}
The top vertical arrows are isomorphisms by Lemma \ref{lemma:Upchi}.
By exactness of $M_\infty(-)$, the bottom vertical arrows are surjective.
Since $M_\infty(\theta(\chi^s)^{R\chi^s})/\fm$ and $M_{\infty}(\theta(\chi^s)^{R\chi})/\fm$ are one-dimensional by Theorem \ref{thm:free} and $M_{\infty}(\sigma(R\chi^s))/\fm$ and $M_{\infty}(\sigma(R\chi))/\fm$ are nonzero, the bottom vertical arrows are isomorphisms.
We let $\lbar h_\chi$ be the unique map which makes the diagram commute.

Let $f$ be an element in $\Hom_{k_E[\GL_2(k_L)]}(\sigma(R\chi), M_\infty^\vee[\fm])$.
Let $f_\infl\in\Hom_{k_E[\GL_2(k_L)]}(\theta(\chi^s)^{R\chi}, M_\infty^\vee[\fm])$ and $\tld{f}\in\Hom_{k_E[\GL_2(k_L)]}(\theta(\chi^s)^\chi, M_\infty^\vee[\fm])$ be the elements obtained by the vertical isomorphisms.
Then 
\begin{align*}
g_\chi(f(\varphi^{R\chi})) &= g_\chi(\tld{f}(S_{i(\chi)} S_0 \varphi^{\chi^s})) \\
&= g_\chi(S_{i(\chi)} \tld{f}( S_0 \varphi^{\chi^s})) \\
&= g_\chi(R \tld{f}( S_0 \varphi^{\chi^s})) \\
&= R\Pi (\tld{f}( S_0 \varphi^{\chi^s})) \\
&= S_{i(\chi^s)} h_\chi^\vee(f_\infl)(\varphi^{\chi^s}) \\
&= \lbar h_\chi^\vee(f)(\varphi^{\chi^s}).
\end{align*}
\end{proof}
We deduce the following special case of Theorem~\ref{patchingdiagram}, which is also a consequence of the results of~\cite{BD}. Accordingly, in what follows we will often make the assumption that $\chi \not = R\chi^s$.
\begin{lemma}\label{fixedpoint}
Assume~$\chi = R\chi^s$. Then~$g_\chi$ does not depend on~$M_\infty$.
\end{lemma}
\begin{proof}
The above discussion has shown that the dual~$h_\chi$ to~$g_\chi$ can be identified with the reduction of~$U_p(\chi)$.
\end{proof}

\subsection{Lifting~$U_p(\chi)$.}

We translate the above calculation for the principal series type $\theta(\chi^s) = \sigma(\tau(\chi^s))$ (this defines the tame inertial type $\tau(\chi^s)$) into one for the \emph{central type} $\theta = \sigma(\tau)$ (recall that $\tau$ is the central inertial type). 
Again, $\chi: I \to \mO_E^\times$ is a character appearing in $D_0(\rhobar)^{I_1}$ with $\chi \not = R\chi^s$.
Since a composition of saturated inclusions need not be saturated, from the top line of~(\ref{diag:Up}) we see that there is a unique $e(\chi)\in \bN$ such that the diagram 
\begin{equation}\label{diag:Upbig}
\begin{tikzcd}
M_\infty(\theta(\chi^s)^{R\chi}) \arrow{r} \arrow{rd}[swap]{p^{-e(\chi)}U_p(\chi)} & M_\infty(\theta(\chi^s)^{R\chi^s}) \arrow{d}{h_\chi} \arrow{r}{\pr_{\chi^s}} & M_\infty(\sigma(R\chi^s))/\fm \arrow{d}{\overline{h}_\chi}\\
& M_\infty(\theta(\chi^s)^{R\chi}) \arrow{r}{\pr_\chi} & M_\infty(\sigma(R\chi))/\fm
\end{tikzcd}
\end{equation}
commutes, where the unlabelled horizontal map is induced from a saturated inclusion.

\begin{pp}\label{commonquotient}
There exists a unique quotient~$Q(\chi^s)^{R\chi^s}$ of~$\theta(\chi^s)^{R\chi^s}/\varpi_E$ which is isomorphic to a quotient of~$\theta^{R\chi^s}/\varpi_E$ and whose Jordan--H\"older factors are exactly $\JH(\thetabar(\chi^s)) \cap \JH(\thetabar)$. 
Similarly, there exists a unique quotient~$Q(\chi^s)^{R \chi}$ of $\theta(\chi^s)^{R\chi}/\varpi_E$ which is isomorphic to a quotient of~$\theta^{R\chi}/\varpi_E$ and whose Jordan--H\"older factors are exactly $\JH(\thetabar(\chi^s)) \cap \JH(\thetabar)$.
\end{pp}
\begin{proof}
The duals $(\theta(\chi^s)^{R\chi^s}/\varpi_E)^\vee$ and $(\theta^{R\chi^s}/\varpi_E)^\vee$ are representations with the same socle, which is isomorphic to~$\sigma(R\chi^s)^\vee$, is generic and appears with multiplicity one as a Jordan--H\"older factor. By~\cite[Proposition~3.6]{BPmodp}, these representations embed in a twist of what they denote~$V_{\textbf{2p-2-r}}$, where~$\br$ is such that $\sigma(R\chi^s)^\vee \cong (r_0, r_1, \ldots, r_{f-1}) \otimes \eta$ for some character~$\eta$. Since~$\sigma(R\chi^s)$ is a modular weight for the generic representation~$\rhobar$, all the~$r_i$ satisfy the constraint $0 \leq r_i \leq p-2$ (and indeed, even stronger ones). Then the proposition follows from~\cite[Theorem~4.7]{BPmodp}.
\end{proof}

Consider the diagram
\[
\begin{tikzcd}
Q(\chi^s)^{R\chi} \arrow{d}{\iota_Q} & \arrow{l}{\pi} \arrow{d}{\iota} \theta(\chi^s)^{R\chi}\\
Q(\chi^s)^{R\chi^s} & \arrow{l}{\pi} \theta(\chi^s)^{R\chi^s},
\end{tikzcd}
\]
where $\pi$ denotes the natural projection maps and $\iota$ denotes a saturated inclusion.
Then $\pi\circ \iota$ factors through the top map.
Let $\iota_Q$ be the unique map that makes the diagram commute.

Note that the quotient maps $\theta(\chi^s)^{R\chi}/\varpi_E \to Q(\chi^s)^{R\chi}$ and $ \theta(\chi^s)^{R\chi^s}/\varpi_E \to Q(\chi^s)^{R\chi^s}$ induce isomorphisms on~$M_{\infty}$ (indeed, their kernels have no Jordan--H\"older factors in~$W(\rhobar)$ since $W(\rhobar) \subseteq \JH(\thetabar)$).
Moreover, the surjections $\pr_{\chi^s}$ and $\pr_\chi$ factor through these quotients.
Passing to quotients, (\ref{diag:Upbig}) becomes
\begin{equation}\label{diag:Q}
\begin{tikzcd}
M_\infty(Q(\chi^s)^{R\chi}) \arrow{r}{\iota_Q} \arrow{rd}[swap]{p^{-e(\chi)}U_p(\chi)} & M_\infty(Q(\chi^s)^{R\chi^s}) \arrow{d}{h_\chi} \arrow{r}{\pr_{\chi^s}} & M_\infty(\sigma(R\chi^s))/\mathfrak{m} \arrow{d}{\overline{h}_\chi} \\
& M_\infty(Q(\chi^s)^{R\chi}) \arrow{r}{\pr_{\chi}} & M_\infty(\sigma(R\chi))/\mathfrak{m}.
\end{tikzcd}
\end{equation}

Fix a surjection~$\alpha: \theta^{R\chi^s} \to Q(\chi^s)^{R\chi^s}$.
Consider the diagram 
\[
\begin{tikzcd}
\theta^{R\chi} \arrow{r}{\alpha} \arrow{d}{\iota} & Q(\chi^s)^{R\chi} \arrow{d}{\iota_Q} & \arrow{l}{\pi} \arrow{d}{\iota} \theta(\chi^s)^{R\chi}\\
\theta^{R\chi^s} \arrow{r}{\alpha} & Q(\chi^s)^{R\chi^s} & \arrow{l}{\pi} \theta(\chi^s)^{R\chi^s}
\end{tikzcd}
\]
The map $\alpha \circ \iota$ factors through $\iota_Q$.
We call the unique top map that makes the diagram commute $\alpha$ as well.

We have the following diagram
\begin{equation}
\begin{tikzcd}\label{diag:Upimagecomplete}
M_\infty(\theta^{R\chi}) \arrow{r}{\alpha}\arrow{d}{\iota} & M_{\infty}(Q(\chi^s)^{R\chi}) \arrow{d}{\iota_Q} & \arrow{l}{\pi} \arrow{d}{\iota} M_{\infty}(\theta(\chi^s)^{R\chi}) \\
M_\infty(\theta^{R\chi^s})\arrow{r}{\alpha} \arrow{d}{\tld{h}_\chi} & M_{\infty}(Q(\chi^s)^{R\chi^s}) \arrow{d}{h_{\chi}} & \arrow{l}{\pi} M_{\infty}(\theta(\chi^s)^{R\chi^s}) \arrow{d}{h_\chi}\\
M_{\infty}(\theta^{R\chi}) \arrow{r}{\alpha} & M_{\infty}(Q(\chi^s)^{R\chi}) & \arrow{l}{\pi} M_{\infty}(\theta(\chi^s)^{R\chi}).
\end{tikzcd}
\end{equation}
where~$\tld{h}_\chi$ is any lift that makes the diagram commute.
Note that $\tld{h}_\chi$ exists because $M_\infty(\theta^{R\chi^s})$ is free of rank one over $R_\infty(\tau)$.
In fact, $\tld{h}_\chi$ is necessarily an isomorphism. 
It is surjective by Nakayama's lemma because $h_\chi$ is surjective.
This implies that it is injective since source and target are free of rank one over~$R_\infty(\tau)$ by Theorem \ref{thm:free}.
This free of rank one fact also implies that the composition of the left column is given by multiplication by some element~$\tld{U}_p(\chi) \in R_\infty(\tau)$. 
Then the commutative diagram 
\begin{equation}\label{diag:tldUp}
\begin{tikzcd}
M_\infty(\theta^{R\chi}) \arrow{r}{\iota} \arrow{rd}[swap]{\tld{U}_p(\chi)} & M_\infty(\theta^{R\chi^s}) \arrow{d}{\tld{h}_\chi} \arrow{r}{\pr_{\chi^s}\circ\alpha} & M_\infty(\sigma(R\chi^s))/\mathfrak{m} \arrow{d}{\overline{h}_\chi} \\
& M_\infty(\theta^{R\chi}) \arrow{r}{\pr_{\chi}\circ\alpha} & M_\infty(\sigma(R\chi))/\mathfrak{m}.
\end{tikzcd}
\end{equation}
is a lift of (\ref{diag:Q}).

Now recall that since
\[
\JH(\thetabar(\chi^s)) \cap W(\rhobar) \subseteq \JH(\thetabar) \cap W(\rhobar)
\]
by construction, by~\cite[Theorem~7.2.1]{EGS} there exists an ideal~$I(\chi)$ of $R_\infty(\tau)/\varpi_E$ such that the special fibre $R_\infty(\tau(\chi^s)) / \varpi_E$ is the quotient $R_\infty(\tau)/(\varpi_E, I(\chi))$.

\begin{pp}\label{pp:lift}
If we let $1+\fm_{R_\infty(\tau)}$ denote the $1$-units of $R_\infty(\tau)$, then the equivalence class 
\[
\tld{U}_p(\chi) (1+\fm_{R_\infty(\tau)}) \in R_\infty(\tau)/(1+\fm_{R_\infty(\tau)})
\]
(the orbit under multiplication) is uniquely characterized by the following two properties:
\begin{enumerate}
\item we have the equality
\begin{equation}\label{eqn:inclusion}
M_\infty(\iota)(M_\infty(\theta^{R\chi})) = \tld{U}_p(\chi) M_\infty(\theta^{R\chi^s}), and
\end{equation}
\item the image of $\tld{U}_p(\chi)$ in~$R_\infty(\tau(\chi^s))/\varpi_E$ is equal to that of $p^{-e(\chi)}U_p(\chi)$.
\end{enumerate}
Moreover, $\tld{U}_p(\chi) (1+\fm_{R_\infty(\tau)})$ is independent of the choice of~$M_\infty$ (once the integer $h$ is fixed).
\end{pp}
\begin{proof}
Note that $\tld{U}_p(\chi)$ satisfies both of these conditions.
Indeed, 
\[
\tld{U}_p(\chi) \circ \iota = \iota \circ \tld{U}_p(\chi) = (\iota \circ \tld{h}_\chi) \circ \iota 
\]
while 
\begin{align*}
h_\chi\circ \iota_Q &\equiv \tld{U}_p(\chi) \pmod{\varpi_E+I(\tau(\chi))} \\
&\equiv p^{-e(\chi)}U_p(\chi) \pmod{\varpi_E}
\end{align*}
by (\ref{diag:Upimagecomplete}).

The condition (\ref{eqn:inclusion}) uniquely characterizes $\tld{U}_p(\chi)$ up to multiplication by $R_\infty(\tau)^\times$.
The comparison to $U_p(\chi)$ characterizes it up to $1+\fm_{R_\infty(\tau)}$.
To prove the last part, it suffices to show that these two properties are independent of $M_\infty$.
This is clear for the second property.
We now consider the first property.
We begin by noting that $R_\infty(\tau)$ is normal ($R_1$ and $S_2$) since it is $\mO_E$-flat and the special fiber is reduced and so $R_0$ and $S_1$.
Then the principal ideal giving $M_\infty(\iota)(M_\infty(\theta^{R\chi}))$ in $M_\infty(\theta^{R\chi^s})$ is determined by its corresponding Weil divisor, which is 
\begin{equation}\label{eqn:weil}
\sum_{\sigma \in \JH(\thetabar)} m(\sigma) \overline{X}^\psi(\sigma),
\end{equation}
where $m(\sigma)$ denotes the multiplicity of $\sigma$ as a Jordan--H\"older factor of $\coker\, \iota$ and $\overline{X}^\psi(\sigma)$ is as in \cite[\S 6]{EGS}.
We conclude by noting that (\ref{eqn:weil}) does not depend on the choice of $M_\infty$.
\end{proof}

\subsection{Proof of Theorem~\ref{patchingdiagram}.}\label{iterate}
Now fix a cycle $\chi_0, \ldots, \chi_{k-1}$ so that $R\chi_i^s = R\chi_{i+1}$ and $\chi_i \not = R\chi_i^s$ for all~$i$ (otherwise we apply Lemma~\ref{fixedpoint}). 
We will prove that $\overline{h}_{\chi_0} \cdots \overline{h}_{\chi_k-1}$ is independent of~$M_{\infty}$, which will imply Theorem~\ref{patchingdiagram} by Propositions~\ref{pp:groupoidrep} and~\ref{pp:dualmap}. The last preliminary step is to iterate the constructions above for the~$\chi_i$.

Fix a surjection $\alpha_0: \theta^{R \chi^{s}_0} \to Q(\chi^s_0)^{R\chi^s_0}$. We have seen that there exists a map~$\tld{h}_{\chi_0}$, well-defined up to 1-units, such that the following diagram commutes.
\[
\begin{tikzcd}
M_\infty ( \theta^{R\chi_0^s} ) \arrow{r}{\alpha_0} \arrow{d}{\tld{h}_{\chi_0}} & M_\infty ( Q(\chi_0^s)^{R \chi^s_0} ) \arrow{d}{h_{\chi_0}}&\arrow{l}{\pi} M_\infty ( \theta(\chi_0^s)^{R\chi^s_0} ) \arrow{d}{h_{\chi_0}}\\
M_\infty ( \theta^{R\chi_0} ) \arrow{r}{\alpha_0} & M_\infty ( Q(\chi_0^s)^{R\chi_0} ) & \arrow{l}{\pi} M_\infty ( \theta(\chi_0^s)^{R\chi_0} )
\end{tikzcd}
\]

Now define a surjection~$\alpha_1: \theta^{R\chi_1^s} \to Q(\chi_1^s)^{R\chi_1^s}$ by requiring that~$\alpha_1$ makes the following diagram commute.
\begin{equation}\label{eqn:alpha}
\begin{tikzcd}
{} & Q(\chi_1^s)^{R\chi_1}\arrow{dr}{\pr_{\chi_1}} & {}\\
\theta^{R\chi_1} = \theta^{R\chi_0^s} \arrow{ur}{\alpha_1} \arrow{dr}{\alpha_0} & {} & \sigma(R\chi_1) = \sigma(R\chi_0^s)\\
{} & Q(\chi_0^s)^{R\chi_0^s} \arrow{ur}{\pr_{\chi_0^s}} & {}
\end{tikzcd}
\end{equation}
Then define inductively maps $\alpha_i: \theta^{R\chi_i^s} \to Q(\chi_i^s)^{R\chi_i^s}$. These are maps between representations of~$\GL_2(k_L)$, defined independently of~$M_\infty$. 

\begin{rk} It is usually not the case that $\alpha_k = \alpha_0$ for a cycle of length~$k$, and determining the scalar by which they differ is actually the main representation-theoretic input to the computation of Breuil's constants. It is for this reason that we have taken some care in the definition of~$\alpha_i$.
\end{rk}

We are now in a position to complete the proof of Theorem~\ref{patchingdiagram}. 
Consider the amalgamation of diagrams~(\ref{diag:tldUp}) using (\ref{eqn:alpha}) along the cycle~$\chi_0, \ldots, \chi_{k-1}$, which we draw next for~$k = 3$.
\[
\begin{tikzcd}
M_\infty ( \theta^{R\chi^s_0} \arrow{r}{} ) \arrow{dr}[swap]{\tld{U}_p(\chi_0)}& M_\infty ( \theta^{R \chi^s_1} \arrow{r}{} \arrow{dr}[swap]{\tld{U}_p(\chi_0)}  )& M_\infty ( \theta^{R\chi^s_2} \arrow{r}{} \arrow{dr}[swap]{\tld{U}_p(\chi_0)}) & M_\infty ( \theta^{R\chi^s_0} \arrow{r}{\pr_{\chi^s_0}\circ\alpha_3} \arrow{d}{\tld{h}_{\chi_0}}) & M_\infty ( \sigma(R\chi^s_0) )/\fm \arrow{d}{\overline{h}_{\chi_0}}\\
{} & M_\infty ( \theta^{R\chi^s_0}\arrow{r}{}) \arrow{dr}[swap]{\tld{U}_p(\chi_2)}& M_\infty ( \theta^{R\chi^s_1} \arrow{r}{} \arrow{rd}[swap]{\tld{U}_p(\chi_2)} )& M_\infty ( \theta^{R\chi^s_2}\arrow{r}{\pr_{\chi^s_2}\circ\alpha_2}  \arrow{d}{\tld{h}_{\chi_2}}) & M_\infty ( \sigma(R\chi^s_2)  )/\fm\arrow{d}{\overline{h}_{\chi_2}}\\
{} & {} & M_\infty ( \theta^{R\chi^s_0} \arrow{r}{} ) \arrow{dr}[swap]{\tld{U}_p(\chi_1)}& M_\infty ( \theta^{R\chi^s_1}\arrow{r}{\pr_{\chi^s_1}\circ\alpha_1} \arrow{d}{\tld{h}_{\chi_1}} )& M_\infty ( \sigma(R\chi^s_1) )/\fm\arrow{d}{\overline{h}_{\chi_1}}\\
{} & {} & {} & M_\infty ( \theta^{R\chi^s_0}) \arrow{r}{\pr_{\chi^s_0}\circ\alpha_0} & M_\infty ( \sigma(R\chi^s_0) )/\fm.
\end{tikzcd}
\]
Here, the unlabelled horizontal arrows are induced by saturated inclusions. The diagonal arrows are $\tld{U}_p$-operators. The rectangles commute by definition of the~$\alpha_i$ together with~(\ref{diag:Q}) and (\ref{diag:Upimagecomplete}). By construction, there exists a positive integer $\nu = \nu(\chi_0, \ldots, \chi_{k-1})$ such that the composition
\[
\theta^{R\chi^s_0} \to \theta^{R\chi^s_1} \to \cdots \to \theta^{R\chi^s_{k-1}} \to \theta^{R\chi^s_0}
\]
of saturated inclusions is multiplication by $p^\nu$. 
From the diagram we see that
\begin{equation}\label{eqn:ppower}
p^\nu \tld{h}_{\chi_{1}}\cdots\tld{h}_{\chi_{k-1}}\tld{h}_{\chi_{0}} = \tld{U}_p(\chi_{0})\cdots\tld{U}_p(\chi_{k-1}). 
\end{equation}
If $c(\chi_0, \ldots, \chi_{k-1}) \in k_E^\times$ is such that $\alpha_k = c(\chi_0, \ldots, \chi_{k-1}) \alpha_0$, then~$c(\chi_0, \ldots, \chi_{k-1}) $ is independent of~$M_\infty$ (it is even independent of the choice of surjection~$\alpha_0$) and we have
\begin{equation}\label{eqn:twoconst}
c(\chi_0, \ldots, \chi_{k-1}) \overline{h}_{\chi_1}\cdots\overline{h}_{\chi_k}\overline{h}_{\chi_0} = p^{-\nu}\tld{U}_p(\chi_{0})\cdots\tld{U}_p(\chi_{k-1}) 
\end{equation}
in $R_{\infty}(\tau)^\times/1+\fm_{R_\infty(\tau)}$, and the theorem follows since the~$\tld{U}_p(\chi_i)$ are well-defined and independent of~$M_\infty$ up to multiplication by~$1+\fm_{R_\infty(\tau)}$ by proposition \ref{pp:lift}.

\section{Breuil's constants.} \label{sec:constants}

We now proceed to the computation of the constant $g_{\chi_{k-1}}\cdots g_{\chi_0} \in k_E^\times$ in the cases considered by Breuil in \cite{Breuilparameters}. 
By Proposition \ref{pp:dualmap} and (\ref{eqn:twoconst}), this amounts to a computation of two units 
\[
c(\chi_0, \ldots, \chi_{k-1}) \in k_E^\times \text{ and } p^{-\nu(\chi_0, \ldots, \chi_{k-1})} \tld{U}_p(\chi_0)\cdots \tld{U}_p(\chi_{k-1}) \in R_\infty(\tau)^\times/(1+\fm_{R_\infty(\tau)}) \cong k_E^\times.
\] 
Although Breuil's cycles are only defined for semisimple mod~$p$ representations, we make some of our calculations on Galois deformation rings in more generality.
Until Lemma~\ref{lemma:U_p'}, the only assumptions on~$\rhobar$ are our running assumptions of genericity from \S \ref{Galoisreps}.

\subsection{The product $\tld{U}_p(\chi_0)\cdots \tld{U}_p(\chi_{k-1})$.}\label{sec:prodUp}

In order to compute this product, we must relate the above patching constructions with deformation ring calculations in \S \ref{sec:defring}.
This involves a translation of the results of \cite[\S 7-8]{EGS} into the notation of \S \ref{sec:defring}.
Rather than make this translation explicit, we indicate how the methods of \emph{loc.~cit.}~can be used to prove the results that we need. By construction of the~$\tld{U}_p$, our final result will only be well-defined up to 1-units in the deformation ring, but this will suffice for our purposes.

\paragraph{Notation.} Recall from \S \ref{sec:central} that~$\tau$ is the central inertial type for~$\rhobar$. If~$\tau_*$ is a tame inertial type, then the map $\kappa^{\tau_*}$ from \S \ref{sec:defk} induces an injective, formally smooth map 
\begin{equation}\label{eqn:kappa}
R^{\tau_*}_\rhobar \ra R^{\tau_*,\square}_{\overline{\fM}^{\tau_*}}.
\end{equation}
Recall also that we have attached a weight~$\sigma_J$ to each subset~$J \subset \bZ/f$, and that all modular weights of~$\rhobar$ have this form. 
If~$\chi \in (\soc_K D_0(\rhobar))^{I_1}$, we write~$J(\chi)$ for the subset such that $\sigma_{J(\chi)}^{I_1} \cong \chi$.

\begin{lemma}\label{lemma:wtsupp}
Suppose that $\sigma \in W(\rhobar)$ and recall from~\S\ref{sec:iserrewts} the ideal~$I(\sigma)$.
Then the scheme-theoretic support of 
\[
M_\infty(\sigma)\otimes_{R^\tau_\rhobar} R^{\tau,\square}_{\overline{\fM}^\tau}
\]
is $\Spec R_\infty\otimes_{R^\tau_\rhobar} (R^{\tau,\square}_{\overline{\fM}^\tau}/I(\sigma))$.
\end{lemma}
\begin{proof}
There exists a type $\tau_*$ such that $\JH(\lbar\sigma(\tau_*)) \cap W(\rhobar) = \{\sigma\}$. 
By Proposition \ref{pp:intersect}, if $\sigma \cong F(\ft_\mu(s_\rhobar\omega_J))$ one could take $w_{*,j} = \id$ for all $j$ and $w'_{*,j}$ to be $\id$ if and only if $j\notin J$ in the notation of \S \ref{sec:comp}.
Then $I(\tau_*)+(\varpi_E)$ coincides with $I(\sigma)$. 
Let $\theta_*^\sigma$ be the lattice in~$\sigma(\tau_*)$ with cosocle $\sigma$.
Then the surjection $\theta_*^\sigma \ra \sigma$ induces an isomorphism
\[
M_\infty(\thetabar_*^\sigma) \ra M_\infty(\sigma).
\]
The scheme-theoretic support of $M_\infty(\thetabar_*^\sigma)$ is $\Spec R^{\tau_*}_\rhobar/\varpi_E$, and~$R_\rhobar^{\tau_*}/\varpi_E$ is a quotient of~$R_\rhobar^{\tau}/\varpi_E$ since~$\tau$ is central for~$\rhobar$. 
Since~(\ref{eqn:kappa}) is flat, it suffices to show that 
\[
R_\infty(\tau^*)/\varpi_E \otimes_{R^\tau_\rhobar} R^{\tau,\square}_{\overline{\fM}^\tau} = R_\infty\otimes_{R^\tau_\rhobar} (R^{\tau,\square}_{\overline{\fM}^\tau}/I(\tau_*)+(\varpi_E)).
\]
This follows from Proposition~\ref{pp:compare}.
\end{proof}

\begin{lemma}\label{lemma:gauge}
Let $J_1 \subset J_\rhobar$ be such that $j\in J_\rhobar \setminus J_1$.
Let $J_2$ be $J_1 \cup \{j\}$.
Suppose that $\tau_*$ is a tame inertial type such that $\JH(\lbar\sigma(\tau_*))$ contains both $\sigma_1 = \sigma_{J_1}$ and $\sigma_2 = \sigma_{J_2}$.
Let $\theta_*^2 \subset \theta_*^1$ be a saturated inclusion of $\mO_E$-lattices in $\sigma(\tau_*)$ with cosocles isomorphic to $\sigma_2$ and $\sigma_1$, respectively.
Then 
\begin{equation}\label{eqn:gauge}
M_\infty(\theta_*^2)  \otimes_{R^{\tau_*}_\rhobar} R^{\tau_*,\square}_{\overline{\fM}^{\tau_*}} = Y_{f-1-j}M_\infty(\theta_*^1)  \otimes_{R^{\tau_*}_\rhobar} R^{\tau_*,\square}_{\overline{\fM}^{\tau_*}}.
\end{equation}
\end{lemma}
\begin{proof}
One checks from Lemma \ref{lemma:wtsupp} that $Y_{f-j-2}$ annihilates the quotient
\[
(M_\infty(\theta_*^1)/M_\infty(\theta_*^2))  \otimes_{R^{\tau_*}_\rhobar} R^{\tau_*,\square}_{\overline{\fM}^{\tau_*}}
\]
(we can assume without loss of generality that $E/\bQ_p$ is an unramified extension by Lemma \ref{lem:restscal}).
The inclusion $\supset$ in (\ref{eqn:gauge}) then follows.
By \cite[Theorem 5.2.4]{EGS}, we see that $p\tau_*^1\subset \tau_*^2$ is also a saturated inclusion.
Arguing as above we see that 
\begin{equation}\label{eqn:gauge'}
pM_\infty(\theta_*^1)  \otimes_{R^{\tau_*}_\rhobar} R^{\tau_*,\square}_{\overline{\fM}^{\tau_*}} \supset X_{f-1-j}M_\infty(\theta_*^2)  \otimes_{R^{\tau_*}_\rhobar} R^{\tau_*,\square}_{\overline{\fM}^{\tau_*}}.
\end{equation}
Combining (\ref{eqn:gauge'}) with the earlier inclusion and the equation $X_{f-1-j}Y_{f-1-j} = p$ yields 
\[
pM_\infty(\theta_*^1)  \otimes_{R^{\tau_*}_\rhobar} R^{\tau_*,\square}_{\overline{\fM}^{\tau_*}} \supset X_{f-1-j}M_\infty(\theta_*^2)  \otimes_{R^{\tau_*}_\rhobar} R^{\tau_*,\square}_{\overline{\fM}^{\tau_*}} \supset pM_\infty(\theta_*^1)  \otimes_{R^\tau_\rhobar} R^{\tau_*,\square}_{\overline{\fM}^{\tau_*}}.
\]
This forces that all inclusions above are in fact equalities.
\end{proof}

Lemma \ref{lemma:gauge} and an induction argument yields the following.

\begin{pp}\label{pp:gauge}
Suppose that $\tau_*$ is a tame type such that $\JH(\taubar_*)$ contains both $\sigma_1 = \sigma_{J_1}$ and $\sigma_2 = \sigma_{J_2}$.
Let $\theta_*^2 \subset \theta_*^1$ be a saturated inclusion of $\mO_E$-lattices in $\sigma(\tau_*)$ with cosocles isomorphic to $\sigma_2$ and $\sigma_1$, respectively.
Let 
\begin{equation}\label{eqn:gaugefactor}
f_j(J_1,J_2) = 
\begin{dcases}
X_j &\mbox{if } f-1-j\in J_1 \setminus J_2\\
Y_j &\mbox{if } f-1-j\in J_2 \setminus J_1\\
1 &\mbox{otherwise.}
\end{dcases}
\end{equation}
Then 
\[
M_\infty(\theta_*^2)  \otimes_{R^{\tau_*}_\rhobar} R^{\tau_*,\square}_{\overline{\fM}^{\tau_*}} = \big(\prod_j f_j(J_1,J_2)\big) M_\infty(\theta_*^1)  \otimes_{R^{\tau_*}_\rhobar} R^{\tau_*,\square}_{\overline{\fM}^{\tau_*}}.
\]
\end{pp}

We are now ready to compare our calculations on deformation rings of Galois representations and Kisin modules. Fix a character~$\chi$ appearing in $D_0(\rhobar)^{I_1}$. 
Let $\tld{U}'_p(\chi)$ be the element of $R^{\tau,\square}_{\overline{\fM}^\tau}$ defined in \S \ref{sec:central}, with the type $\theta(\chi)[p^{-1}]$ playing the role of~$\sigma$.
We begin by relating $\tld{U}'_p(\chi)$ to $\tld{U}_p(\chi)$.

\begin{lemma}\label{lemma:ideal}
We have 
\[
\kappa^\tau(\tld{U}_p(\chi)) R^{\tau,\square}_{\overline{\fM}^\tau} = \tld{U}'_p(\chi) R^{\tau,\square}_{\overline{\fM}^\tau}. 
\]
\end{lemma}
\begin{proof}

Let $\theta^2 \subset \theta^1$ be a saturated inclusion of $\mO_E$-lattices in $\sigma(\tau)$ with cosocles isomorphic to $\sigma(R\chi)$ and $\sigma(R\chi^s)$, respectively.
By the first property of $\tld{U}_p(\chi)$ in Proposition \ref{pp:lift}, it suffices to show that
\begin{equation} \label{eqn:ideal}
M_\infty(\theta^2) \otimes_{R^\tau_\rhobar} R^{\tau,\square}_{\overline{\fM}^\tau} = \tld{U}'_p(\chi) (M_\infty(\theta^1) \otimes_{R^\tau_\rhobar} R^{\tau,\square}_{\overline{\fM}^\tau}).
\end{equation}

Let $\theta(\chi)^2 \subset \theta(\chi)^1$ be a saturated inclusion of lattices in $\sigma(\tau(\chi))$ with cosocles isomorphic to $\sigma(R\chi)$ and $\sigma(R\chi^s)$, respectively.
By definition of~$U_p(\chi)$, we have
\[
M_\infty(\theta(\chi)^2) \otimes_{R^{\tau(\chi)}_\rhobar} R^{{\tau(\chi)},\square}_{\overline{\fM}^{\tau(\chi)}} = p^{-e(\chi)}\kappa^{\tau(\chi)}(U_p(\chi)) (M_\infty(\theta(\chi)^1) \otimes_{R^{\tau(\chi)}_\rhobar} R^{{\tau(\chi)},\square}_{\overline{\fM}^{\tau(\chi)}}).
\]
By Proposition~\ref{pp:gauge}, this implies that
\begin{equation}\label{eqn:gaugeup}
p^{-e(\chi)}\kappa^{\tau(\chi)}(U_p(\chi)) \equiv \prod_j  f_j(J_1,J_2) \mod \left (  R^{\tau(\chi),\square}_{\overline{\fM}^{\tau(\chi)}} \right )^\times
\end{equation}
with $J_1$ and $J_2$ corresponding to $\sigma(R\chi^s)$ and $\sigma(R\chi)$, respectively.
By definition, $\tld U_p'(\chi)$ is congruent modulo $\left (  R^{\tau,\square}_{\overline{\fM}^\tau} \right )^\times$ to a product~$\prod_j g_j$ for certain factors~$g_j \in \{X_j, Y_j, 1 \}$ (note that~$X_\alpha + [\alpha], X_{\alpha'} + [\alpha']$ are units). 
Furthermore, Proposition~\ref{pp:reduceup} and (\ref{eqn:gaugeup}) imply that the image of $\prod_j g_j$ in $R_{\overline \fM^{\tau(\chi)}}^{\tau(\chi), \square}/p$ is
\[
\prod_j f_j(J_1, J_2) \in R_{\overline \fM^{\tau(\chi)}}^{\tau(\chi), \square}/p.
\]

Then necessarily we must have
\[
g_j \equiv f_j(J_1, J_2) \mod \left (  R^{\tau,\square}_{\overline{\fM}^\tau} \right )^\times
\]
for all~$j$. But then (\ref{eqn:ideal}) follows from Proposition~\ref{pp:gauge} since the formula (\ref{eqn:gaugefactor}) is independent of the tame type.
\end{proof}

\begin{lemma}\label{lemma:U_p'}
Let $\chi$ be a character appearing in~$D_0(\rhobar)^{I_1}$. Then $\kappa^\tau(\tld{U}_p(\chi))$ and $\tld{U}'_p(\chi)$ coincide up to multiplication by a $1$-unit of $R^{\tau,\square}_{\overline{\fM}^\tau}$.
\end{lemma}
\begin{proof}
Given lemma \ref{lemma:ideal}, this is a direct consequence of propositions \ref{pp:reduceup} and \ref{pp:lift}. 
\end{proof}
\paragraph{We now make some extra assumptions:} 
\begin{enumerate}
\item we assume that~$\rhobar$ is semisimple and we will keep this assumption until the end of \S \ref{sec:constants}. 
\item we fix a cycle~$\chi_0, \ldots, \chi_{k-1}$ of characters appearing in $(\soc_K D_0(\rhobar))^{I_1}$, such that $R\chi_i^s = \chi_{i+1}$ for all $i \in \bZ/k$: this is precisely the situation of~\cite[Question~9.5]{Breuilmodp} (though our genericity assumptions are slightly less general). 
\item \label{item:breuil}
by~\cite[Lemme~6.1]{Breuilparameters}, we can assume that $R\chi_i^s \not = \chi_i^s$ for all~$i$ (i.e., $\sigma(\chi_i^s) \notin W(\rhobar)$), and by Lemma~\ref{fixedpoint} we can assume that $R\chi_i^s \not = \chi_i$ for all~$i$. 
\end{enumerate}

\begin{lemma}
Assumption \eqref{item:breuil} implies that each~$H$-character $\chi_i$ has multiplicity one in the central type~$\theta_E = \sigma(\tau)$.
\end{lemma}
\begin{proof}
If~$\theta_E$ is cuspidal, this is true because every $H$-character has multiplicity one in~$\theta_E$. If~$\theta_E$ is principal series, then the only exceptions correspond to the two characters coming from $I_1$-invariants.
However, if~$\chi_i$ is one of these characters then~$\sigma(\chi_i^s)$ is in~$\JH(\lbar\sigma(\tau)) = W(\rhobar)$ so that~$R\chi_i^s = \chi_i^s$, which contradicts assumption~\eqref{item:breuil}.
\end{proof}

\begin{pp}\label{productvalue}
Let $J\subset \bZ/f$ be $J(\chi_i)$ for some $i\in \bZ/k$, so that $\chi_i \cong \sigma_J^{I_1}$.
Let~$J_0 = \delta_\red J$, so that~$\chi_i$ corresponds to the formal weight~$\lambda_{J_0}$.
Let~$\delta$ be $\delta_\irr$ or $\delta_\red$ according to whether~$\rhobar$ is irreducible or reducible.
Then the product 
\begin{equation}\label{eqn:produ}
\prod_{i=0}^{k-1} \tld{U}_p(\chi_i)
\end{equation}
is the product of a $1$-unit of $R^\tau_\rhobar$, the scalar
\[
(-p)^{\frac{1}{2}k|J_0\Delta \delta(J_0)|},
\]
and the scalar
\[
\begin{dcases}
[\alpha]^{-\frac{k|J_0^c|}{f}}[\alpha']^{-\frac{k|J_0|}{f}} &\mbox{if } \rhobar \textrm{ is reducible}\\
[-\alpha]^{-\frac{k}{2}} &\mbox{if } \rhobar \textrm{ is irreducible}
\end{dcases}
\]
where~$\alpha$ and~$\alpha'$ are as in \S\ref{Galoisreps}.
\end{pp}
\begin{rk}\label{rk:indJ}
For any $J \subset \{0,\ldots,f-1\}$, we have that $|J_0 \Delta \delta(J_0)| = |\delta(J_0) \Delta \delta^2(J_0)|$ by Lemma~\ref{deltabehaviour}.
Hence, the formulas in proposition \ref{productvalue} do not depend on the choice of $i \in \bZ/k$ in the definition of $J$.
\end{rk}
\begin{proof}

By Lemma \ref{lemma:U_p'}, it suffices to show that 
\begin{equation}\label{eqn:produ'}
\prod_{i=0}^{k-1} \tld{U}'_p(\chi_i)
\end{equation}
is of the same form up to multiplying by a $1$-unit of $R^{\tau,\square}_{\overline{\fM}^\tau}$.
Recall that for each $i$, $\tld{U}'_p(\chi_i)$ has a factorization 
\[
\prod_{j= 0}^{f-1} \tld{x}^{(j)}_i,
\]
where $\tld{x}^{(j)}_i$ is either $1$, $Y_{j}$, or $-X_{j}$ (multiplied by one of $(X_\alpha+[\alpha])^{-1}$ and $(X_{\alpha'}+[\alpha'])^{-1}$ if $j=0$).
Let $\tld{y}^{(j)}_i$ be $\tld{x}^{(j)}_i$ without the modification at $j=0$.

Fix $0 \leq i \leq k-1$ and let $J$ be $J(\chi_i)$.
We claim that 
\begin{equation}\label{eqn:Delta}
\big|\{\tld{y}^{(j)}_i\neq 1|0\leq j \leq f-1\}\big|=|J_0\Delta \delta(J_0)|.
\end{equation}

By Lemma~\ref{charactersw}, there are permutations $s^*, w, w' \in W$ such that 
\[
\theta(\chi)[1/p] \cong R_{s_\rhobar w}(\mu - s_\rhobar w' \eta), \, s^*F(s^*)^{-1} = s_\rhobar w \text{ and } \chi = (s^*)^{-1}(\mu-s_\rhobar w'\eta)|_H,
\]
and furthermore $j \in J_0$ if and only if~$s^*_j \ne \id$.
(Recall that the formal weight of~$\sigma_{J(\chi_i)}$ corresponds in all cases to~$\delta_\red(J(\chi_i))$, hence $\chi_j \in \{p-2-r_j, p-3-r_j\}$ if and only if~$j \in J_0$.)
From \S \ref{sec:defk}, we see that $\tld{y}^{(f-1-j)}_i = 1$ if and only if $w_j = \id$.
But since $w_j = s_{\rhobar,j}^{-1} s^*_j (s^*_{j-1})^{-1}$, one finds that $\tld{y}_i^{(f-1-j)} \neq 1$ if and only if $j \in J_0\triangle \delta^{-1}(J_0)$.
This establishes (\ref{eqn:Delta}).

By Remark \ref{rk:indJ}, we conclude that 
\begin{equation}\label{eqn:countXY}
\big|\{\tld{y}^{(j)}_i\neq 1|0\leq i\leq k-1, \, 0\leq j \leq f-1\}\big|=k|J_0\Delta \delta(J_0)|.
\end{equation}

Since (\ref{eqn:produ}) is a power of $p$ up to a unit by (\ref{eqn:ppower}), the same is true for (\ref{eqn:produ'}).
This implies that $Y_j$ and $-X_j$ appear among $\{\tld{y}^{(j)}_i\}_{i,j}$ with equal frequency.
We conclude from (\ref{eqn:countXY}) that 
\begin{equation}\label{eqn:y}
\prod_{i=0}^{k-1}\prod_{j=0}^{f-1} y^{(j)}_i = (-p)^{\frac{1}{2}k|J_0\Delta \delta(J_0)|}.
\end{equation}

For each $i$, $\tld{x}^{(0)}_i$ contains a factor of either $X_\alpha+[\alpha]$ or $X_{\alpha'} + [\alpha']$.
Moreover, we see from Corollary \ref{cor:frob'}, \S \ref{sec:orphi}, and (\ref{eqn:deform}) that these two possibilities correspond precisely to $f-2 \notin J(\chi_i)$ and $f-2 \in J(\chi_i)$, respectively.
Thus, we have
\[
\prod_{i=0}^{k-1} \frac{\tld{x}^{(0)}_i}{\tld{y}^{(0)}_i} = 
\begin{dcases}
(X_\alpha+[\alpha])^{-\frac{k|J_0^c|}{f}}(X_{\alpha'}+[\alpha'])^{-\frac{k|J_0|}{f}} &\mbox{if } \rhobar \textrm{ is reducible}\\
(X_\alpha+[\alpha])^{-\frac{k}{2}}(X_{\alpha'}+[\alpha'])^{-\frac{k}{2}} &\mbox{if } \rhobar \textrm{ is irreducible}
\end{dcases}
\]
(by the beginning of \S \ref{sec:irred}, $J_0^c$ is always in the $\delta$-orbit of $J_0$ when $\rhobar$ is irreducible).
In conjunction with (\ref{eqn:y}) and recalling that $\alpha' = -1$ if $\rhobar$ is irreducible, this finishes the proof.
\end{proof}

\begin{rk}
In the rest of this section, we will work with the parametrization by formal weights instead of the parametrization by the extension graph.
This is the reason for working with~$J_0$ in Proposition~\ref{productvalue}.
\end{rk}

\subsection{The constant~$c(\chi_0, \ldots, \chi_{k-1})$.}

Recall our running assumptions from \S \ref{iterate} that $R\chi_i = \chi_i$, that $R\chi_i^s = \chi_{i+1}$, and that $R\chi_i^s \not = \chi_i^s, \chi_i$ for each~$i\in \bZ/k$. 
We will actually compute the constant $c(\chi_{k-1}^s, \chi_{k-2}^s, \ldots, \chi_0^s)$, which is enough because of the following lemma.

\begin{lemma}\label{inverseconstants}
The equality $c(\chi_0, \ldots, \chi_{k-1}) = c(\chi_{k-1}^s, \ldots, \chi_{0}^s)^{-1}$ holds.
\end{lemma}
\begin{proof}
Informally, this is going around the cycle in the opposite direction. More precisely, we have defined $c(\chi_0, \ldots, \chi_{k-1})$ starting from a map $\alpha_0: \theta^{R \chi_0^s} \to Q(\chi_0^s)^{R\chi_0^s}$, constructing~$\alpha_i$ inductively, and putting $\alpha_k = c(\chi_0, \ldots, \chi_{k-1})\alpha_0$. Now we have to start from some $\theta^{R\chi_{k-1}} \to Q(\chi_{k-1})^{R\chi_{k-1}}$ and we know that $R \chi_{k-1} = R \chi_{k-2}^s$, so we choose to start from~$\alpha_{k-2}$. The next map to be constructed will be some $\theta^{R\chi_{k-2} }\to Q(\chi_{k-2})^{R\chi_{k-2}}$, and we know that $R\chi_{k-2} = R\chi_{k-3}^s$ and~$\alpha_{k-3}$ makes the pertinent diagram commute. Iterating, we deduce that $\alpha_0 = c(\chi_{k-1}^s, \ldots, \chi_{0}^s)\alpha_k$, and the lemma follows.
\end{proof}

We begin by reducing our goal to certain explicit computations on the central type~$\theta$, that can be performed via the methods of \S \ref{repprelim}. Write~$\psi_i = \chi_{k-1-i}^s$, and fix a surjection $\alpha_0: \theta^{\psi_0^s} \to Q(\psi_0^s)^{\psi_0^s}$ (observe that $R\psi_0^s = R\chi_{k-1} = \chi_{k-1} = \psi_0^s$), and construct the~$\alpha_i$ as in \S \ref{iterate}. This relabels the~$\alpha_i$ for our new cycle $\psi_0, \ldots, \psi_{k-1}$, and they will not change until the end of the paper.

By construction, there exists a unique $\nu(\psi_i^s) \in \bZ_{\geq 0}$ such that the composition 
\begin{equation}\label{eqn:satinc}
\theta(\psi_i^s)^{R\psi_i} \to \theta(\psi_i^s)^{\psi_i} \to \theta(\psi_i^s)
\end{equation} 
of saturated inclusions is $p^{\nu(\psi_i^s)}$ times the saturated inclusion. By Lemma~\ref{generators} together with the isomorphism
\begin{align*}
\theta(\psi_i) &\to \theta(\psi_i^s)^{\psi_i} \\
\varphi^{\psi_i} &\mapsto S_0 \varphi^{\psi_i^s}
\end{align*}
it follows that $p^{-\nu(\psi_i^s)}S_{i(\psi_i)}S_0 \varphi^{\psi_i^s}$ generates the image of the saturated inclusion $\theta(\psi_i^s)^{R\psi_i} \to \theta(\psi_i^s)$ (see Definition \ref{defn:R} for the definition of $S_{i(\psi_i)}$).

\begin{lemma}~\label{exchange}
There exists a unique $0 < i^+(\psi_i^s) < q-1$ and constant $c(\psi_i^s) \in W(k_L) \setminus \{0\} \subset W(k_E)$ such that 
\[
S_{i(\psi_i)} S_0 \varphi^{\psi_i^s} = c(\psi_i^s) S^+_{i^+(\psi_i^s)}\varphi^{\psi_i^s}.
\]
in the integral induction $\theta(\psi_i^s) = \Ind_I^K(\psi_i^s)$.  The equality~$i^+(\psi_i^s) = q-1-i(\psi_i)$ holds. 
\end{lemma}
\begin{proof}
This follows from Lemma~\ref{relationsbasistype} as the character~$R\psi_i$ has multiplicity one in~$\theta(\psi_i^s)$. Indeed, otherwise $R\chi_{k-1-i}^s = \chi_{k-1-i}$ or $R \chi_{k-1-i}^s = \chi_{k-1-i}^s$, which we have assumed not to hold. The equality follows from Lemma~\ref{shift}.
\end{proof}

\begin{defn}
We write~$\tld{c}(\psi_i^s)$ for the leading Teichm\"uller coefficient in the $p$-adic expansion of~$c(\psi_i^s)$.
\end{defn}

Next, we show that the reductions mod~$p$ of these constants~$\tld{c}(\psi_i^s)$ are closely related to the constant~$c(\chi_{k-1}^s, \ldots, \chi_0^s)$ we are trying to compute. 

\begin{lemma}\label{nonvanishing}
Let~$\iota : \theta(\psi_i^s)^{R\psi_i} \to \theta(\psi_i^s)$ be a saturated inclusion.
Then the vector~$S_{i^+(\psi_i^s)}^+\varphi^{\psi_i^s}$ is contained in and generates~$\theta(\psi_i^s)^{R\psi_i}$ over $\mO_E[\GL_2(k_L)]$.
\end{lemma}
\begin{proof}
Recall that~$R\psi_i$ has multiplicity one in~$\theta(\psi_i^s)[1/p]$. The result then follows from Lemma \ref{generators1}.

\end{proof}

\begin{pp}\label{nextalpha}
Recall the maps~$\pr$ from~(\ref{eqn:alpha}).
Assume~$x \in \theta^{\psi_{i}^s}$ is such that
\[
\pr_{\psi_i^s}\alpha_{i}(x) = \varphi^{\psi_{i}^s} \in \sigma(\psi_{i}^s).
\] 
Then $S^+_{i^+(\psi_i^s)}(x)$ is contained in the image of the saturated inclusion $\theta^{\psi_{i-1}^s} \to \theta^{\psi_i^s}$, and 
\[
\pr_{\psi_{i-1}^s}\alpha_{i-1} \left ( \tld{c}(\psi_i^s)S^+_{i^+(\psi_i^s)}(x) \right ) = \varphi^{\psi_{i-1}^s} \in \sigma(\psi_{i-1}^s).
\]
\end{pp}
\begin{proof}
Taking definitions into account, this follows from the discussion above and inspection of the following commutative diagram.
\[
\begin{tikzcd}
{} & \theta^{\psi_i^s} \arrow{d}{\alpha_{i}} & \theta^{\psi_{i-1}^s} \arrow{d}{\alpha_{i}} \arrow{l}{\iota} \arrow{r}{\alpha_{i-1}} & Q(\psi_{i-1}^s)^{\psi_{i-1}^s} \arrow{d}{\pr_{\psi_{i-1}^s}}\\
\sigma(\psi_{i}^s) & \arrow{l}{\pr_{\psi_{i}^s}} Q(\psi_{i}^s)^{\psi_{i}^s}  & \arrow{l}{\iota} Q(\psi_{i}^s)^{R\psi_{i}}\arrow{r}{\pr_{\psi_{i}}} & \sigma(\psi_{i-1}^s) = \sigma(R\psi_{i})\\
{} & \theta(\psi_i^s)^{\psi_i^s} \arrow{u}{\pi} & \arrow{l}{\iota} \theta(\psi_i^s)^{R\psi_i} \arrow{u}{\pi} & {}
\end{tikzcd}
\]
We give the details as a sequence of lemmas.

\begin{lemma}\label{stepI}
The equality $\pr_{\psi_i} \pi(\tld{c}(\psi_i^s)S_{i^+(\psi_i^s)}^+\varphi^{\psi_i^s}) = \varphi^{\psi_{i-1}^s}$ holds in~$\sigma(\psi_{i-1}^s)$.
\end{lemma}
\begin{proof}
By definition, 
\[
\pr_{\psi_i}\pi(p^{-\nu(\psi_i^s)}S_{i(\psi_i^s)}S_0\varphi^{\psi_i^s}) = \varphi^{\psi_{i-1}^s},
\]
and the claim follows from Lemmas~\ref{exchange} and~\ref{nonvanishing}.
(The map~$\pr_{\psi_i}$ in~(\ref{definingprII}) is defined in terms of a nonsaturated inclusion, so we need to divide by~$p^{-\nu(\psi_i^s)}$.)
\end{proof}

\begin{lemma}\label{stepII}
The equality $\alpha_i(x) = \pi(\varphi^{\psi_i^s})$ holds in~$Q(\psi_i^s)^{\psi_i^s}$.
\end{lemma}
\begin{proof}
By~\cite[Lemma~2.7(i)]{BPmodp}, the only quotient of~$\theta(\psi_i^s)/\varpi_E$ in which~$\psi_i^s$ has multiplicity two is~$\theta(\psi_i^s)/\varpi_E$ itself. 
Since~$\psi_i$ does not appear in~$\soc_K D_0(\rhobar)$, the type $\theta(\psi_i^s)$ is not the central type $\theta$, and so $\psi_i^s$ has multiplicity one in~$Q(\psi_i^s)^{\psi_i^s}$. 
It follows that $\alpha_i(x) = \pi(\varphi^{\psi_i^s})$, since they have the same image~$\varphi^{\psi_i^s}$ in~$\sigma(\psi_i^s)$.
\end{proof}

\begin{lemma}\label{stepIII}
The vector~$S_{i^+(\psi_i^s)}^+(x)$ is an element of~$\theta^{\psi_{i-1}^s}$.
\end{lemma}
\begin{proof}
By assumption, the character~$R\psi_i = \psi_{i-1}^s$ has multiplicity one in~$\theta[1/p]$.
The result then follows from Lemma \ref{generators1}.

\end{proof}

Now we conclude the proof of the proposition.
Lemmas~\ref{stepII} and~\ref{stepIII} imply that
\begin{equation}\label{equalityafteriota}
\iota\alpha_iS_{i^+(\psi_i^s)}^+(x) = S_{i^+(\psi_i^s)}^+ \alpha_i(x) = S_{i^+(\psi_i^s)}^+\pi (\varphi^{\psi_i^s}) = \pi S_{i^+(\psi_i^s)}^+(\varphi^{\psi_i^s})
\end{equation}
in $Q(\psi_{i}^s)^{\psi_i^s}$. 
By Lemma~\ref{nonvanishing} we deduce an equality
\[
\iota\alpha_iS_{i^+(\psi_i^s)}^+(x)  = \iota \pi S_{i^+(\psi_i^s)}^+(\varphi^{\psi_i^s}).
\]
Now~$\iota$ need not be injective, but we know that it is nonzero and that the vector $\pi S_{i^+(\psi_i^s)}^+(\varphi^{\psi_i^s})$ generates $Q(\psi_i^s)^{R\psi_i}$, so we deduce that
\begin{equation}\label{stepIV}
\alpha_iS_{i^+(\psi_i^s)}^+(x)  =  \pi S_{i^+(\psi_i^s)}^+(\varphi^{\psi_i^s}) \in Q(\psi_i^s)^{R\psi_i}.
\end{equation}
We conclude that
\begin{align*}
\pr_{\psi_{i-1}^s} \alpha_{i-1}(\tld{c}(\psi_i^s)S^+_{i^+(\psi_i^s)}(x)) = \,\,\, &\pr_{\psi_i} \alpha_i(\tld{c}(\psi_i^s) S_{i^+(\psi_i^s)}^+(x)) \\
\overset{(\ref{stepIV})}{=} &\pr_{\psi_i} \pi (\tld{c}(\psi_i^s) S_{i^+(\psi_i^s)}^+(\varphi^{\psi_i^s})) \\
 =  \,\,\, &\varphi^{\psi_{i-1}^s} \text{ by Lemma~\ref{stepI}}.
\end{align*}
\end{proof}

By Proposition~\ref{nextalpha}, we deduce that if $x \in \theta^{\psi_{k}^s}$, and $\alpha_k(x)$ maps to~$\varphi^{\psi_{k}^s} \in \sigma(\psi_{k}^s)$, then 
\[
\alpha_0 \left (\tld{c}(\psi_1^s)S^+_{i^+(\psi_1^s)}\cdots \tld{c}(\psi_{k}^s)S^+_{i^+(\psi_{k}^s)}(x) \right )
\]
maps to~$\varphi^{\psi_0^s} = \varphi^{\psi_{k}^s}$. The composition $\theta^{\psi_0^s} \to \theta^{\psi_1^s} \to \cdots \to \theta^{\psi_k^s}$ of saturated inclusions is multiplication by some~$p^\nu$, and we deduce that
\[
p^\nu[c(\psi_0, \ldots, \psi_{k-1})]\cdot x = \tld{c}(\psi_1^s)S^+_{i^+(\psi_1^s)}\cdots \tld{c}(\psi_{k}^s)S^+_{i^+(\psi_{k}^s)}(x).
\]
In the next sections, we will compute the product~$\prod_{i = 1}^{k} \tld{c}(\psi_i^s)$ as well as the composition~$S_{i^+(\psi_1^s)}^+\cdots S_{i^+(\psi_{k}^s)}^+$, using the methods of \S \ref{repprelim}.

\subsection{The irreducible case.}\label{sec:irred}

Here we compute the constants of the previous section when~$\rhobar$ is irreducible. Recall the bijections $J \mapsto \lambda_J$ and $J \mapsto \sigma(\lambda_J, \rhobar)$, from subsets~$J \subseteq\{0, \ldots, f-1 \}$ to formal weights of~$\rhobar$, resp. modular Serre weights of~$\rhobar$ (they are defined in~(\ref{formalweights}), resp.~(\ref{weightdefinition})). The integers~$r_i$ are defined through
\[
\rhobar|_{I_{\bQ_p}} \cong \fourmatrix{\omega_{2f}^{\sum_{j=0}^{f-1}(r_j+1)p^j}}{0}{0}{\omega_{2f}^{q\sum_{j=0}^{f-1}(r_j+1)p^j}}\] 
up to twist. In this section, we write~$\delta$ for~$\delta_{\mathrm{irr}}$. If~$\lambda = \lambda_J$, then we define~$\delta \lambda = \lambda_{\delta J}$, and similarly, if $\chi = \sigma(\lambda_J, \rhobar)^{I_1}$, then we define $\delta \chi = \sigma(\lambda_{\delta J}, \rhobar)^{I_1}$. By the proof of~\cite[Proposition~5.1]{Breuilparameters}, and our assumptions on the characters~$\chi_i$, we have $\chi_{i+1} = \delta \chi_i$ for all~$i$.

We remark that~$k$, the number of characters in our cycle, is an even number when~$\rhobar$ is irreducible: this is because $\delta^f(J) = J^c$ (the complement of~$J$). For future reference, we record the following lemma.

\begin{lemma}\label{differenceodd}
The number $|J \triangle \delta J|$ is odd.
\end{lemma}
\begin{proof}
It is enough to prove that $|J \triangle \delta_{\mathrm{red}}(J)|$ is even, because $\delta_{\mathrm{red}}(J)$ differs from~$\delta_\irr(J)$ precisely at~$f-1$. We have $j \in J \triangle \delta_\red J$ if and only if $\ch_J(j, j+1) \in \{(1, 0), (0, 1)\}$.

Define an involution $*: J \triangle \delta_\red J \to J \triangle \delta_\red J$ as follows: if~$\ch_J(j, j+1) = (1, 0)$, then $*j$ is the first number following~$j$ such that $\ch_J(*j, *j+1) = (0, 1)$. If $\ch_J(j, j+1) = (0, 1)$, then $*j$ is the first number preceding~$j$ such that $\ch_J(*j, *j+1) = (1, 0)$. The lemma follows since~$*$ has no fixed points.
\end{proof}

\paragraph{Computation of~$i^+(\psi_{k-1-i}^s)$.}
By definition, $i^+(\psi_{k-1-i}^s)$ is the only number in the interval~$[1, q-2]$ such that $S_{i^+(\psi_{k-1-i}^s)}^+$ sends the character~$\chi_i$ to~$\chi_{i+1}$. 
As such, we will also write it as~$i^+(\chi_i)$. 
The modular weight of~$\rhobar$ corresponding to~$\chi_i$ can be written as $\sigma(\lambda_J, \rhobar) = (\lambda_0(r_0), \ldots, \lambda_{f-1}(r_{f-1}) )\otimes \det^{e(\lambda)(r)}$ for some~$J$, and so
\[
\chi_i: \fourmatrix{a}{0}{0}{d} \mapsto a^{\lambda(r)} (ad)^{e(\lambda)(r)}
\]
where we have written $\lambda(r) = \sum_{i=0}^{f-1} \lambda_i(r_i)p^i$.
Our goal in this paragraph is to prove the following proposition.

\begin{pp}\label{stateintegerparameter}
The $j$-th $p$-adic digit of~$i^+(\chi_i)$ is
\[
i^+(\chi_i)_j  =
\begin{dcases}
(\delta\lambda)_j(r_j) &\text{when~$\rhobar$ is irreducible}\\
p-1 &\text{when~$\rhobar$ is reducible}.
\end{dcases}
\]
\end{pp}
\begin{proof}
This is a direct consequence of Lemma~\ref{i^+(chi_i)} and Proposition~\ref{integerparameter} to follow.
\end{proof}

\begin{rk}\label{Hureference}
Some of the computations in this section have appeared in~\cite{Hu} for reducible~$\rhobar$ (the combinatorics is similar). To see the compatibility between Proposition~\ref{stateintegerparameter} and~\cite[Lemme~5.4(ii)]{Hu}, notice that the set there denoted~$I(\lambda)$ coincides with $\delta_\red J_\lambda$ in our notation. Since we will also need the (slightly different) irreducible case, we have decided to give a full argument.
\end{rk}

We need to recall~\cite[Lemma 2.2]{BPmodp}. 
The Jordan--H\"older factors of $\Ind_I^{\GL_2(k_L)}(\chi_i^s)$ are in bijection with subsets of $\{ 0, \ldots, f-1 \}$.
We first define a bijection between the set of subsets of $\{ 0, \ldots, f-1 \}$ and a set of formal weights~$\mP(x_i)$, which we denote~$\mu \mapsto J(\mu)$. 
Writing~$\mu$ as an $f$-tuple~$(\mu_i)$, it is characterized by
\[
\mu_i = 
\begin{cases}
p-1-x_i &\text{ if } \ch_{J(\mu)}(i-1, i) = (1, 1)\\ 
x_i - 1 &\text{ if } \ch_{J(\mu)}(i-1, i) = (1, 0)\\ 
p-2-x_i &\text{ if } \ch_{J(\mu)}(i-1, i) = (0, 1)\\
x_i &\text{ if } \ch_{J(\mu)}(i-1, i) = (0, 0).
\end{cases}
\]
Note that $j \in J(\mu)$ if and only if $\mu_j \in \{p-2-x_j, p-1-x_j\}$. 
The Jordan--H\"older factor of~$\Ind_I^{\GL_2(k_L)}(\chi_i^s)$ corresponding to~$J(\mu)$ is $\tau = \mu(\lambda(r)) \otimes \det^{e(\mu)(\lambda(r))}\eta$, where $e(\mu)$ is as defined in (\ref{eqn:e}).

The following lemma is implicit in~\cite{BPmodp} but since it is fundamental to our method we give some details.

\begin{lemma}\label{indices}
Let~$t \in [1, q-2]$ be an integer and define $\chi: \fourmatrix{a}{*}{0}{d} \to a^t\eta(ad)$. 
Let~$\mu \in \mP(x_i)$, corresponding to the constituent $\tau = \mu(t) \otimes \det^{e(\mu)(t)}\eta$ of~$\Ind_I^{\GL_2(k_L)}\chi^s$. 
(We write~$\mu(t)$ for the evaluation of the formal weight~$\mu$ at the $p$-adic digits of~$t$.)
Assume that~$\tau$ is neither the socle nor the cosocle of the induction. 
Let~$\varphi$ be a generator of the $\chi^s$-eigenspace for~$I/I_1$ in~$\Ind_I^{\GL_2(k_L)}\chi^s$.

Then $S_{\sum_{j \in J(\mu)}p^j(p-1-\mu_j(t_j))} \varphi$ generates $\tau^{I_1}$ in every quotient of $\Ind_I^{\GL_2(k_L)}\chi^s$ where~$\tau$ is a subrepresentation.
\end{lemma}
\begin{proof}
Compare~\cite[Lemma~2.7]{BPmodp}. 
By our assumption on~$\tau$, the character~$\tau^{I_1}$ has multiplicity one, hence it is enough to prove that the given vector has the same $H$-eigencharacter as
\begin{displaymath}
\tau^{I_1} = a^{\mu(t)}(ad)^{e(\mu)(t)}\eta(ad) = a^t (a^{-1}d)^{e(\mu)(t)}\eta(ad).
\end{displaymath}
An application of Lemma~\ref{shift} implies that the $H$-eigencharacter of $S_{\sum_{j \in J(\mu)}p^j(p-1-\mu_j(t_j))} \varphi$ is
\begin{displaymath}
a^{t}\eta(ad)(a^{-1}d)^{\sum_{j \in J(\mu)}p^j(p-1-\mu_j(t_j))},
\end{displaymath}
so we need to prove that $e(\mu)(t) \equiv \sum_{j \in J(\mu)}p^j(p-1-\mu_j(t_j)) \mod q-1$. We will actually prove this as an equality in~$\bZ$. 

Write
$e(\mu)_i = 2x_i-p+1, 1, 2x_i-p+2, 0$
according to whether $\ch_{J(\mu)}(i-1, i) = (1, 1), (1, 0), (0, 1), (0, 0)$ respectively.
By definition, we have
\begin{equation}\label{eqn:emu}
e(\mu) = 
\begin{dcases}
\half \left (\sum_{i=0}^{f-1} e(\mu)_i p^i \right ) &\text{ if } \mu_{f-1} \in \{x_{f-1}, x_{f-1}-1\}\\
\half \left (p^f - 1 + \sum_{i=0}^{f-1} e(\mu)_i p^i \right ) &\text{ otherwise}.
\end{dcases}
\end{equation}

We are going to rearrange this sum so that the summands~$-p$ do not appear. Define an integer
\begin{equation}\label{eqn:dmui}
d(\mu)_i = 
\begin{dcases}
e(\mu)_i + p - 1 &\text{ if } \ch_{J(\mu)}(i-1, i) = (1, 1) \\
e(\mu)_i - 1 &\text{ if } \ch_{J(\mu)}(i-1, i) = (1, 0) \\
e(\mu)_i +p &\text{ if } \ch_{J(\mu)}(i-1, i) = (0, 1) \\
e(\mu)_i &\text{ if } \ch_{J(\mu)}(i-1, i) = (0, 0)
\end{dcases}
= 
\begin{dcases}
2x_i &\text{ if } \ch_{J(\mu)}(i-1, i) = (1, 1). \\
0 &\text{ if } \ch_{J(\mu)}(i-1, i) = (1, 0). \\
2x_i + 2 &\text{ if } \ch_{J(\mu)}(i-1, i) = (0, 1). \\
0 &\text{ if } \ch_{J(\mu)}(i-1, i) = (0, 0).
\end{dcases}
\end{equation}
Then 
\begin{equation}\label{eqn:dmu}
\sum_i d(\mu)_ip^i = 
\begin{dcases}
\sum_i e(\mu)_ip^i & \text{ if } f-1 \notin J(\mu)\\
p^f - 1 + \sum_i e(\mu)_ip^i & \text{ if } f-1 \in J(\mu).
\end{dcases}
\end{equation}
(It may be helpful to visualize this process as ``carrying" a coefficient $e(\mu)_i$ to the left, when it is equal to~$1$.)

Since $f-1 \notin J(\mu)$ is equivalent to~$\mu_{f-1} \in \{x_{f-1}-1, x_{f-1}\}$, we deduce from (\ref{eqn:emu}), (\ref{eqn:dmui}), and (\ref{eqn:dmu}) that
\[
e(\mu) = \half \left (\sum_{i-1 \not \in J(\mu), i \in J(\mu)}(2x_i+2)p^i + \sum_{i-1, i \in J(\mu)} 2 x_i p^i \right )
\]
and since $x_i+1 = p-1-(p-2-x_i)$ and~$x_i = p-1-(p-1-x_i)$, the claim follows.
\end{proof}

\begin{lemma}\label{formalcomposition}
Let~$\lambda$ be a formal weight for~$\rhobar$ and let $\mu \in \mP(x_i)$ be such that $\delta(\lambda) = \mu \circ \lambda$ (that is, its $i$-th component is the composition~$\mu_i \circ \lambda_i$). Then
\begin{displaymath}
\mathrm{det}^{e(\mu\circ\lambda)(r)} = \mathrm{det}^{e(\mu)(\lambda(r))}\mathrm{det}^{e(\lambda)(r)}.
\end{displaymath}
\end{lemma}
\begin{proof}
Working in the group of formal weights under composition, we see that the expression for~$e(\mu)$ contains a $p^f-1$ if and only if exactly one of~$e(\delta(\lambda))$ and~$e(\lambda)$ contains a~$p^f-1$.
The rest follows from the identity $r_i - (\mu\circ\lambda)_i(r_i) = r_i - \lambda_i(r_i) + \lambda_i(r_i) - \mu_i(\lambda_i(r_i))$.
\end{proof}

\begin{lemma}\label{i^+(chi_i)}
Let~$\chi_i = \sigma(\lambda_J, \rhobar)^{I_1}$, as before. Assume $\delta(\lambda) = \mu \circ \lambda$ with $\mu \in \mP(x_i)$. Then
\[
i^+(\chi_i) = q-1-\sum_{j \in J(\mu)} p^j(p-1-(\delta\lambda)_j(r_j)).
\]
\end{lemma}
\begin{proof}
By our assumptions, $\chi_{i+1}$ is the $I/I_1$-character appearing in the weight $((\delta \lambda)_0(r_0), \ldots, (\delta\lambda)_{f-1}(r_{f-1})) \otimes \det^{e(\delta \lambda)(r)}$, where by definition $\delta \lambda$ is the formal weight corresponding to the set $\delta J$. 
We claim that the hypotheses of Lemma~\ref{indices} hold: then it follows from Lemma~\ref{formalcomposition} and Lemma~\ref{indices} that if $\delta (\lambda) = \mu \circ \lambda$ then $S_{\sum_{j \in J(\mu)} p^j(p-1-(\delta\lambda)_j(r_j))}$ sends $\chi_i^s$ to~$\chi_{i+1}$. 
But then $i^+(\chi_i) = q-1-\sum_{j \in J(\mu)} p^j(p-1-(\delta\lambda)_j(r_j))$ by Lemma~ \ref{shift}. 

To check the assumptions of Lemma~\ref{indices}, observe that $\chi_{i+1}$ is the eigencharacter of 
\[
\soc_K \operatorname{im}(\thetabar(\chi_i^s) \to D_0(\rhobar)). 
\]
If $\thetabar(\chi_i^s)$ embeds in~$D_0(\rhobar)$ then $R\chi_i^s = (\soc_K \thetabar(\chi_i^s))^{I_1} = \chi_i$, which contradicts out assumptions. Similarly, if $\thetabar(\chi_i^s) \to D_0(\rhobar)$ factors through $\cosoc_K \thetabar(\chi_i^s)$ then $R \chi_i^s = \chi_i^s$, another contradiction.
\end{proof}
To conclude the proof of Proposition~\ref{stateintegerparameter}, it is enough to prove the following.

\begin{pp}\label{integerparameter}
Let~$\lambda$ be a formal weight of~$\rhobar$ and define $\mu \in \mP(x_i)$ by $\delta(\lambda) = \mu \circ \lambda$. Then $J(\mu) = J \triangle \delta J$.
\end{pp}
\begin{proof}
The group of formal weights maps onto $W$, the Weyl group of~$\Res_{\bQ_{p^f}/\bQ_p} \GL_2$, according to the sign of the leading coefficients. 
By construction, $j \in J(\mu)$ if and only if~$\mu_j$ has negative leading coefficient, but this happens if and only if $\lambda_j, (\delta \lambda)_j$ have opposite leading coefficients. 
By construction, this is equivalent to $j \in J \triangle \delta J$.
\end{proof}

\paragraph{Computation of~$S_{i^+(\psi_1^s)}^+\cdots S_{i^+(\psi_k^s)}^+(x)$.}\label{irreducible2} 
By construction, Proposition~\ref{pp:contraction} applies to the product $S_{i^+(\psi_1^s)}^+\cdots S_{i^+(\psi_k^s)}^+$.
Let
\[
\beta = \beta(i^+(\psi_1^s), \ldots, i^+(\psi_k^s)) \in W(k_L)
\]
be as in Proposition \ref{pp:contraction}, and choose  $x \in \theta^{\psi_k^s}$ as in Proposition~\ref{nextalpha}, where~$\theta[1/p]$ denotes the central type for~$\rhobar$.
Then Lemma~\ref{S_0^+vanishing} implies that
\begin{equation}\label{contraction}
S_{i^+(\psi_1^s)}^+\cdots S_{i^+(\psi_k^s)}^+(x) = \beta x \in \theta^{\psi_k^s}.
\end{equation}

\begin{pp}\label{coeffirreducible}
We have $v_p(\beta) = \half k |J \triangle \delta J|$ and the leading term of~$\beta$ is congruent to
\begin{displaymath}
(-1)^{\half k |J \triangle \delta J| + \frac{kh}{2f}(1+\sum r_i)}
\end{displaymath}
modulo~$p$, where the invariant~$h$ is defined in~\cite[Lemme~6.2(ii)]{Breuilparameters} (we will not need the explicit definition).
\end{pp}

\begin{rk}
The equality $v_p(\beta) = \half k |J \triangle \delta J|$ follows from Proposition~\ref{productvalue} and the fact that $v_p(\beta)=v_p(\prod_i \tld{U}_p(\chi_i))$. Here, we will give a different proof of this fact.
\end{rk}

\begin{proof}
By Proposition~\ref{pp:contraction} we know that the leading term~$c(\beta)$ of~$\beta$ is the Teichm\"uller lift of
\begin{displaymath}
(-1)^{v_p(\beta) + (k-2)(f-1)} \prod_{\substack{0 \leq t \leq k-1 \\ 0 \leq j \leq f-1}} i^+(\psi_t^s)_j!
\end{displaymath}
and the valuation is
\begin{equation}\label{valuationterm}
v_p(\beta) = \frac{1}{p-1} \left ( \sum_{\substack{0 \leq t \leq k-1 \\ 0 \leq j \leq f-1}}p - 1 - i^+(\psi_t^s)_j \right )
\end{equation}

We begin to compute $c(\beta)$.
By Proposition~\ref{stateintegerparameter} and the proof of~\cite[Lemme~6.2(ii)]{Breuilparameters}, we have 
\begin{displaymath}
\prod_{\substack{0 \leq t \leq k-1 \\ 0 \leq j \leq f-1}}(p-1- i^+(\psi_t^s)_j)! = (-1)^{\frac{kh}{2f}(1+\sum r_i)}.
\end{displaymath}
Hence $c(\beta)$ is also congruent modulo~$p$ to
\begin{equation}\label{partU}
(-1)^{v_p(\beta) + (k-2)(f-1) + kf+\sum_{j, t}(p-1- i^+(\psi_t^s)_j)+\frac{kh}{2f}(1+\sum r_i)},
\end{equation}
using that $x! (p-1-x)! \equiv (-1)^{x+1} \pmod p$ for $0\leq x\leq p-1$.
Hence, since~$k$ is even and $p-1$ divides $\sum_{j, t}(p-1- i^+(\psi_t^s)_j)$, we see that~(\ref{partU}) also equals
\begin{displaymath}
(-1)^{v_p(\beta)+\frac{kh}{2f}(1+\sum r_i)}.
\end{displaymath}

Now we consider~$v_p(\beta)$.
There is some cancellation in pairs in the sum~(\ref{valuationterm}). If the character~$\chi_{k-1}$ corresponds to~$J \subset \{0, \ldots, f-1\}$, then we can apply Proposition~\ref{stateintegerparameter} and rewrite~(\ref{valuationterm}) as
\begin{equation}\label{eqn:p-1}
\frac{1}{p-1} \left (  \sum_{j \in J \triangle \delta J} p-1-(\delta\lambda)_j(r_j) + \sum_{j \in \delta J \triangle \delta^2 J} p-1-(\delta^2 \lambda)_j(r_j) + \cdots + \sum_{j \in \delta^{k-1} J \triangle J} p-1-\lambda_j(r_j) \right ).
\end{equation}
Recall that the complement $J^c$ is contained in the $\delta$-orbit of~$J$: more precisely, it is equal to $\delta^{k/2}(J)$, by definition of~$\delta = \delta_\irr$. 
By Lemma~\ref{complement} to follow, we find that
\[
p-1-\delta\lambda_j(r_j) + p-1-(\delta^{k/2+1}\lambda)_j(r_j) = p-1.
\]
Since the integer $|J \triangle \delta J|$ is constant throughout the $\delta$-orbit of~$J$, (\ref{eqn:p-1}) equals~$\half k |J \triangle \delta J|$.
\end{proof}

\begin{lemma}\label{complement}
If $j \in J \triangle \delta J$, then $\lambda_{\delta J, j} + \lambda_{\delta(J^c), j} = p-1$.
\end{lemma}
\begin{proof}
By Lemma~\ref{deltabehaviour}, it suffices to prove that if $j \in \delta^{-1}J \triangle J$ then $\lambda_{J, j} + \lambda_{J^c, j} = p-1$. This is immediate from the definition of $J \mapsto \lambda_J$, see~(\ref{formalweights}).
\end{proof}

\paragraph{Computation of~$\prod_{i=1}^{k} \tld{c}(\psi_i^s)$.}\label{irreducible3} 

Recall that the constants~$c(\psi_i^s)$ are defined by~$S_{i(\psi_i)}S_0 \varphi^{\psi_i^s} = c(\psi_i^s) S^+_{i^+(\psi_i^s)} \varphi^{\psi_i^s}$, that~$i(\psi_i) = q-1 - i^+(\psi_i^s) \not \in \{0, q-1\}$, and that~$\tld{c}(\psi_i^s)$ is the leading term of~$c(\psi_i^s)$.
In this section we compute the product~$\prod_{i=1}^{k} \tld{c}(\psi_i^s)$ by the method of~\cite{BD}, and we prove that it equals~$1$.
	
Fix an index~$i$. The character~$\psi_i^s = \chi_{k-1-i}$ appears in the socle of~$D_0(\rhobar)$ and so it is the $I/I_1$-eigencharacter of a modular weight of~$\rhobar$. We write~$\lambda$ for the corresponding formal weight, and~$J$ for the corresponding subset of~$\{0, \ldots, f-1\}$. We have therefore the formula $\psi_i^s: \fourmatrix{a}{0}{0}{d} \mapsto [a]^{\lambda(r)}[ad]^{e(\lambda)(r)}$. 

\begin{lemma}\label{coefficientsJacobi}
We have an equality
\begin{displaymath}
c(\psi_i^s) = \bJ(i(\psi_i), q-1-\lambda(r)).
\end{displaymath}
\end{lemma}
\begin{proof}
Our genericity assumptions imply that $i(\psi_i), \lambda(r) \not \in \{0, q-1\}$. So we have $\bJ(i(\psi_i), q-1-\lambda(r)) = \bJ_0(i(\psi_i), q-1-\lambda(r))$, and applying Lemma~\ref{relations1} we find
\begin{displaymath}
S_{i(\psi_i)}S_0 \varphi^{\psi_i^s} =  [(-1)^{e(\lambda)(r)}] \bJ_0(i(\psi_i), q-1-\lambda(r)) S_{i(\psi_i)- \lambda(r)}\varphi^{\psi_i^s}.
\end{displaymath}
The assumptions of the lemma are satisfied because $i(\psi_i) \not \in \{0, q-1\}$, and if $q-1 | i(\psi_i) - \lambda(r)$ then $R\psi_i = \psi_i$, which contradicts our assumption $R\chi_{k-1-i}^s \ne \chi_{k-1-i}^s$. (Recall that, by definition~\ref{defn:R}, the operator~$S_{i(\psi_i)}$ applied to the $\psi_i$-eigenvector~$S_0\varphi^{\psi_i^s}$ yields an $R\psi_i$-eigenvector. On the other hand, $S_{\lambda(r)}$ applied to a~$\psi_i$-eigenvector would yield a $\psi_i^s\alpha^{-\lambda(r)} = \psi_i$-eigenvector, by Lemma~\ref{shift}.) 
Hence we have an equality
\[
[(-1)^{e(\lambda)(r)}] \bJ_0(i(\psi_i), q-1-\lambda(r)) S_{i(\psi_i)- \lambda(r)}\varphi^{\psi_i^s} = c(\psi_i^s) S_{i^+(\psi_i^s)}^+ \varphi^{\psi_i^s}.
\]
To conclude, we are going to apply an auxiliary operator~$S_z$ to both sides of this equality. The number~$z$ can be chosen arbitarily, except that we require that Lemma~\ref{relations1} and Lemma~\ref{relations3} can be applied to deduce that
\begin{gather*}
S_z S_{i(\psi_i)- \lambda(r)}\varphi^{\psi_i^s} = [(-1)^{e(\lambda)(r)}] \bJ_0(z,q-1-i(\psi_i)) S_{z - i(\psi_i)} \varphi^{\psi_i^s},\\
S_z S_{i^+(\psi_i^s)}^+ \varphi^{\psi_i^s} = \bJ_0(z, i^+(\psi_i^s)) S_{z+i^+(\psi_i^s)} \varphi^{\psi_i^s}.
\end{gather*}
This rules out at most four choices of~$z$, which is fine since $p \geq 5$. 
Since $i(\psi_i) = q-1-i^+(\psi_i^s)$ and these Jacobi sums do not vanish, we deduce that
\[
\bJ_0(i(\psi_i), q-1- \lambda(r)) S_{z+i^+(\psi_i^s)} \varphi^{\psi_i^s} = c(\psi_i^s) S_{z+i^+(\psi_i^s)} \varphi^{\psi_i^s}.
\]
The lemma follows since~$S_{z+i^+(\psi_i^s)}\varphi^{\psi_i^s} \not = 0$ (all the $S_i$-operators are nonzero on $\varphi^{\psi_i^s}$, which generates the cosocle of the lattice~$\Ind_I^K(\psi_i^s)$ where these computations are taking place).
\end{proof}

\begin{lemma}\label{firstsimplification}
Define $\epsilon_j = 1$ if $j \not \in J \triangle \delta J \text{ and } j-1 \in J \triangle \delta J$, and $\epsilon_j = 0$ otherwise.  
Then $v_p(c(\psi_i^s)) = f - |J \triangle \delta J|$, and the leading term of~$c(\psi_i^s)$ is
\[
(-1)^{f-1+v_p(c(\psi_i^s))} \prod_{j \in J \triangle \delta J}\frac{(p-1-\lambda_j(r_j))!(p-1-(\delta \lambda)_j(r_j))!}{0!} \prod_{j \not \in J \triangle \delta J}\frac{(p-1-\lambda_j(r_j))!}{(p-1-\lambda_j(r_j) + \epsilon_j)!}.
\]
\end{lemma}
\begin{proof}
In this proof we shorten notation to~$\lambda_j = \lambda_j(r_j)$. 
For $j \in [0,f-1]$, let $\Sigma_j \in [0,p-1]$ and $Q \in \{0,1\}$ be such that
\[
\sum_{j \in \{0, \ldots, f-1\}} p^j(p-1-\lambda_j) + \sum_{j \in J \triangle \delta J} p^j(p-1-(\delta\lambda)_j) = \sum_{j \in \{ 0, \ldots, f-1\}} \Sigma_j p^j + Q(q-1).
\]
Then $\Sigma_j$ is equal to
\[
\Sigma_j =
\begin{dcases}
0 &\text{ if } j \in J \triangle \delta J\\
p-1-\lambda_j + \epsilon_j &\text{ if } j \notin J \triangle \delta J.
\end{dcases}
\]
This follows immediately (by ``carrying") from the fact that
\[
(p-1-\lambda_j) + (p-1-(\delta \lambda)_j) =
\begin{dcases}
 p &\text{ if } \ch_{J \triangle \delta J}(j, j-1) = (1, 0)\\ 
p-1 & \text{ if } \ch_{J \triangle \delta J}(j, j-1) = (1, 1),
\end{dcases}
\]
as can be checked directly from~(\ref{formalweights}) (or see the appendix~\S A).

By Lemma~\ref{coefficientsJacobi}, Proposition~\ref{stateintegerparameter} and Theorem~\ref{Stickelbergertheorem}, it follows~that
\begin{align*}
v_p(c(\psi_i^s)) & = \frac 1 {p-1} \left (\sum_{j \in J \triangle \delta J} p-1-(p-1-\lambda_j+p-1-\delta\lambda_j - \Sigma_j) + \sum_{j \not \in J \triangle \delta J} p-1-(p-1-\lambda_j + 0 -\Sigma_j) \right )\\
& = \frac{1}{p-1} \left ( \sum_{j \in J \triangle \delta J} p-1-(p-1-\lambda_j + p-1-\delta\lambda_j) + \sum_{j \not \in J \triangle \delta J} p-1+\epsilon_j \right ) \\
& = f-|J \triangle \delta J| + \frac{1}{p-1} \left ( \sum_{j \in J \triangle \delta J, j-1 \not \in J \triangle \delta J} -1 + \sum_{j \not \in J \triangle \delta J, j-1 \in J \triangle \delta J} 1 \right ) = f-|J \triangle \delta J|
\end{align*}
and the leading term 
\[
(-1)^{f-1+v_p(c(\psi_i^s))} \prod_{j \in J \triangle \delta J} \frac{(p-1-\lambda_j)!(p-1-\delta\lambda_j)!}{\Sigma_j!} \prod_{j \not \in J \triangle \delta J} \frac{(p-1-\lambda_j)!}{\Sigma_j!}
\]
is as claimed.

\end{proof}

\begin{lemma}\label{secondsimplification}
The leading term of~$c(\psi_i^s)$ is the Teichm\"uller lift of
\begin{displaymath}
(-1)^{f-1+v_p(c(\psi_i^s))+ |J \triangle \delta J|+\sum_{j \in J \triangle \delta J}(p-1-\delta \lambda_j(r_j))} \frac{\prod_{j \in J \triangle \delta J, j-1 \not \in J \triangle \delta J}(p-1- \lambda_j(r_j))}{\prod_{j \not \in J \triangle \delta J, j-1 \in J \triangle \delta J}(p-\lambda_j(r_j))}.
\end{displaymath}
\end{lemma}
\begin{proof}
Again we shorten notation to~$\lambda_j = \lambda_j(r_j)$. We simplify the two factors in the expression of Lemma~\ref{firstsimplification}.
For the second factor, by definition of~$\epsilon_j$ we have cancellation of numerator with denominator whenever $j, j-1 \not \in J \triangle \delta J$.
In the numerator of the first factor, we apply the identity  $x!(p-1-x)! \equiv (-1)^{x+1} \mod p$ together with
\[
p-1- \lambda_j = p-1 - (p-1-\delta\lambda_j) \text{ when } \ch_{J \triangle \delta J}(j, j-1) = (1, 1),
\]
to simplify the whole leading term to
\[
(-1)^{f-1+v_p(c(\psi_i^s))+\sum_{j,\, j-1 \in J \triangle \delta J}(p-1-\delta\lambda_j+1)}\prod_{\ch_{J \triangle \delta J}(j ,j-1) = (1, 0)}(p-1-\lambda_j)!(p-1-\delta\lambda_j)!\prod_{\ch_{J \triangle \delta J} (j, j-1) = (0, 1)} \frac{(p-1-\lambda_j)!}{(p-\lambda_j)!}
\]
To conclude, we simplify the second factor to $\prod_{\ch_{J \triangle \delta J}(j, j-1) = (0, 1)} \frac 1 {p-\lambda_j}$, and we apply the equality
\[
p-1-\lambda_j = p - (p-1- \delta\lambda_j) \text{ when } \ch_{J \triangle \delta J}(j, j-1) = (1, 0)
\]
together with $x!(p-1-x)! \equiv (-1)^{x+1} \mod p$ to simplify the first factor to
\[
(-1)^{\sum_{\ch_{J \triangle \delta J}(j, j-1) = (1, 0)} (p-1-\delta \lambda_j + 1)}\prod_{\ch_{J \triangle \delta J}(j, j-1) = (1, 0)}(p-1-\lambda_j).
\]
\end{proof}

\begin{pp}\label{prodirreducible}
The product $\prod_{i=1}^k c(\psi_i^s)$ has valuation~$kf - k|J\triangle \delta J|$ and leading term equal to~$1$.
\end{pp}
\begin{proof}
By Lemma~\ref{firstsimplification}, $v_p(c(\psi_i^s))$ does not depend on~$i$ and the product has valuation $kf - k|J \triangle \delta J|$. For the leading term, we have seen in the proof of Proposition~\ref{coeffirreducible} that
$\sum_{i=0}^{k-1}\sum_{j \in \delta^i J \triangle \delta^{i+1} J}(p-1-\delta^{i+1}\lambda_j(r_j))$ is divisible by~$p-1$, hence it is even when $p$ is odd. It follows that the product of all the sign terms from Lemma~\ref{secondsimplification} is a positive sign, because~$k$ is even. 

To conclude, we need to prove that
\[
\prod_{t=0}^{k-1}\frac{\prod_{j \in \delta^t J \triangle \delta^{t+1}J, j-1 \not \in \delta^t J \triangle \delta^{t+1}J}(p-1-(\delta^t\lambda)_j(r_j))}{\prod_{j \not \in \delta^t J \triangle \delta^{t+1}J, j-1 \in \delta^t J \triangle \delta^{t+1}J} (p-(\delta^t\lambda)_j(r_j))} = 1.
\]
For $\fj \in \{J, \delta J \ldots, \delta^{k-1} J\}$, 
let us define
\begin{align*}
(j, \fj) & = p-1-\lambda_{\fj, j}(r_j) \text{ if } j \in \fj \triangle \delta \fj, j-1 \not \in \fj \triangle \delta \fj\\
	& =  (p-\lambda_{\fj, j}(r_j))^{-1} \text{ if } j \not \in \fj \triangle \delta \fj, j-1 \in \fj \triangle \delta \fj\\
	& = 1 \text{ otherwise}.
\end{align*}
Then we need to prove
\begin{equation}\label{toprove}
\prod_{j=0}^{f-1}\prod_{\fj \in \{J, \ldots, \delta^{k-1}J\}} (j, \fj) = 1,
\end{equation}
and for this it will suffice to fix an index~$j$ and prove that the inner product is equal to~$1$.

Having fixed~$j$, we claim that there are two commuting involutions on the set 
\[
\mS_j = \{ \fj \in \{J, \ldots, \delta^{k-1}J\} : \ch_{\fj \triangle \delta \fj}(j, j-1) = (1, 0) \text{ or } (0, 1)\}.
\]
One is the complement $\fj \mapsto \fj^c$, since $\fj \triangle \delta \fj = \fj^c \triangle \delta \fj^c$. The other, denoted~$\fj \mapsto *\fj$, is defined in the same way as in the proof of Lemma~\ref{differenceodd}: for instance, if $\ch_{\fj \triangle \delta \fj}(j, j-1) = (1, 0)$, then $*\fj$ is the first element of $\mS_j$ following~$\fj$ such that $\ch_{*\fj \triangle \delta *\fj}(j, j-1) =~(0, 1)$. This is well-defined since $\delta_{\mathrm{red}}(\fj \triangle \delta \fj) = \delta \fj \triangle \delta^2 \fj$.

Decomposing~$\mS_j$ in orbits for the resulting action of~$\bZ/2 \times \bZ/2$, it suffices to prove that, if~$\fj \in \mS_j$, then
\[
(j, \fj)(j, \fj^c)(j, *\fj)(j, *\fj^c) = 1.
\]

To do so, we assume without loss of generality that $j \in \fj \triangle \delta \fj$ (otherwise, we relabel~$\fj$ to~$*\fj$). Now it suffices to prove that
\[
\{p-1-\lambda_{\fj, j}(r_j), p-1-\lambda_{\fj^c, j}(r_j)\} = \{p-\lambda_{*\fj, j}(r_j), p-\lambda_{*\fj^c, j}(r_j)\},
\]
or equivalently that
\[
\{\lambda_{\fj, j}, \lambda_{\fj^c, j}\} = \{\lambda_{*\fj, j}-1, \lambda_{*\fj^c, j}-1\}.
\]

Knowing that $\ch_{\fj \triangle \delta \fj}(j, j-1) = (1, 0)$ does not determine~$\lambda_{\fj, j}$ uniquely. However, since $\ch_\fj + \ch_{\fj^c} = 1$, it determines $\{\lambda_{\fj, j}, \lambda_{\fj^c, j}\}$. More precisely, we have 
\[
\{\ch_\fj(j-1, j, j+1), \ch_{\fj^c}(j-1, j, j+1)\} = \begin{cases}
\{(0, 0, 1), (1, 1, 0)\} \text { if } 0 < j < f-1\\
\{(1, 0, 1), (0, 1, 0)\} \text{ if } j = 0\\
\{(1, 1, 1), (0, 0, 0)\} \text{ if } j = f-1,
\end{cases}
\]
and we deduce from the definition of $\fj \mapsto \lambda_\fj$ (see~(\ref{formalweights})) that
\[
\{\lambda_{\fj, j}, \lambda_{\fj^c, j}\} = \{p-3-x_j, x_j\} \text{ or } \{x_0-1, p-2-x_0\}.
\]
An entirely analogous argument using $\ch_{*\fj \triangle \delta *\fj}(j, j-1) = \ch_{*\fj^c \triangle \delta *\fj^c}(j, j-1) = (0, 1)$ implies that
\[
\{\lambda_{*\fj, j}, \lambda_{*\fj^c, j}\} = \{x_j+1, p-2-x_j\} \text{ or } \{p-1-x_0, x_0\}
\]
and the claim follows.

\end{proof}

\subsection{The split reducible case.}\label{sec:red}
In this section, $\rhobar$ is a split reducible representation with tame inertial weights~$r_i$ (defined as in~\S\ref{Galoisreps}).
The situation is similar to the irreducible case, with some combinatorial differences, and we will indicate how the computations of the previous sections can be modified to account for this case. Recall the following notational convention: for $J \subset \{0, \ldots, f-1\}$ and $x, y, z \in \{0, 1\}$, we write $\br(j, J) = (x, y, z)$ if $\ch_J(j-1, j, j+1) = (x, y, z)$.

The weights of~$\rhobar$ are again in bijection with the subsets of~$\{0, \ldots, f-1\}$. If~$\chi_i = \sigma(\lambda_J, \rhobar)^{I_1}$ corresponds to~$J$, the subset corresponding to~$\chi_{i+1} = R\chi_i^s$ is~$\delta_{\mathrm{red}}(J)$, by the proof of~\cite[Proposition~5.1]{Breuilparameters} (compare also~\cite[Lemma~15.2]{BPmodp}). 
In this section, we write~$\delta$ for~$\delta_{\mathrm{red}}$.

\paragraph{Computation of~$i^+(\psi_{k-1-i}^s)$.}\label{reducible1}
Assume $\chi_i = \sigma(\lambda, \rhobar)^{I_1}$ and~$\lambda = \lambda_J$, so that $j \in J$ if and only if $\lambda_j \in \{p-2-x_j, p-3-x_j \}$, as in~\cite{Breuilparameters} and~\S\ref{sec:serrewts}.
The following proposition is proved in the same way as Proposition~\ref{stateintegerparameter}.

\begin{pp}
Write~$i^+(\psi_{k-1-i}^s) = i^+(\chi_i) = \sum_{j=1}^{f-1} i^+(\chi_i)_jp^j$. Then we have the equality
\begin{align*}
i^+(\chi_i)_j  =\, & \delta(\lambda)_j \text{ if } j \in J \triangle \delta J\\
&  p-1 \text{ if } j \not \in J \triangle \delta J.
\end{align*}
\end{pp}

\paragraph{Computation of~$S_{i^+(\psi_1^s)}^+\cdots S_{i^+(\psi_k^s)}^+(x)$.}\label{reducible2}  As in the previous section, there exists a nonzero 
\[
\beta = \beta(i^+(\psi_1^s), \ldots, i^+(\psi_k^s)) \in~W(k_L)
\] 
such that
\[
S_{i^+(\psi_1^s)}^+\cdots S_{i^+(\psi_k^s)}^+(x) = \beta x
\]
in the central type of~$\rhobar$ (where~$x$ is as in Proposition~\ref{nextalpha}). 
Here, however, $k$ need not be even. 

\begin{pp}\label{coeffreducible}
The valuation of~$\beta$ is
\[
v_p(\beta) = \half k |J \triangle \delta J|,
\]
and its leading term is the Teichm\"uller lift of
\[
(-1)^{v_p(\beta) +\frac{kh}{f}\sum_{j=0}^{f-1}r_j}\cdot(-1)^k.
\]

\end{pp}

\begin{rk}
Again, the computation of~$v_p(\beta)$ also follows from Proposition~\ref{productvalue}, and we will not need the explicit definition of the invariant~$h$, which can be found in~\cite[Lemme~6.2]{Breuilparameters}. 
\end{rk}

\begin{proof}

We apply Proposition~\ref{pp:contraction} and the identity appearing in the proof of~\cite[Lemme~6.2(i)]{Breuilparameters} to deduce that the leading coefficient~$c(\beta)$ is the Teichm\"uller lift of
\begin{displaymath}
(-1)^{v_p(\beta) +\frac{kh}{f}\sum_{j=0}^{f-1}r_j}\cdot(-1)^k
\end{displaymath}
where the factor~$(-1)^k$ comes from the fact that $(k-2)(f-1) + kf$ has the same parity as~$k$. 

To compute the valuation, we consider the equality
\begin{displaymath}
v_p(\beta) =  \frac{1}{p-1} \left (  \sum_{t=0}^{k-1}\sum_{j \in \delta^{t}J \triangle \delta^{t+1} J} p-1-(\delta^{t+1}\lambda)_j \right )
\end{displaymath}
where~$J$ corresponds to~$\chi_{k-1}$ (of course, since we are summing over~$t$, we could have taken any other~$\chi_i$). 

Fix an index~$j$ and sum over~$t$. Recall that $\delta^t J \triangle \delta^{t+1}J = \delta^t(J \triangle \delta J)$, so that
\begin{equation}\label{shiftdelta}
\frac{1}{p-1} \left (  \sum_{t=0}^{k-1}\sum_{j \in \delta^{t}J \triangle \delta^{t+1} J} p-1-(\delta^{t+1}\lambda)_j(r_j) \right ) = \frac{1}{p-1} \left (  \sum_{t=0}^{k-1}\sum_{j+t \in J \triangle \delta J} p-1-(\delta\lambda)_{j+t}(r_j) \right )
\end{equation}
(Recall that $\delta^{t+1}(\lambda)_j$ is a polynomial in a variable~$x_j$, and it depends on~$\ch_J(j+t, j+t-1)$ in the same way as $(\delta \lambda)_{j+t}$, which is a polynomial in~$x_{j+t}$. Here we are evaluating~$x_{j+t}$ at~$r_j$.)

Since~$\ch_J$ is periodic of period~$k$, so is~$\ch_{J \triangle \delta J}$, hence we can rewrite~(\ref{shiftdelta}) as
\[
\frac k {f(p-1)} \left (\sum_{t=0}^{f-1} \sum_{j+t \in J \triangle \delta J} p-1- (\delta \lambda)_{j+t}(r_j) \right ) = \frac k {f(p-1)} \left ( \sum_{x \in J \triangle \delta J} p-1-(\delta \lambda)_x(r_j) \right ).
\]
By definition~(\ref{formalweights}), we see that if $x\in J \Delta \delta J$, then
\[
p-1-(\delta\lambda)_x(r_j) =
\begin{dcases}
p-2-r_j &\text{ if } x\in J \\ 
r_j+1 &\text{ if } x \notin  J
\end{dcases}
\]
In Lemma~\ref{differenceodd}, we have constructed an involution~$*$ on~$J \triangle \delta J$, with the property that 
$x \in J$ if and only if $*x \notin J$.
It follows immediately that
\[
\frac k {f(p-1)} \left ( \sum_{x \in J \triangle \delta J} p-1-(\delta \lambda)_x(r_j) \right ) = \frac k {f(p-1)} \half |J \triangle \delta J|(p-1).
\]
The claim follows since summing over indices~$j$ multiplies this number by~$f$.
\end{proof}

\paragraph{Computation of $\prod_{i=1}^{k} \tld{c}(\psi_i^s)$.}\label{reducible3}
Fix an index~$i$, and assume~$\psi_i^s$ corresponds to a subset~$J$. Here Lemmas~\ref{firstsimplification} and~\ref{secondsimplification} are still valid, and proved in the same way, hence 
\[
v_p(c(\psi_i^s)) = f - |J \triangle \delta J|
\]
and the leading coefficient~$\tld{c}(\psi_i^s)$ is the Teichm\"uller lift of
\begin{displaymath}
(-1)^{f-1+v_p(c(\psi_i^s))+\sum_{j \in J \triangle \delta J}(p-1-\delta \lambda_j(r_j)) + |J \triangle \delta J|} \frac{\prod_{j \in J \triangle \delta J, j-1 \not \in J \triangle \delta J}(p-1- \lambda_j(r_j))}{\prod_{j \not \in J \triangle \delta J, j-1 \in J \triangle \delta J}(p-\lambda_j(r_j))}.
\end{displaymath}

\begin{pp}\label{prodreducible}
We have the equality $\prod_{i=1}^{k} \tld{c}(\psi_i^s) = (-1)^{k(f-1)+kf}$.
\end{pp}

\begin{proof}
Let~$J$ correspond to~$\psi_k^s$. In the product $\prod_{i=1}^{k} \tld{c}(\psi_i^s)$ we first notice a sign contribution of
\begin{displaymath}
(-1)^{k(f-1)+ kv_p(c(\psi_i^s)) + k|J \triangle \delta J|} = (-1)^{k(f-1)+kf}
\end{displaymath}
(it will cancel together with the factor of $(-1)^k$ from the previous section, since the exponent has the same parity as~$k$). What remains is
\begin{equation}\label{eqn:sgnc}
(-1)^{\sum_{t=0}^{k-1}\sum_{j \in \delta^t J \triangle \delta^{t+1} J}(p-1-\delta^{t+1} \lambda_j(r_j))} \prod_{t=0}^{k-1}\left ( \frac{\prod_{j \in \delta^tJ \triangle \delta^{t+1} J, j-1 \not \in \delta^t J \triangle \delta^{t+1} J}(p-1- \delta^t\lambda_j(r_j))}{\prod_{j \not \in \delta^t J \triangle \delta^{t+1} J, j-1 \in \delta^t J \triangle \delta^{t+1} J}(p-\delta^t\lambda_j(r_j))} \right ).
\end{equation}
The sign is positive because the exponent is divisible by~$p-1$. 
Note also that $j-1 \in \delta^t J \Delta \delta^{t+1}J$ if and only if $j \in \delta^{t-1} J \Delta \delta^t J$.
Then (\ref{eqn:sgnc}) is equal to 
\begin{gather*}
\prod_{j=0}^{f-1} \prod_{\fj \in \{J, \ldots, \delta^{k-1} J\}} \frac {\prod_{j\in \delta^{-1}\fj \Delta \delta \fj,\,  j \notin \delta^{-1}\fj \Delta \fj}p-1-\lambda_{\fj, j}(r_j)}{\prod_{j\in \delta^{-1}\fj \Delta \delta \fj,\,  j \in \delta^{-1}\fj \Delta \fj} p - \lambda_{\fj, j}(r_j)}.
\end{gather*}

Fix an index~$j$. We will prove that the corresponding factor in the product above is equal to~$1$. 
By~(\ref{formalweights}) (see also \S \ref{sec:app}), for any $\fj \in \{J, \ldots, \delta^{k-1} J\}$ we see that
\begin{gather*}
p-1-\lambda_{\fj, j}(r_j) = \begin{cases}
 r_j+2  &\text{ if } \br(j, \fj) = (1, 1, 0)\\
  p-1-r_j &\text{ if } \br(j, \fj) = (0, 0, 1)
\end{cases}
\qquad
p-\lambda_{\fj, j}(r_j) = \begin{cases}
 r_j+2  &\text{ if } \br(j, \fj) = (0, 1, 1)\\
  p-1-r_j &\text{ if } \br(j, \fj) = (1, 0, 0)
\end{cases}
\end{gather*}

If $\br(j, \fj) = (1, 1, 0), (0, 0, 1), (0, 1, 1), (1, 0, 0)$ (equivalently, $j \in \delta^{-1} \fj \triangle \delta \fj$), define
\[
f(\fj,j)=
\begin{dcases}
r_j+2 &\text{ if }  \ch_{\fj}(j-1, j) = (1, 1)\\ 
(p-1-r_j)^{-1} &\text{ if } \ch_{\fj}(j-1, j) = (1, 0)\\ 
(r_j+2)^{-1} &\text{ if } \ch_{\fj}(j-1, j) = (0, 1)\\
p-1-r_j &\text{ if } \ch_{\fj}(j-1, j) = (0, 0).
\end{dcases}
\]
Otherwise, let $f(\fj,j)$ be $1$.

Then we have to prove that
\[
\prod_{\fj \in \{J, \ldots, \delta^{k-1} J\}} f(\fj,j) = 1
\]
To do so, it suffices to prove that the number of~$\fj$ such that $\br(j, \fj) = (1, 1, 0), \text{ resp. } (0, 0, 1)$ is the same as the number of~$\fj$ such that $\br(j, \fj) = (0, 1, 1), \text{ resp. } (1, 0, 0)$. For this, define 
\[
\mR_j = \{ \fj \in \{J, \ldots, \delta^{k-1} J\}:  j \in \delta^{-1} \fj \Delta \delta \fj \}.
\] 

It is enough to construct an involution $\dagger: \mR_j \to \mR_j$ such that $\br(\fj,j) \mapsto \br(\dagger \fj,j)$ exchanges $(1, 1, 0)$ with $(0, 1, 1)$ and $(0, 0, 1)$ with~$(1, 0, 0)$. Notice that if $\br(j, \fj) = (1, 1, 0)$ and $\fj = \delta^t J$ then $\br(j+t, J) = (1, 1, 0)$, so there exists some positive integer~$r$ such that $\br(j + t + r, J) = (0, 1, 1)$. Hence it makes sense to define $\dagger \fj$ as the first element of~$\mR_j$ following~$\fj$ such that $\br(j, \dagger \fj) = (0, 1, 1)$. The other~$\dagger \fj$ are defined similarly, as in the proof of Lemma~\ref{differenceodd}, and the proposition follows. 
\end{proof}

\subsection{Computation of the constant.}
Here we put together all of the above and prove the following theorem, which by~\cite[Th\'eor\`eme~6.4]{Breuilparameters} implies Theorem~\ref{thm:functor}, our second main result. 

\begin{thm}\label{completecomputation}
Let~$\rhobar: G_L \to \GL_2(k_E)$ be a semisimple Galois representation as in \S \ref{Galoisreps}. Let~$\chi_0, \ldots, \chi_{k-1}$ be a cycle of characters on~$D_0(\rhobar)$ satisfying the conditions in \S \ref{sec:prodUp}. Assume that $\det(\rhobar) \circ\mathrm{Art}_L(p) = 1$. Then we have
\begin{displaymath}
g_{\chi_{k-1}}\cdots g_{\chi_0} = \begin{dcases}
(-1)^{\half k + \frac{kh}{2f}(1+\sum_{j=0}^{f-1} r_j)} \text{when $\rhobar$ is irreducible}\\
(-1)^{\frac{kh}{f}\sum_{j=0}^{f-1}r_j} \alpha^{\left ( |J|-|J^c| \right )k/f} \text{ when~$\rhobar$ is reducible}.
\end{dcases}
\end{displaymath}
\end{thm}
\begin{proof}
By Lemma~\ref{inverseconstants} and the discussion after Proposition~\ref{nextalpha}, the inverse of the constant~$c(\chi_0, \ldots, \chi_{k-1})$ is the product of $\prod_{i=0}^{k-1}\tld{c}(\psi_i^s)$ with the leading term of~$\beta = \beta(\psi_1^s, \ldots, \psi_k^s)$ (the valuation of~$\beta$ doesn't intervene because we are working with saturated inclusions).  Combining Propositions~\ref{coeffirreducible}, \ref{prodirreducible}, \ref{coeffreducible}, \ref{prodreducible}, this product is equal to
\begin{gather*}
(-1)^{\half k |J \triangle \delta J| + \frac{kh}{2f}(1+\sum_{j=0}^{f-1} r_j)} \text{ in the irreducible case}\\
(-1)^{\half k |J \triangle \delta J| + \frac{kh}{f}\sum_{j=0}^{f-1}r_j} \text{ in the reducible case}.
\end{gather*}
By Proposition~\ref{productvalue}, the factor~$(-1)^{\half k |J \triangle \delta J|}$ cancels in both cases, and there is an extra factor of
\begin{gather*}
(-\alpha)^{-\half k} \text{ in the irreducible case}\\
\alpha^{-k|J^c|/f}(\alpha')^{-k|J|/f} \text{ in the reducible case}.
\end{gather*}
Recall that by our conventions the determinant of~$\rhobar$ evaluated at~$\mathrm{Art}_L(p)$ is $\alpha$ in the irreducible case and $\alpha\alpha'$ in the reducible case. 
Since we are assuming $\det (\rhobar)\circ\mathrm{Art}_L(p) = 1$, this means that the factor contributes $(-1)^{\half k}$, respectively  $\alpha^{(|J|-|J^c|)k/f}$, which is in accordance with~\cite[Th\'eor\`eme~6.4]{Breuilparameters}, since we defined $r_i$ to be $\mu_{1, i} - \mu_{2, i}-1$ (see \S \ref{sec:serrewts}) and $M_\infty[\fm]$ to be $\pi(\rhobar^\vee)$ rather than $\pi(\rhobar)$ (see Definition \ref{defn:piglob} and Remark \ref{rk:rbar}).
\end{proof}

\input{global}

\appendix
\section{Explicit formulas for weight cycling.} \label{sec:app}
Although the computations in~\S \ref{sec:irred}, \S\ref{sec:red} are ultimately based only on the definitions in~(\ref{formalweights}), it might be helpful to have explicit formulas for the behaviour of $\delta \in \{\delta_\red, \delta_\irr\}$ on formal weights. We collect this information in the following tables.

\begin{equation}\label{table1}
\begin{array}{c|c|c|c}
\br(j, J) = \ch_J(j-1, j, j+1) \in J & \ch_{J \triangle \delta J}(j, j-1) & \text{weight~$\lambda_j$} & \text{weight~$\delta(\lambda)_j$} \\
\hline
(1, 1, 1) & (0, 0) & p-3-x_j & p-3- x_j \\
(1, 1, 0) & (1, 0) & p-3-x_j & x_j + 1 \\
(1, 0, 1) & (1, 1) &x _j + 1 & p-2-x_j\\
(1, 0, 0) & (0, 1) & x_j + 1 & x_j\\
(0, 1, 1) & (0, 1) & p-2-x_j & p-3-x_j\\
(0, 1, 0) & (1, 1) & p-2-x_j & x_j + 1 \\
(0, 0, 1) & (1, 0) & x_j & p-2-x_j \\
(0, 0, 0) & (0, 0) & x_j & x_j \\
\end{array}
\end{equation}

Here, we assume $0 < j < f-1$ and $\delta \in \{\delta_\red, \delta_\irr\}$. The remaining cases are treated in what follows.

\begin{equation}\label{table2}
\begin{array}{c|c|c|c}
\br(0, J) = \ch_J(f-1, 0, 1) & \ch_{J \triangle \delta_\irr J}(0, f-1) & \text{weight~$\lambda_0$} & \text{weight~$\delta_\irr(\lambda)_0$}\\
\hline
(1, 1, 1) & (0, 1) & p-1-x_0 & p-2- x_0 \\
(1, 1, 0) & (1, 1) & p-1-x_0& x_0 \\
(1, 0, 1) & (1, 0) & x_0 - 1 & p-1-x_0\\
(1, 0, 0) & (0, 0) & x_0 - 1 & x_0 - 1 \\
(0, 1, 1) & (0, 0) & p-2-x_0 & p-2-x_0\\
(0, 1, 0) & (1, 0) & p-2-x_0 & x_0 \\
(0, 0, 1) & (1, 1) & x_0 & p-1-x_0 \\
(0, 0, 0) & (0, 1) & x_0 & x_0 - 1 \\
\end{array}
\end{equation}

\begin{equation}\label{table3}
\begin{array}{c|c|c|c}
\br(f-1, J) = \ch_J(f-2, f-1, 0) & \ch_{J \triangle \delta_\irr J}(f-1, f-2) & \text{weight~$\lambda_{f-1}$} & \text{weight~$\delta_\irr(\lambda)_{f-1}$} \\
\hline
(1, 1, 1) & (1, 0) & p-3-x_{f-1} & x_{f-1}+1 \\
(1, 1, 0) & (0, 0) & p-3-x_{f-1} & p-3-x_{f-1}\\
(1, 0, 1) & (0, 1) & x_{f-1} + 1 & x_{f-1} \\
(1, 0, 0) & (1, 1) & x_{f-1} + 1 & p-2-x_{f-1}\\
(0, 1, 1) & (1, 1) & p-2-x_{f-1} & x_{f-1}+1 \\
(0, 1, 0) & (0, 1) & p-2-x_{f-1} & p-3-x_{f-1}\\
(0, 0, 1) & (0, 0) & x_{f-1} & x_{f-1} \\
(0, 0, 0) & (1, 0) & x_{f-1} & p-2-x_{f-1} \\
\end{array}
\end{equation}

\bibliographystyle{amsalpha}
\bibliography{ref}

\end{document}

%% file: introduction.tex
\section{Introduction.}

Let $p$ be a rational prime, $f$ be a positive integer, and let $L$ be the unramified extension of $\bQ_p$ of degree $f$.
Let $k_E$ be a finite extension of the residue field $k_L$ of $L$.
For a field $k$, we implicitly choose a separable closure $k^{\mathrm{sep}}/k$ and set $G_k = \mathop{Gal}(k^{\mathrm{sep}}/k)$.
A perhaps optimistic hope is that a mod $p$ local Langlands correspondence attaches to every continuous Galois representation $\rhobar: G_L \ra \GL_2(k_E)$ a $\GL_2(L)$-representation $\pi(\rhobar)$ over $k_E$ satisfying several properties including compatibility with the classical local Langlands under reduction and local-global compatibility (more on this later).
Such a correspondence has been established for $\GL_2(\bQ_p)$ in \cite{BreuilmpLLC,Colmez,Paskunas,Emerton}, but appears to be significantly more complex when $f>1$ or $\GL_2$ is replaced by a higher rank group.
One of the obstacles in realizing, or even formulating, this correspondence is that supersingular irreducible $k_E[\GL_2(L)]$-representations are not classified, and seem to be very difficult to even construct.

\paragraph{} The classical way of studying representations of $p$-adic (or Lie) groups is first to study their restriction to a compact subgroup, and then to study residual symmetries, for example the algebra of Hecke operators.
In this way, infinite dimensional representations are replaced with combinatorial objects.
The module theory of the pro-$p$ Iwahori Hecke algebra captures much of the mod $p$ representation theory in the case of $\GL_2(\bQ_p)$, but less in any more complicated case.
In \cite[\S 9]{BPmodp}, Breuil and Pa{\v s}k{\=u}nas instead consider a refinement of Hecke modules, which they call (basic) $0$-diagrams.
For our purposes, a $0$-diagram $\mathcal{D}$ is a $\GL_2(k_L)$-representation $D_0$ with an automorphism $\Pi$ of $D_0^{I_1}$ whose square acts by a nonzero scalar, where $I_1 \subset \GL_2(k_L)$ is the subgroup of unipotent upper triangular matrices.
From a $\GL_2(L)$-representation $\pi$ with central character, one obtains a $0$-diagram $\mathcal{D}(\pi)$ by taking $D_0$ to be $\pi^{K_1}$ and $\Pi$ to be given by the action of $\fourmatrix{0}{1}{p}{0}$, where $K_1$ is the kernel of the reduction map $\GL_2(\mO_L) \ra \GL_2(k_L)$.
Note that if $\pi$ is nonzero and of characteristic $p$, then so is $\pi^{K_1}$ since $K_1$ is a pro-$p$ group.
In \cite[\S 13]{BPmodp}, they construct a family of mod $p$ $0$-diagrams for each \emph{generic} $\rhobar$.
From each of these $0$-diagrams $\mathcal{D}$, \cite[\S 9]{BPmodp} construct a family of mod $p$ $\GL_2(L)$-representations $\pi$ such that $\mathcal{D}$ is a subdiagram of $\mathcal{D}(\pi)$ in the obvious sense.
For each generic $\rhobar$, the family of $0$-diagrams has size $1$ if $f=1$, but is infinite otherwise; moreover, for each basic $0$-diagram, the family of mod $p$ $\GL_2(L)$-representations has size $1$ if $f=1$, but is infinite if $f>1$ (see~\cite{Hu10})---this is a long way from the hoped-for correspondence when $f>1$.

\paragraph{} Local-global compatibility gives another approach to the mod $p$ local Langlands correspondence.
The idea is that a global mod $p$ Langlands correspondence (closely related to generalizations of Serre's conjecture) is often easy to formulate, if difficult to prove.
Namely, as in the classical picture, it suffices to use the Satake parameterization to define an unramified mod $p$ local Langlands correspondence away from $p$.
Mod $p$ local-global compatibility at $p$ states that the restriction of the global correspondence to places dividing $p$ should recover the local correspondence.
So one way to proceed is to globalize $\rhobar$ to a Galois representation $\rbar$ which is known to correspond to a space of mod $p$ automorphic forms on a quaternion algebra, and then to restrict this representation to get a candidate $\GL_2(L)$-representation $\pi_{\mathrm{glob}}(\rhobar)$ (see Definition \ref{defn:piglob} and \S \ref{sec:modpaut}).
A drawback of this construction is that it is completely unclear if $\pi_{\mathrm{glob}}(\rhobar)$ is independent of the various global choices.
The representation $\pi_{\mathrm{glob}}(\rhobar)$ is best understood when $\rbar$ satisfies favorable conditions like the Taylor--Wiles hypothesis.
We will assume this to be the case in the rest of this section (see \S \ref{sec:modpaut} for more details).

\begin{thm}\cite{Breuilmodp,EGS,HW,LMS,Le}.\label{thm:infamily}
If $\rhobar$ is generic, $\mathcal{D}(\pi_{\mathrm{glob}}(\rhobar))$ is isomorphic to one of the $0$-diagrams attached to $\rhobar$ by \cite{BPmodp}.
\end{thm}

More precisely, \cite{Breuilmodp} shows that $\mathcal{D}(\pi_{\mathrm{glob}}(\rhobar))$ contains one of the $0$-diagrams attached to $\rhobar$ conditional on a conjecture later established in \cite{EGS} (under a slightly stronger genericity hypothesis). 
Building on this, the sequence of works \cite{HW,LMS,Le} shows that in fact this inclusion is an isomorphism (subject to still stronger genericity hypotheses).
However, since this family of $0$-diagrams is infinite when $f>1$, one may still ask whether $\mathcal{D}(\pi_{\mathrm{glob}}(\rhobar))$ depends on global choices.
The following is our first main result (combining Theorems \ref{patchingdiagram} and \ref{thm:patching}).

\begin{thm}\label{thm:local}
If $\rhobar$ is generic, $\mathcal{D}(\pi_{\mathrm{glob}}(\rhobar))$ depends only on $\rhobar$.
\end{thm}

\noindent See \S \ref{Galoisreps} for the precise genericity conditions we need: they are slightly more restrictive than those in~\cite{BPmodp}. 
Theorem~\ref{thm:local} narrows the family of diagrams constructed by \cite{BPmodp} to a single one for each generic $\rhobar$.
In fact, our proof can be made completely explicit, which we do to some extent in the case when $\rhobar$ is semisimple (see Theorem \ref{thm:functor}).

\paragraph{} The mod $p$ (and $p$-adic) local Langlands correspondence for $\GL_2(\bQ_p)$ (see \cite{Colmez,Kisindef,Paskunas,CDP}) is realized by Colmez's functor.
In \cite{Breuilparameters}, Breuil observes that this functor is rather combinatorial if $\rhobar$ is semisimple and can be generalized to the $\GL_2(L)$-case.
Namely, to a \emph{Diamond} $0$-diagram $\mathcal{D}$ (a $0$-diagram of a certain form, see \S \ref{sec:diamond}) he attaches a $(\varphi,\Gamma)$-module $M(\mathcal{D})$, which recovers Colmez's functor in the $\GL_2(\bQ_p)$ case when $\rhobar$ is semisimple (recall that the family of $0$-diagrams in this case is a singleton).
In other words, if $f$ is $1$ and $\rhobar$ is semisimple, then $M(\mathcal{D}(\rhobar))$ corresponds to $\rhobar^\vee$ under Fontaine's equivalence (see \cite[\S 1]{Breuilparameters}).
Our other main result generalizes this to arbitrary $f$.

\begin{thm}\label{thm:functor}
If $\rhobar:G_L \ra \GL_2(k_E)$ is generic and semisimple, then $M(\mathcal{D}(\pi_{\mathrm{glob}}(\rhobar)))$ corresponds under Fontaine's equivalence to the tensor induction $\ind_L^{\otimes\bQ_p} \rhobar^\vee$ from $G_L$ to $G_{\bQ_p}$.
\end{thm}

By~\cite[Th\'eor\`eme~6.4]{Breuilparameters}, this theorem is equivalent to a determination of certain parameters of the diagram, which appears in Theorem~\ref{completecomputation} in the case $\det(\rhobar) \circ \mathrm{Art}_L(p) = 1$ (the other cases follow by twisting). This gives a positive answer to~\cite[Question~9.5]{Breuilmodp} when~$\rhobar$ is generic. 

\begin{rk}
Hu~\cite{Hu} obtained an analogue of this question in the setting of~\cite{BD} where $\rhobar$ is maximally nonsplit.
We repeat certain calculations from~\cite{Hu} in slightly different contexts for the sake of completeness.
(See Remark~\ref{Hureference} for an instance of this.)
\end{rk}

\paragraph{}
We now explain these parameters further. 
By Theorem \ref{thm:infamily}, the $\GL_2(k_L)$-action on $\mathcal{D}(\pi_{\mathrm{glob}}(\rhobar))$ is known (by construction, the $\GL_2(k_L)$-actions within a single family are isomorphic).
To prove Theorems \ref{thm:local} and \ref{thm:functor}, we must understand how the actions of $\GL_2(k_L)$ and $\Pi$ interact.
We show that this interaction corresponds exactly to a character of a certain finite groupoid $\mathcal{G}$ with generators indexed by $I$-isotypic lines in $D_0^{I_1}$.
It suffices to show that the action of group elements of $\mathcal{G}$ depend only on $\rhobar$.
There has been a body of work studying the action of group elements which are themselves \emph{generators} (see \cite{BD,Hu,HLM,LMP,Enns,PQ}), often with the goal of showing that $\rhobar$ can be recovered from $\pi_{\mathrm{glob}}(\rhobar)$.
The methods in these works have a common theme.
The study of $\Pi$ can be reduced to the study of the action of a Hecke operator $U_p$ and a certain element of the group algebra of $\GL_2(k_L)$.
The action of $U_p$ can then be studied by using the classical local Langlands correspondence on a characteristic zero lift of $\pi_{\mathrm{glob}}(\rhobar)$ and then taking the reduction mod $p$.
This method cannot work directly because as soon as one considers a product of generators, the lifts one considers are typically distinct and not directly comparable.
And so we need another tool.

\paragraph{} The reason for imposing the Taylor--Wiles hypothesis is that we would like to use the Taylor--Wiles patching method.
In fact, our results apply whenever $\pi_{\mathrm{glob}}(\rhobar)$ satisfies certain axioms, namely it comes from a \emph{patched module with an arithmetic action}.
The study of Colmez's functor led to a proof of many cases of the Fontaine--Mazur conjecture in \cite{KisinFM}, through the construction of restricted local deformation rings at $p$ (see \cite{Kisinsemistable}), a classical inertial local Langlands correspondence relating representations of compact open subgroups and representations of the inertia subgroups of local Galois groups, and the Taylor--Wiles patching method.
In the other direction, the techniques introduced by Kisin have led to a number of developments in the inertial mod $p$ and $p$-adic Langlands correspondences beyond $\GL_2(\bQ_p)$ (see, e.g., \cite{GLS,EGS,LLLM,LLLM2,LLL}).
Another approach to the Fontaine--Mazur conjecture is to prove local-global compatibility (see \cite{Emerton,CEGGPS2}), where
Hecke operators must play a crucial role.
Our approach is similar to \cite{CEGGPS2} in that we use the interpolation from \cite{CEGGPS} of the action of Hecke operators on Taylor--Wiles patched modules, but we further interpolate \emph{between} different deformation rings.

\paragraph{} While the various characteristic zero lifts of spaces of mod $p$ automorphic forms one might consider are incomparable, they can be interpolated using the Taylor--Wiles patching method.
The corresponding Hecke operators can then be identified with elements of various tamely potentially Barsotti--Tate deformation rings.
In turn, these elements can be compared with elements of a single tamely potentially Barsotti--Tate deformation ring, corresponding to what we call the central inertial type.
For this comparison, it is essential that we are working with patched modules which are maximal Cohen--Macaulay on deformation spaces.
In this way, understanding the action of group elements of $\mathcal{G}$ is reduced to understanding a certain product of elements of a tamely potentially Barsotti--Tate deformation ring and the action of (a product of) elements of the group algebra of $\GL_2(k_L)$ on the Deligne--Lusztig representation corresponding to the central inertial type by inertial local Langlands.
For the reader only interested in the proof of Theorem \ref{thm:local} in an axiomatic context (see Theorem \ref{patchingdiagram}), a proof with minimal reference to previous sections is given in \S \ref{sec:diagrams}.
To prove Theorem \ref{thm:functor}, we explicitly calculate these products in \S \ref{sec:constants} using \S \ref{repprelim}-\ref{sec:defring}.
Using the Taylor--Wiles method, \S \ref{sec:global} gives examples of global contexts satisfying the axioms of \S \ref{sec:diagrams}.

\subsection{Acknowledgments.}

The authors would like to thank A.~Caraiani, M.~Emerton, T.~Gee, D.~Geraghty, V.~Pa{\v s}k{\=u}nas, and S.W.~Shin for sharing a note which contains an inspiring reinterpretation of \cite{BD} using Taylor--Wiles patching.
DL would like to thank C.~Breuil for proposing \cite[Question~9.5]{Breuilmodp} and encouraging him to work on it, as well as M.~Emerton for related conversations.
AD joins DL in thanking Breuil and Emerton for their interest and encouragement, and is grateful to T.~Gee for mentioning the problem to him.
The authors thank the referee for a close reading of and comments and corrections on earlier versions of this article and V.~Pa{\v s}k{\=u}nas for pointing out an error in an earlier version.

The authors would like to thank L.~Berger, G.~Dospinescu, P.~Gille, W.~Niziol, V.~Pilloni, S.~Rozensztajn, A.~Thuillier, l'Universit\'e de Lyon, and l'ENS de Lyon for organizing and hosting the trimester on Algebraic Groups and Geometrization of the Langlands program where this collaboration began.
AD was supported by the Engineering and Physical Sciences Research Council [EP/L015234/1], The EPSRC Centre for Doctoral Training in Geometry and Number Theory (The London School of Geometry and Number Theory), University College London, and Imperial College London.
DL was supported by the National Science Foundation under agreement No.~DMS-1703182 and an AMS-Simons travel grant. 

\subsection{Notation and conventions.}
Throughout, $p > 5$ is a prime number (if $p \leq 5$ then there are no generic representations in the sense of \S \ref{Galoisreps}). Let~$L = \bQ_{p^f}$, write $\mO_L$ for the ring of integers, $k_L$ for the residue field, and $q = p^f$, and fix an algebraic closure~$\lbar L / L$.
We fix a coefficient field $E/ \bQ_p$ (a finite extension) with ring of integers~$\mO_E$, uniformizer $\varpi_E$, and residue field~$k_E$. 
By taking a further extension, we will assume that $E$ is large enough for our purposes.
In particular, we assume that every irreducible $\overline{E}$-representation of $\GL_2(k_L)$ is defined over~$E$. 
We assume there exist embeddings of $k_L$ in~$k_E$ and we fix one of them, denoted $\sigma_0 : k_L \to k_E$. We will number embeddings as $\sigma_j = \sigma_0 \circ \varphi^{-j}$, where~$\varphi: x \mapsto x^p$. 

For $\lambda \in k_L$ we write~$[\lambda] \in \mO_E$ for the Teichm\"uller lift arising from the induced embedding $\sigma_0: W(k_L) \to \mO_E$. 

The group~$\GL_2(\mO_L)$ will be denoted~$K$, and the first congruence subgroup is
\[
K_1 = \ker(\GL_2(\mO_L) \to \GL_2(k_L)).
\]  
According to the context, we write~$H$ for the diagonally embedded $\mO_L^\times \times\mO_L^\times$ in~$\GL_2(\mO_L)$ or its image in~$\GL_2(k_L)$, which is the diagonal torus. Similarly, $I$ denotes the Iwahori subgroup of $\GL_2(\mO_L)$ or the upper triangular Borel subgroup of~$\GL_2(k_L)$, and $I_1$ denotes the Sylow pro-$p$ subgroup of~$I$ or the upper unipotent subgroup of~$\GL_2(k_L)$. 
We write $\alpha$ for the $H$-character $\fourmatrix{a}{0}{0}{d} \mapsto [ad^{-1}]$. If~$\chi$ is a character of~$H$, then $\chi^s$ is its conjugate under the nontrivial element $s$ of the Weyl group $S_2$. 
Hence, if $\chi: \fourmatrix a 0 0 d \mapsto [a]^r[ad]^s$, then $\chi^s = \chi \alpha^{-r}$. If~$x \in k_E^\times$, we write~$\nr_x$ for the unramified character $G_{\bQ_p} \to k_E^\times$ that sends geometric Frobenius elements to~$x$. 

The determinant character of~$\GL_2(k_L)$ is naturally valued in~$k_L^\times$, and we obtain a $k_E^\times$-valued character through our fixed embedding~$\sigma_0: k_L \to k_E$. We will usually denote $\sigma_0 \circ \det$ by~$\det$, and write~$\det_E$ if we need to distinguish between the two.

We define two maps, $\delta_{\mathrm{red}}$ and~$\delta_{\mathrm{irr}}$, from the set of subsets of~$\{0, \ldots, f-1 \}$ to itself by
\begin{gather*}
j \in \delta_{\mathrm{red}}(J) \text{ if and only if } j+1 \in J\\
\begin{cases}
j<f-1 : j \in \delta_{\mathrm{irr}}(J) \text{ if and only if } j+1 \in J\\
f-1 \in \delta_{\mathrm{irr}}(J) \text{ if and only if } 0 \not \in J
\end{cases}
\end{gather*}
so that~$\delta_{\red}$ is a left shift. The following lemma follows from the definitions.

\begin{lemma}\label{deltabehaviour}
For all $\delta \in \{\delta_\red, \delta_\irr\}$, we have equalities $\delta(J^c) = (\delta J)^c$, $J \triangle \delta J = J^c \triangle \delta J^c$ and $\delta_\red(J \triangle \delta J) = \delta J \triangle \delta^2 J$.
\end{lemma}

For a subset~$J$  of~$\bZ/f = \{0, \ldots, f-1\}$, we will write~$\ch_J$ for the characteristic function of~$J$.
We will sometimes employ the following shorthand in conjunction with~$\delta_\red$: if $J \subset \{0, \ldots, f-1\}$, we will write $\br(j, J) = \ch_J(j-1, j, j+1)$.

If~$x$ is a positive integer, we define numbers~$x_j \in [0, p-1]$ by writing $x = \sum_{j=0}^{f-1} x_ip^i + Q(q-1)$ for some~$Q > 0$.
(The only possible ambiguity occurs when $q-1|x$, in which case we put $x_j=0$ for all~$j$.)

If $F$ is a local or global field, we write $\mathrm{Art}_F$ for the Artin reciprocity map, normalized so that the global map is compatible with the local ones, and the local map takes a uniformizer to a geometric Frobenius element.
If $F$ is a number field and $w$ is a finite place of $F$, we will write $\Frob_w$ for a geometric Frobenius element at $w$.
We let $\varepsilon: G_L \ra \bZ_p^\times \ra \mO_E^\times$ denote the cyclotomic character and normalize Hodge--Tate weights so that $\mathrm{HT}_\kappa(\varepsilon) = 1$ for all embeddings $\kappa: L \ra E$.

%% file: kisinwithquestions.tex
\section{Galois deformation rings and Frobenius eigenvalues.}\label{sec:defring}

\subsection{Kisin modules.}\label{sec:kisin}

\paragraph{Notation.}
Recall that $L$ is $\bQ_{p^f}$. 
Let $e$ be $p^f-1$ and $\pi$ be a root of $E(u) = u^e+p$ in our fixed algebraic closure of~$L$.
Denote by $K$ the extension $L(\pi)$ and by $\Delta$ the Galois group $\Gal(K/L)$.
Then the character $[\omega_f]$ factors through $\Delta$. 

Let $\ba \in X(\bT)$ correspond to the tuple of integers $(a_{k,j})_j$ for $1\leq k\leq n$ and $0\leq j\leq f-1$.
Let
\[
\ba_k = \sum_{i=0}^{f-1} a_{k,i}p^i
\]
and 
\[
\ba_k^{(j)} = \sum_{i=0}^{f-1} a_{k,i-j}p^i,
\]
so that $\ba_k^{(j)} \equiv p^j\ba_k \pmod{q-1}$.
Write $\eta_k$ for the character $[\omega_f]^{-\ba_k}$.
Finally, let $\tau$ be the principal series tame inertial type $\tau(1,\ba)$ for $W_L$.
In other words, the $\mO_E$-dual $\tau^\rd$ is isomorphic to $\oplus_{k=1}^n \eta_k|_{I_L}$.

\paragraph{}

If~$R$ is an $\mO_E$-algebra, we give $(\mO_L \otimes_{\bZ_p} R)[\![u]\!]$ an action of~$\Delta$ by $\mO_L \otimes_{\bZ_p} R$-linear ring automorphisms, by $\widehat{g}: u \mapsto ([\omega_{f,L}](g) \otimes 1)u$ (where $[\omega_{f,L}]$ is the $\mO_L$-valued character such that $[\omega_f] = \sigma_0 \circ [\omega_{f,L}]$). 
There is also a Frobenius homomorphism $\varphi: (\mO_L \otimes_{\bZ_p} R)[\![u]\!] \to (\mO_L \otimes_{\bZ_p} R)[\![u]\!]$ which is $\varphi \otimes 1$ on $\mO_L \otimes_{\bZ_p} R$ and sends~$u$ to~$u^p$.

\begin{defn}\label{def:kisin}
Let $R$ be an $\mO_E$-algebra.
A Kisin module over $R$ is a finitely generated projective $(\mO_L \otimes_{\bZ_p} R)[\![u]\!]$-module $\fM_R$ together with a Frobenius map $\phi_{\fM_R} : \varphi^*(\fM_R) \ra \fM_R$ whose cokernel is annihilated by some power of $E(u)$.
A Kisin module over $R$ with descent datum is a Kisin module over $R$ with a semilinear action of $\Delta$ which commutes with $\phi_{\fM_R}$.
We say that the descent datum is of type $\tau$ if $\fM_R/u$ is isomorphic to $(\mO_L \otimes_{\bZ_p} R) \otimes_{\mO_E} \tau^\rd   \cong (\tau^\rd\otimes_{\mO_E} R)^f$.
\end{defn}

\begin{rk}
The dual in Definition \ref{def:kisin} appears because we will use contravariant functors to Galois representations in \S \ref{sec:phi}.
\end{rk}

In what follows we will work with Kisin modules whose $(\mO_L \otimes_{\bZ_p} R)[\![u]\!]$-rank is equal to~$n$, though in applications $n$ will be $2$. 
The decomposition 
\[
\mO_L\otimes_{\bZ_p} \mO_E \cong \prod_{\iota: k_E \hookrightarrow k_E}\mO_E \cong \prod_{j=0}^{f-1} \mO_E
\] 
given by the bijection $j \leftrightarrow \iota_j$ gives a decomposition of a Kisin module $\fM_R$ as $\oplus_{j=0}^{f-1} \fM_R^{(j)}$, where $\fM_R^{(j)}$ is a projective $R[\![u]\!]$-module of rank $n$ for all $j$.

\paragraph{Eigenbases.}
Suppose now that $\tau$ is a principal series tame inertial type and that $\fM_R$ is a Kisin module over $R$ with descent datum $\tau$.
Let $v$ be $u^e$.
For each $1\leq k\leq n$, let $\fM^{(j)}_{R,k}$ be the $R[\![v]\!]$-submodule of $\fM_R^{(j)}$ on which $\Delta$ acts by $\eta_k$, i.e.~
\[
\fM_{R,k}^{(j)} = \{ m\in \fM_R^{(j)}|\widehat{g}(m) = \eta_k(g) m\}.
\]
Then $\fM_{R,k}^{(j)}$ is a projective $R[\![v]\!]$-module of rank $n$.

Suppose that the characters $\eta_k$ are distinct.
An eigenbasis $\beta$ of a Kisin module $\fM_R$ over $R$ with descent datum of type $\tau$ is a tuple $(f_k^{(j)})_{k,j}$ where each $(f^{(j)}_1, \ldots, f^{(j)}_n)$ is an $R[\![u]\!]$-basis of~$\fM_R^{(j)}$ such that $f_k^{(j)} \in \fM_{R,k}^{(j)}$.
Note that this notion depends implicitly on a choice of ordering of the characters of $\tau$.
Let $\beta^{(j)}$ be the tuple $(f_k^{(j)})_k$.

Given an eigenbasis $\beta$ of a Kisin module $\fM_R$ over $R$ with descent datum of type $\tau$, for each $j$ let $C^{(j)}\in \mathrm{Mat}_n(R[\![u]\!])$ be the matrix such that 
\[
\phi_{\fM}^{(j)}(\varphi^*(\beta^{(j)})) = \beta^{(j+1)} C^{(j)}.
\]
In other words, the Frobenius map~$\phi_{\fM}$ induces a semilinear map $\fM_R^{(j)} \to \fM_R^{(j+1)}$, and~$C^{(j)}$ is a matrix over $R[\![u]\!]$ expressing the image of $\beta^{(j)}$ in terms of $\beta^{(j+1)}$.

\subsubsection{Kisin modules and base change}\label{sec:kisinbasechange}

Suppose that $\tau = \oplus_{k=1}^n \eta_k$ is an $n$-dimensional tame inertial type for $W_L$. 
There is a permutation $s_\tau \in S_n$ such that $\eta_k^q = \eta_{s_\tau^{-1}(k)}$ for all $k$.
Let $d$ be the order of $s_\tau$ and let $L'$ be the unramified extension of $L$ of degree $d$.
Then the base change to $L'$ of $\tau$ is a principal series tame inertial type $\tau'$ for $W_{L'}$ (in this context, ``base change'' amounts to regarding~$\tau$ as a representation of~$I_{L} = I_{L'}$).

Let $R$ be an $\mO_E$-algebra.
Let $\sigma:\mO_{L'}[\![u]\!] \ra \mO_{L'}[\![u]\!]$ be the automorphism which acts trivially on $u$ and acts on $\mO_{L'}$ by the action induced by the $q$-th power automorphism of $k_{L'}$.
Then we define a Kisin module over $R$ with descent datum of type $\tau$ to be a pair $(\fM_R, \iota)$ where $\fM_R$ is a Kisin module over $R$ with descent datum of type $\tau'$ and $\iota: \sigma^*(\fM_R) \ra \fM_R$ is an isomorphism such that $\iota \circ \sigma^*(\iota) \circ \cdots \circ (\sigma^{d-1})^*(\iota)$ is the identity morphism.
An eigenbasis of $(\fM_R, \iota)$ is then defined to be an eigenbasis $\beta$ of $\fM_R$ such that $\iota(\sigma^*(\beta))$ is a permutation of the basis $\beta$. Since~$\iota$ can be identified with a $\sigma$-semilinear map $\fM_R \to \fM_R$, this means that~$\iota$ permutes the vectors in~$\beta$.

\subsection{\'Etale $\varphi$-modules and Galois representations.}\label{sec:phi}

\subsubsection{Generalities.}

Recall from \S\ref{sec:kisin} the definitions of $K$ and $L'$.
Let $L_\infty$ be an infinite extension obtained by adjoining to~$L$ compatible $p$-power roots of $-p$.
Let $K'$ be the compositum of the fields $K$ and $L'$ and let $K_\infty'$ be the compositum of $K'$ and $L_\infty$.
Let $\Delta'$ be the Galois group $\Gal(K'/L)$: it is isomorphic to the product of~$\Delta$ with $\Gal(L'/L)\cong \bZ/d$.

\paragraph{Equivalence of categories.} Let $\mO_{\mE,L}$ denote the $p$-adic completion of $\mO_L(\!(v)\!)$, and let $\mO_{\mE^{\mathrm{un}},L}$ denote the $p$-adic completion of a maximal connected \'etale extension of $\mO_{\mE,L}$.
We similarly let $\mO_{\mE,K'}$ denote the $p$-adic completion of $\mO_{L'}(\!(u)\!)$, and let $\mO_{\mE^{\mathrm{un}},K'}$ denote the $p$-adic completion of a maximal connected \'etale extension of $\mO_{\mE,K'}$.
Then $\mO_{\mE, K'}$ is naturally a Galois finite \'etale extension of~$\mO_{\mE, L}$ by putting $v = u^e$.
This identifies $\mO_{\mE^{\mathrm{un}},L}$ with $\mO_{\mE^{\mathrm{un}},K'}$.
Moreover, classical work of Fontaine (see~\cite{fontaine}) identifies $G_{\mO_{\mE, L}}$ and $G_{\mO_{\mE, K'}}$ with $G_{L_\infty}$ and $G_{K'_\infty}$, respectively.
This identifies $\Gal(\mO_{\mE, K'}/\mO_{\mE, L})$ with $\Gal(K'_\infty/L_\infty)$, which is isomorphic to $\Delta'$.
Under this isomorphism, $\Delta'$ acts on $\mO_{\mE, K'}$ by its natural action on $\mO_{L'}$ and by $[\omega_{f',L'}]$ on $u$ (cf.~\S \ref{sec:kisin}). 

For $R$ a complete local Noetherian $\mO_E$-algebra, let $\Phi\textrm{-}\mathrm{Mod}^\et(R)$ be the category of \'etale $\varphi$-modules over $\mO_{\mE,L} \widehat{\otimes}_{\bZ_p} R$, such that the underlying $\mO_{\mE,L} \widehat{\otimes}_{\bZ_p} R$-module is finite free. Similarly, let $\Phi\textrm{-}\mathrm{Mod}^\et_\dd(R)$ be the category of \'etale $\varphi$-modules over $\mO_{\mE,K'}\widehat{\otimes}_{\bZ_p} R$ with a semilinear action of $\Delta'$. 
Let $\mathrm{Rep}_{G_{L_{\infty}}}(R)$ be the category of (continuous) representations of $G_{L_{\infty}}$ over finite free $R$-modules.
\cite{fontaine} gives an exact anti-equivalence of tensor categories
\begin{align*}
\mathbb{V}^*: \Phi\textrm{-}\mathrm{Mod}^\et(R) &\ra \mathrm{Rep}_{G_{L_{\infty}}}(R)\\
\mM &\mapsto ((\mM\otimes_{\mO_{\mE,L}} \mO_{\mE^{\mathrm{un}},L})^{\varphi=1})^*,
\end{align*}
where the action of $G_{L_\infty} \cong G_{\mO_{\mE, L}}$ comes from $\mO_{\mE^{\mathrm{un}},L}$ and $(-)^*$ denotes the contragredient representation.

A slight extension of the above gives an exact anti-equivalence of tensor categories
\begin{align*}
\mathbb{V}^*_\dd: \Phi\textrm{-}\mathrm{Mod}^\et_\dd(R) &\ra \mathrm{Rep}_{G_{L_{\infty}}}(R)\\
\mM &\mapsto ((\mM \otimes_{\mO_{\mE,K'}} \mO_{\mE^{\mathrm{un}},L})^{\varphi = 1})^*,
\end{align*}
where $G_{L_\infty}$ acts on $\mO_{\mE^{\mathrm{un}},L}$ as above and on $\mM$ through the quotient $G_{L_\infty}/G_{K'_\infty} \cong \Delta'$.
Moreover, by \cite[Theorem 2.1.6 and (12)]{CDM}, there is an equivalence of categories 
\begin{align}\label{eqn:desc}
\Phi\textrm{-}\mathrm{Mod}^\et(R) &\ra \Phi\textrm{-}\mathrm{Mod}^\et_\dd(R) \\
\mM &\mapsto \mM \otimes_{\mO_{\mE,L}} \mO_{\mE,K'}, \nonumber
\end{align}
which is compatible with Fontaine's functors above.
A quasi-inverse of (\ref{eqn:desc}) is given by taking invariants under $\Delta'$.

\paragraph{Matrices of the action.} We can write $\mO_{\mE, L} \cong \mO_{\mE} \otimes_{\bZ_p} \mO_L$, where we write $\mO_{\mE}$ for $\mO_{\mE,\bQ_p}$. As with Kisin modules, the decomposition $\mO_L\otimes_{\bZ_p} \mO_{\mE}$ as $\prod_{j=0}^{f-1} \mO_{\mE}$ (numbering embeddings via $j \leftrightarrow \sigma_j$) gives a decomposition of an element $\mM_R\in \Phi\textrm{-}\mathrm{Mod}^\et(R)$ as $\oplus_{j=0}^{f-1} \mM_R^{(j)}$, where each summand is a finite free $\mO_{\mE} \widehat{\otimes}_{\bZ_p} R \cong \mO_{\mE, L} \widehat \otimes_{\mO_L} R$-module of the same rank as that of $\mM_R$ over~$\mO_{\mE, L} \widehat{\otimes}_{\bZ_p} R$.
If $\gamma^{(j)}$ is an $\mO_{\mE} \widehat{\otimes}_{\bZ_p} R$-basis of $\mM_R^{(j)}$ for all $j$, we let $B^{(j)}$ be the matrix such that 
\[
\varphi_{\mM}^{(j)}(\varphi^*(\gamma^{(j)})) = \gamma^{(j+1)} B^{(j)}.
\]

\subsubsection{Galois representations.}~\label{Galoisreps}
In this section, we state our assuptions on mod~$p$ Galois representations and compute the $\varphi$-module of their restrictions to~$G_{L_\infty}$. We specialize the discussion to~$n = 2$.

\paragraph{Genericity.} Let $\rhobar: G_L \ra \GL_2(k_E)$ be a continuous representation, and write $V(\rhobar^\semis|_{I_L}) \cong R_{s_\rhobar}(\mu)$, with $s_\rhobar = (s_\rhobar,1,\ldots,1)$ such that~$s_{\rhobar}$ is nontrivial if and only if~$\rhobar$ is irreducible.
Write~$\mu$ as $(\mu_{1,j},\mu_{2,j})_j$.
We will assume that $\rhobar$ is \emph{generic}, which will mean that 
\[
2 < \langle \alpha^{(j)},\mu \rangle < p-3
\]
for all $j\in \bZ/f$. (Here $\alpha^{(j)}$ is the positive root in embedding~$j$, so that $\langle \alpha^{(j)},\mu \rangle = \mu_{1, j} - \mu_{2, j}$).
Note that this inequality can hold only if $p>5$.

With these assumptions, if $\rhobar$ is reducible, then $\rhobar$ is isomorphic to an extension of 
\[
\nr_{\alpha'} \omega_f^{\sum_{j=0}^{f-1} \mu_{2,j}p^j} \textrm{ by } \nr_\alpha \omega_f^{\sum_{j=0}^{f-1} \mu_{1,j}p^j}
\]
for some $\alpha$ and $\alpha' \in k_E^\times$.
Otherwise, $\rhobar$ is isomorphic to an induction
\[
\Ind_{G_{L_2}}^{G_L} \left ( \nr_\beta \omega_{2f}^{\sum_{j=0}^{f-1} \mu_{1,j}p^j+p^f\sum_{j=0}^{f-1} \mu_{2,j}p^j} \right )
\]
for some $\beta \in k_E^\times$, where $L_2$ denotes the unramified quadratic extension of $L$. In this case, we denote by~$\alpha$ the value at~$p$ of the determinant of~$\rhobar$, so that $\alpha = - \beta$, and we set~$\alpha' = -1$. The reason for doing so will become apparent when computing with $\varphi$-modules.

When working with modules over~$k_L(\!(v)\!)$, we will identify $s_{\rhobar, j}$ with a permutation matrix and write $v^{\mu_j}$ for the image of~$v$ under the cocharacter $\mu_j : \mathbb{G}_m \to \GL_2$, so that $v^{\mu_j} = \Diag(v^{\mu_{1,j}}, v^{\mu_{2,j}})$.

\begin{pp}\label{pp:rhophi}
There is an $\mM_{k_E} \in \Phi\textrm{-}\mathrm{Mod}^\et(k_E)$ with $\mathbb{V}^*(\mM_{k_E}) \cong \rhobar|_{G_{L_\infty}}$ and a basis $(\gamma^{(j)})_j$ of $\mM_{k_E}$ such that 
\begin{equation} \label{eqn:B}
B^{(j)} = U^{(j)} s^{-1}_{\rhobar,f-1-j} v^{\mu_{f-1-j}}
\end{equation}
for some lower triangular matrix $U^{(j)}\in \GL_2(k_E(\!(v)\!))$. (Recall the isomorphism $\mO_{\mE, L} \otimes_{\mO_L} k_E \cong k_E(\!(v)\!)$).
Furthermore, $U^{(j)}$ is unipotent (resp.~the product of $\begin{pmatrix}
  \alpha &   \\
   & \alpha' \\
 \end{pmatrix}$ and a unipotent matrix) if $j\neq 0$ (resp.~if $j = 0$), where $\alpha' = -1$ if $\rhobar$ is irreducible.
\end{pp}
\begin{proof}
We use Fontaine--Laffaille theory as in \cite[Appendix A]{Breuilmodp}.
Note that $\rhobar$ is Fontaine--Laffaille by the genericity condition.
Suppose that $\rhobar$ is reducible.
By twisting, we assume that $\mu_{2, j} = 0$ for all $j$.
Then there is a Fontaine--Laffaille module $M = \oplus_{j=0}^{f-1} M^{(j)}$ with $M^{(j)} = k_E e^{(j)} \oplus k_E f^{(j)}$ such that 
\[
\Fil^0 M^{(j)} = M^{(j)}, \, \Fil^1 M^{(j)} = \Fil^{\mu_{1, f-j}} M^{(j)} = k_E f^{(j)}, \, \Fil^{\mu_{1, f-j}+1} M^{(j)} = 0, \\
\]
\begin{align*}
\varphi(e^{(j)}) &=  e^{(j+1)}, \\
\varphi_{\mu_{1, f-j}}(f^{(j)}) &= f^{(j+1)}+x_{j-1} e^{(j+1)}, \, \textrm{for  } j\neq 1 \textrm{ and } \\
\varphi(e^{(1)}) &=  \alpha' e^{(2)}, \\
\varphi_{\mu_{1, f-1}}(f^{(1)}) &= \alpha f^{(2)}+\alpha' x_0 e^{(2)},
\end{align*}
for some $x_j \in k_E$ and such that $\rhobar \cong \Hom_{\Fil^.,\varphi_.}(M,A_{\mathrm{cris}} \otimes_{\bZ_p} \bF_p)$ (see e.g.~\cite[(16)]{Breuilmodp}).

Let $\mM_{k_E}$ be defined as in (\ref{eqn:B}) with 
\[
U^{(j)} = \begin{pmatrix}
  1 &   \\
  x_j & 1 \\
 \end{pmatrix}
\]
for $j\neq 0$ and
\[
U^{(0)} = \begin{pmatrix}
  \alpha &   \\
   & \alpha' \\
 \end{pmatrix}
\begin{pmatrix}
  1 &   \\
  x_0 & 1 \\
 \end{pmatrix}
\]
Let $\fM_{k_E}$ be the $\mO_L[\![v]\!]\otimes_{\bZ_p} k_E$-submodule of $\mM_{k_E}$ generated by $(\gamma^{(j)})_j$.
Note that $\varphi$ maps $\fM_{k_E}$ to itself.
Then a calculation (cf.~\cite[\S 7.4]{EGS} with $J = \emptyset$) shows that $\Theta_{p-1}(\fM_{k_E}) \cong \mathcal{F}_{p-1}(M)$, where the functors $\Theta_{p-1}$ and $\mathcal{F}_{p-1}$ are introduced in \cite[Appendix A]{EGS}.
The result now follows from \cite[Propositions A.3.2 and A.3.3]{EGS}.
The case $\rhobar$ is irreducible is proved similarly.
\end{proof}

\subsection{Modular Serre weights.} \label{sec:serrewts}

In this section, we recall some of the various descriptions of the set of modular Serre weights of a Galois representation. 
More specifically, we fix a generic Galois representation $\rhobar: G_L \ra \GL_2(k_E)$ with $V(\rhobar^\semis) \cong R_{s_\rhobar}(\mu)$.
In \cite{BDJ}, a set of Serre weights is defined in terms of the restriction $\rhobar|_{I_L}$ of $\rhobar$ to the inertia group.
We denote this set $W(\rhobar)$, though it is denoted $\mathcal{D}(\rhobar)$ in \cite{Breuilmodp} (see \cite[Proposition A.3]{Breuilmodp}). 

We will need two perspectives on weights, one formulated in terms of combinatorial objects called ``formal weights'', and another based on the map $\ft_\mu$ from \S \ref{sec:extgph}. 
The first perspective allows explicit computations and ready access to the results of~\cite{BPmodp} and~\cite{Breuilparameters}, while the second is better suited to the study of deformations of Kisin modules. Of course, they are ultimately equivalent, and we will give an explicit dictionary in Proposition~\ref{pp:serrewts}.

\paragraph{Formal weights.}
The genericity assumption implies that~$\rhobar$ can be twisted by a character~$\chi$ so that its restriction to inertia has one of the following two forms
\begin{gather*}
\rhobar |_{I_{\bQ_p}} \cong \fourmatrix{\omega_f^{\sum_{j=0}^{f-1}(r_j+1)p^j}}{*}{0}{1} \text{ in the reducible case}\\
\rhobar|_{I_{\bQ_p}} \cong \fourmatrix{\omega_{2f}^{\sum_{j=0}^{f-1}(r_j+1)p^j}}{0}{0}{\omega_{2f}^{q\sum_{j=0}^{f-1}(r_j+1)p^j}} \text{ in the irreducible case}.
\end{gather*}
In the notation of the previous section, we have $r_j+1 = \mu_{1, j}-\mu_{2, j} = \langle \alpha^{(j)}, \mu \rangle$, so that $1 < r_j < p-4$ for all values of~$j$. 

\begin{rk}
These conditions are slightly more restrictive than those in~\cite{Breuilmodp}. These constraints come from our techniques for the study of deformation rings, and we have not tried to optimize our bounds.
\end{rk}

A \emph{formal weight} is an~$f$-tuple of linear polynomials $\lambda = (\lambda_0(x_0), \ldots, \lambda_{f-1}(x_{f-1}))$ with coefficients in~$\bZ$ such that the leading coefficients are either~$1$ or~$-1$.
Note that formal weights form a group under composition.
There is a homomorphism from the group of formal weights to~$W$ sending~$\lambda$ to~$w_\lambda\in W$ where~$w_{\lambda,j}$ is trivial if and only if the leading coefficient of~$\lambda_j$ is~$1$.
(In fact, the group of formal weights is naturally isomorphic to the extended affine Weyl group of $\bG^\der$, and the above homomorphism is the natural map to $W$, which is canonically isomorphic to the Weyl group of $\bG^\der$.)
For a formal weight $\lambda$, we define the linear polynomial in $f$-variables
\begin{equation}\label{eqn:e}
e(\lambda) = 
\begin{dcases}
\half \left ( \sum_{i=0}^{f-1} p^i(x_i - \lambda_i(x_i)) \right ) &\mbox{if } w_{\lambda,f-1} = \id \\
\half \left (p^f - 1 + \sum_{i=0}^{f-1} p^i(x_i - \lambda_i(x_i)) \right ) \quad &\mbox{if } w_{\lambda,f-1} \neq \id.
\end{dcases}
\end{equation}
While $e(\lambda)$ \emph{a priori} has coefficients in $\frac{1}{2}\bZ$, for all formal weights that we consider, the coefficients in fact lie in $\bZ$.

When~$\rhobar$ is semisimple, its weights are parametrized by subsets of~$\{0, \ldots, f-1 \}$ as follows. 
To a subset~$J$, we associate the formal weight~$\lambda_J = (\lambda_0(x_0), \ldots, \lambda_{f-1}(x_{f-1}))$ of $\rhobar$, where~$\lambda_j$ is uniquely determined by whether $j-1, j$ are contained in~$J$ or not, according to the following table. 
(We remark that this is the same normalization as~\cite{Breuilparameters}, and differs from that in~\cite[\S~11]{BPmodp}. Recall that~$\ch_J$ is the characteristic function of~$J$.)
\begin{equation}\label{formalweights}
\begin{array}{c|c|c}
\ch_J(j-1, j) & \substack{\text{$\lambda_j$ when $\rhobar$ is reducible,}\\ \text{or $\rhobar$ is irreducible and $j > 0$}} & \lambda_j, \text{ when~$\rhobar$ is irreducible and~$j = 0$}\\
\hline
(1, 1) & p-3-x_j & p-1-x_0\\
(1, 0) & x_j + 1 & x_0 - 1\\
(0, 1) & p-2-x_j & p-2-x_0\\
(0, 0) & x_j & x_0
\end{array}
\end{equation}
Note that the image of~$\lambda_J$ in $W$ is $F(w_J)$, which is defined in \S \ref{sec:extgph}: this is because the image of~$\lambda_J$ is nontrivial precisely on~$J$, $w_J$ is nontrivial on~$\delta_\red(J)$, and~$F$ is a right shift on the Weyl group.
Moreover, the map from subsets of $\{0,\ldots,f-1\}$ to $W$ taking $J$ to $F(w_J)$ is a bijection.

Then, the modular weights of~$\rhobar$ are precisely the Serre weights of the form
\begin{equation}\label{weightdefinition}
\sigma(\lambda, \rhobar) = (\lambda_0(r_0), \ldots, \lambda_{f-1}(r_{f-1})) \otimes \mathrm{det}^{e(\lambda)(r_0, \ldots, r_{f-1})}\chi
\end{equation}
as~$\lambda$ runs through the formal weights of~$\rhobar$. In this formula, we have followed  our convention of writing~$\det$ for the $k_E^\times$-valued character $\det_E = \sigma_0 \circ \det$ of~$\GL_2(k_L)$.

To describe the weights of a nonsplit reducible~$\rhobar$, let $J_\rhobar$ be the set $\{f-1-j|U^{(j)} \textrm{ is diagonal}\}$. (This coincides with the definition in \cite[(17)]{Breuilmodp}, see the proof of Proposition \ref{pp:rhophi}.) 
We define~$W(\rhobar) \subset W(\rhobar^{\mathrm{ss}})$ to be the set of weights corresponding to subsets $J \subseteq \delta_{\red}(J_{\rhobar})$. This is compatible with~\cite{Breuilmodp}: indeed, the definition there is that the formal weight~$\lambda$ defines an element of~$W(\rhobar)$ if and only if $\lambda_j \in \{p-2-x_j, p-3-x_j \}$ implies that a certain parameter~$\mu_{f-j}$ of \emph{loc.~cit.} (unrelated to our use of the notation $\mu$) is $0$, and by~\cite[(18)]{Breuilmodp} we have $\delta_{\red}(J_{\rhobar}) = \{f-j: \mu_j = 0 \}$.

\paragraph{Extension graph.}
Recall that $\omega_j$ is the fundamental weight of~$\Lambda_W$ which is nonzero in embedding $j$, and that for a subset $J \subset J_\rhobar$ we write $\omega_J = \sum_{j\in J} \omega_j$ (this differs slightly from the notation in \cite[\S 2]{LMS}). 
Our genericity condition on~$\rhobar$ implies that $s_\rhobar \omega_J \in \Lambda^\mu_W$ for all~$J$.
Let $\sigma_J$ be the Serre weight $F(\ft_\mu(s_\rhobar\omega_J))$. 

\begin{pp}\label{pp:serrewts}
The set $W(\rhobar)$ is equal to $\{\sigma_J|J\subset J_\rhobar\}$.
The formal weight corresponding to~$\sigma_J$ is $\lambda_{\delta_\red(J)}$ (both in the reducible and irreducible case).
\end{pp}
\begin{proof}
The case where~$\rhobar$ is semisimple follows from~\cite[Proposition~2.10]{LMS}: we briefly recall the argument.  As~$J$ varies, the Serre weights~$F(w_{J}t_{s_\rhobar\omega'_J-p\pi^{-1}\omega'_J}(\mu) - \eta)$ are the ``obvious weights'' of~$\rhobar$, where by definition $\omega'_J \in X(\bT)$ is a lift of $\omega_J \in \Lambda_W$ and $\tld{w}_J = w_{J}t_{-\pi^{-1}(\omega'_J)} \in~\Omega$. 

Since $w_J t_{-\pi^{-1}\omega'_J}t_{\pi^{-1}s_\rhobar \omega'_J} \in W_a$ (the translation part is in $\Lambda_R$) we find that
\[
\ft_\mu(s_\rhobar \omega_J) = w_J t_{-\pi^{-1}\omega'_J}\cdot (\mu- \eta + s_\rhobar \omega'_J) = w_Jt_{s_\rhobar\omega'_J - p \pi^{-1}\omega'_J}(\mu) - \eta.
\]
Hence~$\sigma_J = F(\ft_\mu(s_\rhobar \omega_J))$ is an obvious weight, and the claim follows since all modular weights of~$\rhobar$ are obvious weights (when~$n = 2$).

Now we make explicit the correspondence with formal weights, which will make it clear that the statement of the proposition extends to nonsplit~$\rhobar$. Assume that~$\rhobar$ is split reducible, so that $s_{\rhobar} = 1$, and that the character~$\chi$ is trivial, so that
\[
\rhobar |_{I_{\bQ_p}} \cong \fourmatrix{\omega_f^{\sum_{j=0}^{f-1}(r_j+1)p^j}}{*}{0}{1}.
\]
Then the components of~$\mu$ are $a_j = r_j+1$ and $b_j = 0$ for all~$j$. 

Recall that~$\pi^{-1}$ is a left shift on~$X(\bT)$. Then we have the following correspondence
\[
\begin{array}{c|c}
(j, j+1) \in J &  \ft_\mu(\omega_J)_j = \left ( \tld{w}_J \cdot (\mu - \eta + \omega'_J) \right )_j\\
\hline
(1, 1) & (b_j-1, a_j-p+1)\\
(0, 1) & (b_j-1, a_j-p)\\
(1, 0) & (a_j, b_j)\\
(0, 0) & (a_j-1, b_j)
\end{array}
\]
Equivalently, the following table computes the components of the Serre weight~$\sigma_J$.
\[
\begin{array}{c|c}
(j-1, j) \in \delta_\red(J) & (\sigma_J)_j\\
\hline
(1, 1) & \Sym^{p - r_j - 3} \otimes \mathrm{det}^{r_j-p + 2}\\
(0, 1) & \Sym^{p- r_j  - 2} \otimes \mathrm{det}^{r_j-p+1}\\
(1, 0) & \Sym^{r_j + 1}\\
(0, 0) & \Sym^{r_j}
\end{array}
\]
A computation with~$e(\lambda)$, as in Lemma~\ref{indices} to follow, then shows that~$\sigma_J$ is the weight attached to the formal weight corresponding to~$\delta_{\red}(J)$. The remaining reducible cases follow by twisting~$\rhobar$ and observing that $J \subseteq J_\rhobar$ if and only if $\delta_\red(J) \subseteq \delta_\red(J_\rhobar)$.
To do the irreducible case, it suffices to observe that since $\omega'_J - s_\rhobar \omega'_J \in \Lambda_R$, we have
\[
\ft_\mu(s_\rhobar\omega_J) = \tld w_J \cdot(\mu - \eta + s_\rhobar \omega'_J).
\]
\end{proof}

\subsection{Orientations and crystalline Frobenius eigenvalues.}

In this section, $n$ is again an arbitrary natural number.
Using Kisin's theory, we describe Frobenius eigenvalues of crystalline Dieudonn\'e modules arising from Kisin modules.
We also introduce the notion of an orientation which plays a key role in calculating Galois deformation rings.

\begin{pp}\label{pp:dpst}
Let $\tau$ be a principal series tame inertial type for~$W_L$.
Suppose that $\fM$ is a Kisin module over $\mO_E$ with descent datum of type $\tau$, such that 
\[
\mathbb{V}^*_\dd(\fM\otimes_{\mO_L[\![u]\!]} \mO_{\mE,K})
\]
is the restriction to $G_{L_\infty}$ of a potentially crystalline representation $V$ of $G_L$.
Then the $\varphi$-module with semilinear $\Delta$-action $D_{\pst}(V^\vee)$ is isomorphic to $(\fM/u)\otimes_{\mO_E} E$ where $(-)^\vee$ denotes the contragredient representation.
\end{pp}
\begin{proof}
This follows from \cite[Proposition 2.1.5]{KisinFcrys} by adding tame descent datum.
Namely, this proposition shows that $(\fM/u)\otimes_{\mO_E} E$ and $D_{\pst}(V^\vee|_{G_K})$ are canonically isomorphic as $\varphi$-modules.
Since the $\Delta$-action is compatible with the identification $V|_{G_{K_\infty}} = \mathbb{V}^*(\fM\otimes_{\mO_L[\![u]\!]} \mO_{\mE,K})$, the isomorphism between $(\fM/u)\otimes_{\mO_E} E$ and $D_{\pst}(V^\vee|_{G_K})$ respects the induced $\Delta$-actions.
\end{proof}

\begin{cor}\label{cor:frob}
Suppose that $\tau$ is a principal series tame inertial type as in \S \ref{sec:kisin} with the characters $\eta_k$ distinct.
Let $\fM$ be a Kisin module over $\mO_E$ with descent datum of type $\tau$, $\beta$ an eigenbasis, and let $C^{(j)} = (c^{(j)}_{ik})_{ik}$ be the matrix defined in \S \ref{sec:kisin}.
Then the $\varphi$-eigenvalue on the $\eta_k$-isotypic part of $(\fM/u)\otimes_{\mO_E} E$ is given by 
\[
\prod_j c_{kk}^{(j)} \pmod{u}.
\]
\end{cor}

\begin{defn}
Let $(a_{k,j})_{k,j}$ be a tuple of integers for $1\leq k\leq n$ and $0\leq j\leq f-1$.
An orientation of $(a_{k,j})_{k,j}$ is a tuple $(s_j)_j \in S_n^f$ such that 
\[
a_{s_j(1),f-1-j} \geq a_{s_j(2),f-1-j} \geq \cdots \geq a_{s_j(n),f-1-j}
\]
and $a_{s_j(1),f-1-j} - a_{s_j(n),f-1-j} < p$.
\end{defn}

\noindent These inequalities in particular imply that if $(s_j)$ is an orientation of the tuple $(a_{k,j})_{k,j}$, then 
\[
\ba^{(j)}_{s_j(1)} \geq \ba^{(j)}_{s_j(2)} \geq \cdots \geq \ba^{(j)}_{s_j(n)}.
\]
Given a principal series tame inertial type $\tau$ as above, there is a choice of the tuple $(a_{k,j})_{k,j}$ giving $\tau$ as in \S \ref{sec:kisin} which has an orientation (though this choice is never unique).

We now fix a tuple of integers $(a_{k,j})_{k,j}$ for $1\leq k\leq n$ and $0\leq j\leq f-1$ for which there exists an orientation $(s_j)_j$, which we also fix.
Let $\tau$ be the principal series tame inertial type defined in terms of $(a_{k,j})_{k,j}$ as in \S \ref{sec:kisin}, let~$R$ be an $\mO_E$-algebra, and let $\fM$ be a Kisin module over $R$ with descent datum of type $\tau$.
Let $^\varphi \fM^{(j)}_k$ be the $R[\![v]\!]$-submodule of $\varphi^*(\fM^{(j)})$ on which $\Delta$ acts by $\eta_k$.
Note that $^\varphi \fM^{(j)}_k$ is typically strictly larger than $\varphi^*(\fM^{(j)}_k)$.

\paragraph{}

If $\beta = (f^{(j)}_k)_{k,j}$ is an eigenbasis for $\fM$, then by the proof of \cite[Lemma 2.9]{LLLM} 
\[
\beta^{(j)}_{s_j(n)} := (u^{\ba_{s_j(k)}^{(j)}-\ba_{s_j(n)}^{(j)}} f^{(j)}_{s_j(k)})_k
\]
is a basis for $\fM^{(j)}_{s_j(n)}$ for each $j$ and
\[
^\varphi \beta^{(j-1)}_{s_j(n)} := (u^{\ba_{s_j(k)}^{(j)}-\ba_{s_j(n)}^{(j)}} \otimes f^{(j-1)}_{s_j(k)})_k
\]
is a basis for $^\varphi\fM^{(j-1)}_{s_j(n)}$.
For each $j$ let $A^{(j)}\in \mathrm{Mat}_n(R[\![v]\!])$ be the matrix such that 
\[
\phi^{(j)}_{\fM}(^\varphi\beta^{(j)}_{s_{j+1}(n)}) = \beta^{(j+1)}_{s_{j+1}(n)}A^{(j)}.
\]
Then \cite[Proposition 2.13]{LLLM} states that
\[
C^{(j)} = \ad_{s_{j+1}} \Diag(u^{\ba_1}, \ldots u^{\ba_n}) (A^{(j)}) := \ad \left ( s_{j+1}\Diag(u^{\ba_{s_{j+1}(1)}^{(j+1)}}, \ldots, u^{\ba_{s_{j+1}(n)}^{(j+1)}} ) \right )(A^{(j)}),
\]
where $\Diag(u^{\ba_1^{(j+1)}},\ldots,u^{\ba_n^{(j+1)}})$ denotes the diagonal matrix with $k$-th diagonal entry $u^{\ba_k^{(j+1)}}$ and $s_{j+1}$ is identified with the permutation matrix with $(s_{j+1}(k),k)$-entry equal to $1$ for all $k$.
Corollary \ref{cor:frob} then implies the following.

\begin{cor}\label{cor:frob'}
Let $\tau$, $\fM$, and $\beta$ be as in Corollary \ref{cor:frob}.
Let $(s_j)_j$ be an orientation for a tuple $(a_{k,j})_{k,j}$ giving rise to $\tau$ as in \S \ref{sec:kisin}.
If we let $A^{(j)}=(x^{(j)}_{ik})_{ik}$ be as above, then the Frobenius eigenvalue on the $\eta_k$-isotypic part of $(\fM/u)\otimes_{\mO_E} E$ is given by 
\[
\prod_j x^{(j)}_{s_{j+1}^{-1}(k)s_{j+1}^{-1}(k)} \pmod{v}.
\]
\end{cor}

\subsubsection{Orientations and base change.}

Recall the notation $\tau$, $s_\tau$, $d$, $L'$, and $\tau'$ from \S\ref{sec:kisinbasechange}.
Then there exists a tuple $(a_{k,j})_{k,j}$ giving $\tau'$ as in \S\ref{sec:kisin} which admits an orientation and such that $a_{k,j} = a_{s_\tau^{-1}(k),j+f}$ for all $j$ and $k$.
We can and do choose an orientation $s' \in W'$ (the Weyl group of $\Res_{L'/\bQ_p}\GL_2$) such that 
\begin{equation}\label{eqn:orbc}
(s'_{j+f})^{-1} = s_\tau (s'_j)^{-1} s_\tau^{-1}
\end{equation}
for all $j$.
If $(\fM_R,\iota)$ is a Kisin module with descent data $\tau$ and $\beta$ is an eigenbasis for $(\fM_R,\iota)$, then we can define $A^{(j)}$ with respect to $s'$ as before.
Then we have that $A^{(j)} = A^{(j+f)}$ for all $j$.

\subsubsection{Orientations and \'etale $\varphi$-modules.}\label{sec:orphi}
Let $\tau$ be a tame inertial type and $\fM_R$ (or $(\fM_R,\iota)$) be a Kisin module over $R$ with descent datum of type $\tau$.
Suppose that $\sigma(\tau)$ is $R_w(\mu)$ for $(w,\mu) \in W\times X(T)$ where $\mu$ is dominant and $p$-restricted.
There exists an element $s^*\in W$ such that $(s^*)^{-1}w F(s^*) = (s_\tau,\id,\ldots,\id)$ for some $s_\tau \in S_n$.
By our discussion of Deligne--Lusztig induction in \S \ref{sec:DL} (compare also \cite[Lemma 4.2]{herzig}), $R_w(\mu)$ is isomorphic to 
\[
R_{(s_\tau,\id,\ldots,\id)}((s^*)^{-1} (\mu)).
\]
Let $L'$ be the unramified extension of $L$ of degree equal to the order $d$ of $s_\tau$.
Let $\tau'$ be the base change to $L'$, so that it is a principal series tame inertial type for $\GL_n(L')$.

Let $a_{k,j}$ be $\left ( (s^*)^{-1} \mu \right)_{k,j}$ for $1\leq k \leq n$ and $0\leq j\leq f-1$.
Then define $a_{k,j}$ for $1\leq k \leq n$ and $0\leq j\leq df-1$ so that $a_{k,j+f} = a_{s_\tau^{-1}(k),j}$ for all $k$ and $j$.
Define $s\in W$ so that $s_j=(s^*_{f-1-j})^{-1}$ for $0\leq j \leq f-1$.
There is a unique $s'\in W'$ satisfying (\ref{eqn:orbc}) for all $j$ such that $s_j = s'_j$ for $0 \leq j \leq f-1$.
Then $s'$ is an orientation for $(a_{k,j})_{k,j}$.

\begin{pp}\label{pp:kphi}
Let $(\fM_R,\iota)$ be a Kisin module over $R$ with descent datum of type $\tau$.
Let $\beta$ be an eigenbasis for $(\fM_R,\iota)$, and define $A^{(j)}$ with respect to the orientation $s'$ as before.
Recall the definitions of $K'$ and $\Delta'$ from before.
Then there is a basis for $(\fM_R\otimes_{\mO_L[\![u]\!]}\mO_{\mE,K'})^{\Delta'}$ such that 
\[
B^{(j)} = A^{(j)} s_j^{-1}v^{\mu_j}.
\]
\end{pp}
\begin{proof}
This follows from the proof of \cite[Proposition 3.4]{Le}.
\end{proof}

\subsection{Deformation rings.} \label{sec:defrings}

For a local mod $p$ Galois representation $\rhobar: G_L \ra \GL_n(k_E)$, let $R_{\rhobar}$ be the framed unrestricted $\mO_E$-deformation ring for $\rhobar$.
If $\psi: G_L \ra \mO_E^\times$ is a lift of $\det \rhobar\overline{\varepsilon}^{-1}$, then let $R^\psi_\rhobar$ be the quotient of $R_\rhobar$ parametrizing lifts of $\rhobar$ with determinant $\psi\varepsilon$. 
For an inertial type $\tau$, let $R_{\rhobar}^{\tau}$ be the quotient of $R_{\rhobar}$ corresponding to potentially crystalline lifts of $\rhobar$ of Galois type $\tau$ and parallel Hodge--Tate weights $(0,1)$. 

From now on, we set~$n = 2$.
As in \S \ref{sec:phi}, suppose that $\rhobar:G_L\ra \GL_2(k_E)$ satisfies $V(\rhobar^\semis|_I) = R_{s_\rhobar}(\mu)$.

\subsubsection{Deformations of Kisin modules with tame descent datas.}\label{sec:defk}

We describe Galois deformation rings when $\sigma_{\mathrm{alg}}$ is trivial and~$\sigma_{\mathrm{sm}} = \sigma(\tau)$ is a generic tame type, using the Kisin modules introduced in \S \ref{sec:kisin}.
Let~$\mM_{k_E}$ be an element of $\Phi\textrm{-}\mathrm{Mod}^\et(k_E)$ with basis $\gamma^{(j)}$ such that $B^{(j)}$ is given by (\ref{eqn:B}) so that $\mathbb{V}^*(\mM_{k_E}) \cong \rhobar|_{G_{L_\infty}}$.

\paragraph{Construction of tame types.}
We now describe tame types whose reductions contain modular weights of a given $\rhobar$.

\begin{pp}\label{pp:intersect}

Let~$\rhobar$ be a generic Galois representation (possibly not semisimple). 
Suppose that $w\in W$, $\nu\in \Lambda_R$ are such that $-\eta+\nu+s_\rhobar w\omega_J\in \Lambda_W^\mu$ for all subsets $J \subset \bZ/f$, and $V$ is the tame type $R_{s_\rhobar w}(\mu - \eta+\nu)$ for $\GL_2(L)$.
Then $\JH(\lbar V) \cap W(\rhobar)$ is nonempty if and only if $-\eta+\nu = -s_\rhobar w' \eta$ for some $w'\in W$ such that 
\begin{itemize}
\item $(w_j,w'_j)\neq (s,\id)$ for all $j\in \bZ/f$ and 
\item if $(w_j,w'_j) = (\id,s)$, then $j\in J_\rhobar$. 
\end{itemize}
In this case, 
\[
W(\rhobar)\cap \JH(\overline{V})=\big\{F(\ft_\mu(s_\rhobar\omega_J)): J_1(V) \subset J \subset J_2(V)^c \cap J_\rhobar\big\}
\]
where 
\[
J_1(V) = \{j\in \bZ/f: w_j = \id \quad \textrm{and} \quad w'_j \neq \id\}, \qquad J_2(V) = \{j\in \bZ/f: w_j = \id \quad \textrm{and} \quad w'_j = \id\},
\]
and $J_2(V)^c$ denotes the complement of $J_2(V)$ in $\bZ/f$.
\end{pp}
\begin{proof}
By Propositions \ref{pp:jh} and \ref{pp:serrewts}, 
\[
\JH(\overline{V}) = \{F(\ft_\mu(s_\rhobar w\omega_{J'} - \eta + \nu): J'\subset \bZ/f\}
\]
and
\[
W(\rhobar)=\big\{F(\ft_\mu(s_\rhobar\omega_J)): J \subset J_\rhobar\big\}.
\]
If these sets intersect, then since~$\ft_\mu$ is injective on~$\Lambda^\mu_W$ we must have that $J^c = J'$, and $-\eta+\nu = -s_\rhobar w' \eta$ for some $w'\in W$.

To compute the intersection we need to classify equalities in~$\Lambda_W$ of the form
\[
w\omega_{J^c} - w'\eta = \omega_J
\]
for fixed~$w, w'$ and varying~$J, J^c$.
If there exists~$j$ such that $(w_j, w_j') = (s, \id)$, the LHS is negative at~$j$ (the coefficient of $\omega_j$ is negative) but the RHS at~$j$ is nonnegative, so there cannot be any
such equalities.
Similarly, if~$(w_j, w_j') = (\id, s)$ then the LHS is positive at~$j$, so we need~$j \in J$, hence $j \in J_\rhobar$.

So assume the two conditions in the proposition.
If $(w_j,w'_j) = (\id,s)$, then $(w\omega_{J^c}-w'\eta)_j$ is nonzero and so $j\in J$.
If $(w_j,w'_j) = (\id,\id)$, then $(w\omega_{J^c}-w'\eta)_j$ cannot be $\omega_j$ and so $j\notin J$.
This implies the inclusion 
\[
W(\rhobar)\cap \JH(\overline{V})\subset \big\{F(\ft_\mu(s_\rhobar\omega_J)): J_1(V) \subset J \subset J_2(V)^c \cap J_\rhobar\big\}.
\]
The reverse inclusion is checked similarly, and this completes the proof of the proposition.
\end{proof}

\begin{rk}
Since the types~$R_{s_\rhobar w}(\mu-s_\rhobar w'\eta)$ have reductions with distinct sets of Jordan--H\"older factors, they are pairwise nonisomorphic.
If $\rhobar$ is semisimple, then $J_\rhobar = \bZ/f$ so that there are $3^f$ tame types satisfying the conditions of Proposition \ref{pp:intersect}.
It is known that for $\rhobar$ generic and semisimple, there are $3^f$ tame types $V$ such that $\JH(\lbar V) \cap W(\rhobar)$ is nonempty.
We conclude that all such tame types appear in Proposition \ref{pp:intersect}, though we will not need this in what follows.
\end{rk}

Proposition~\ref{pp:intersect} gives another perspective on the parametrization of $I_1$-invariants in~$D_0(\rhobar)$ in~\cite[Corollary~14.10]{BPmodp} as follows.
If~$\chi \in (\soc_K D_0(\rhobar))^{I_1}$, we will write~$J(\chi)$ for the subset such that $\sigma_{J(\chi)} \cong \sigma(\chi)$.

\begin{lemma}\label{charactersw}
Let~$\rhobar$ be a generic, possibly not semisimple, Galois representation.
Let~$\chi$ be a character appearing in~$D_0(\rhobar)^{I_1}$.
Then there exist $s^*, w, w' \in W$ such that
\[
\theta(\chi)[1/p] \cong R_{s_\rhobar w}(\mu - s_\rhobar w' \eta), \, s^*F(s^*)^{-1} = s_\rhobar w, \text{ and } \chi = (s^*)^{-1}(\mu-s_\rhobar w'\eta)|_H.
\]
Furthermore, if~$\chi$ appears on~$(\soc_K D_0(\rhobar))^{I_1}$ and~$i\in \bZ/f$, then~$i\in J(\chi)$ if and only if $s^*_{i-1} \ne \id$.
\end{lemma}
\begin{proof}
A character $\chi: H \to \mO_E^\times$ appears in~$D_0(\rhobar)^{I_1}$ if and only if $\lbar \theta(\chi)\cap W(\rhobar)$ is nonempty.
By considerations of central characters, $\theta(\chi)[1/p]$ must then be isomorphic to $R_{s_\rhobar w} (\mu-\eta+\nu)$ as in Proposition~\ref{pp:intersect}.
Note that $-\eta+\nu+s_\rhobar w\omega_J\in \Lambda_W^\mu$ for all subsets $J \subset \bZ/f$ by the genericity assumption on $\rhobar$, which guarantees that $\chi$ appearing in~\cite[Corollary~14.10]{BPmodp} is generic as well.
Then Proposition~\ref{pp:intersect} implies that there exist~$w,w'\in W$ such that $\theta(\chi)[1/p]$ is isomorphic to $R_{s_\rhobar w} (\mu-s_\rhobar w'\eta)$.
By our discussion of Deligne--Lusztig theory in \S \ref{sec:DL}, there exists $s^* \in W$ such that
\[
s^*F(s^*)^{-1} = s_\rhobar w \text{ and } \chi = (s^*)^{-1}(\mu-s_\rhobar w'\eta)|_H.
\]
Now assume~$\chi$ appears in~$(\soc_K D_0(\rhobar))^{I_1}$. 
The weight $(s^*)^{-1}(\mu-s_\rhobar w'\eta)$ need not be $p$-restricted, but there exists $\omega' \in X(\bT)$ lifting some~$-\omega_{J'} \in \Lambda_W$ such that $\sigma(\chi) \cong \sigma_{J(\chi)}$ is isomorphic to $F((s^*)^{-1}(\mu-s_\rhobar w'\eta)+\omega'-p\pi^{-1}(\omega'))$.
The image of $\omega'_i$ in $\Lambda_W/\Lambda_R$ is nonzero if and only if $s^*_{i-1} \neq \id$.
Writing $\tld{w}$ for $t_{-\pi^{-1}(\omega')} (s^*)^{-1}$ and using the definition of $\sigma_{J(\chi)}$, we have that 
\[
F(\tld{w}_{J(\chi)}\cdot(\mu-\eta+s_\rhobar\omega'_{J(\chi)})) \cong \sigma_{J(\chi)} \cong F(\tld w \cdot (\mu-\eta-s_\rhobar w'\eta + s^*\omega'+s^*\eta))
\]
for some choice of $\omega'_{J(\chi)}$ lifting $\omega_{J(\chi)}$.
We conclude that $\ft_\mu(s_\rhobar \omega'_{J(\chi)}) = \ft_\mu(s^*\eta-s_\rhobar w'\eta+s^*\omega')$ so that the image of $\omega'_{J(\chi)}$ and $\omega'$ in $\Lambda_W/\Lambda_R$ coincide (since~$\ft_\mu$ is injective and $s^*\eta - s_\rhobar w' \eta$ has image in~$\Lambda_R$).
It follows that $i \in J(\chi)$ if and only if $s^*_{i-1} \ne \id$.
\end{proof}

\paragraph{Computation of deformation rings.}
Now let $w$ and $w'$ be as in Proposition \ref{pp:intersect}, and let~$\tau$ be the tame type with $\sigma(\tau) \cong R_{s_\rhobar w}(\mu-s_\rhobar w' \eta)$.
As before, let $d$ be the order of $s_\tau$. (So $d = 1$ if~$\tau$ is principal series, and $d = 2$ if~$\tau$ is cuspidal.)
For each $j \in \bZ /df$, set
\begin{align*}
A^{(j)} = \, \, \,  & B^{(j)} v^{-(\mu_{f-1-j} - s_{\rhobar,f-1-j} w'_{f-1-j}\eta_{f-1-j})}s_{\rhobar,f-1-j}w_{f-1-j}\\
\overset{(\ref{eqn:B})}{=} & U^{(j)} v^{w'_{f-1-j}\eta_{f-1-j}}w_{f-1-j},
\end{align*}
where $j$ is interpreted on the right hand sides as an element of $\bZ /f$ via the natural projection.
The matrices $A^{(j)}$ are of the form
\begin{equation} \label{eqn:shape}
 \begin{pmatrix}
  v &   \\
  x_jv & 1 \\
 \end{pmatrix},
\begin{pmatrix}
  1 &  \\
   & v \\
 \end{pmatrix},
\textrm{ or }
\begin{pmatrix}
   & 1  \\
  v & x_j \\
 \end{pmatrix}
\end{equation}
when $j\neq 0$ and some matrix 
\[
\begin{pmatrix}
  \alpha &   \\
   & \alpha' \\
 \end{pmatrix}
\]
multiplied by a matrix in (\ref{eqn:shape}) when $j=0$.

We define $(\overline{\fM}^\tau,\iota)$ to be the Kisin module over $k_E$ with descent datum of type $\tau$ given by the matrices $A^{(j)}$ for a fixed eigenbasis.
Then $\mathbb{V}_\dd^*(\overline{\fM}^\tau \otimes_{\mO_L[\![u]\!]} \mO_{\mE,K'})$ is isomorphic to $\rhobar$ by \S \ref{sec:phi} and Proposition \ref{pp:kphi}.
We are going to write down an explicit deformation of~$\overline{\fM}^\tau$ to a complete Noetherian local $\mO_E$-algebra $R^\tau_{\overline{\fM}^\tau}$. This will admit a formally smooth extension that will be simultaneously a formally smooth extension of the Galois deformation ring $R^\tau_{\rhobar}$. 
Because of this fact, we will be able to carry through most of our computations explicitly on $R^\tau_{\overline{\fM}^\tau}$.

Write~$s$ for the nontrivial element of~$W(\GL_2) = S_2$. Let $R^\tau_{\overline{\fM}^\tau}$ be the ring 
\[
\mO_E[\![(X_j,Y_j)_{j=0}^{f-1},X_\alpha,X_{\alpha'}]\!]/I(\tau)
\]
where $I(\tau)= (f_j(\tau))_{j=0}^{f-1}$ and 
\[
f_j(\tau) = 
\begin{dcases}
Y_j &\mbox{if } (w_{f-1-j}, w'_{f-1-j}) = (\id, \id) \\
X_j &\mbox{if } (w_{f-1-j}, w'_{f-1-j}) = (\id, s)\\
Y_j &\mbox{if } (w_{f-1-j}, w'_{f-1-j}) = (s, s) \textrm{ and } x_j \neq 0 \\
X_jY_j-p &\mbox{if } (w_{f-1-j}, w'_{f-1-j}) = (s, s) \textrm{ and } x_j = 0.
\end{dcases}
\]

Let $\fM^\tau$ be the Kisin module over $R^\tau_{\overline{\fM}^\tau}$ with descent datum of type $\tau$ given by the matrices $A^{\tau,(j)}$, with respect to a fixed eigenbasis, which are defined to be

\begin{equation}\label{eqn:deform}
\begin{dcases} 
\\
\hspace{4mm} \begin{pmatrix}
  v+p &   \\
  (X_j+[x_j])v & 1 \\
 \end{pmatrix} &\mbox{if } A^{(j)} =  \begin{pmatrix}
  v &   \\
  x_jv & 1 \\
 \end{pmatrix}, \\
\\
\hspace{4mm} \begin{pmatrix}
  1 & -Y_j  \\
   & v+p \\
 \end{pmatrix} &\mbox{if } A^{(j)} =  \begin{pmatrix}
  1 &   \\
   & v \\
 \end{pmatrix}, \\
\\
\hspace{4mm} \begin{pmatrix}
  -Y_j & 1 \\
  v & X_j \\
 \end{pmatrix} &\mbox{if } A^{(j)} = \begin{pmatrix}
   & 1  \\
  v &  \\
 \end{pmatrix}, \\
\\
\hspace{4mm} \begin{pmatrix}
  -p(X_j+[x_j])^{-1} & 1 \\
  v & X_j+[x_j] \\
 \end{pmatrix} \hspace{4mm} &\mbox{if } A^{(j)} = \begin{pmatrix}
   & 1  \\
  v & x_j \\
 \end{pmatrix} \textrm{ and } x_j \neq 0, \\
\\
\end{dcases}
\end{equation}
when $j\neq 0$, and the product of the diagonal matrix 
\[
\begin{pmatrix}
  X_\alpha +[\alpha] &   \\
   & X_{\alpha'}+[\alpha'] \\
 \end{pmatrix}
\]
with the appropriate matrix in (\ref{eqn:deform}) when $j=0$.

Let $R^{\tau,\square}_{\overline{\fM}^\tau}$ represent the moduli problem which maps a local complete Noetherian $\mO_E$-algebra $R$ to the set of pairs $(f,b)$ where $f: R^\tau_{\overline{\fM}^\tau} \ra R$ is an $\mO_E$-algebra homomorphism and $b$ is an $R$-basis for 
\[
\mathbb{V}^*_\dd(\fM^\tau \otimes_{ R^\tau_{\overline{\fM}^\tau}, f} R \otimes_{\mO_L[\![u]\!]} \mO_{\mE,K'}).
\]
Then the natural forgetful map 
\[
\Spf R^{\tau,\square}_{\overline{\fM}^\tau} \ra \Spf R^\tau_{\overline{\fM}^\tau}
\]
is formally smooth. We also denote by $I(\tau)$ the ideal $I(\tau)R^{\tau,\square}_{\overline{\fM}^\tau} \subset R^{\tau,\square}_{\overline{\fM}^\tau}$.

We now define a map $\kappa^\tau: \Spf R^{\tau,\square}_{\overline{\fM}^\tau} \ra \Spf R^\tau_\rhobar$.
The natural map $\Spf R^\tau_\rhobar \ra \Spf R^\square_{\rhobar|_{G_{L_\infty}}}$ induced by restriction is a closed embedding by \cite[Proposition 3.12]{LLLM}.
The map 
\begin{align*}
\Spf R^{\tau,\square}_{\overline{\fM}^\tau} &\ra \Spf R^\square_{\rhobar|_{G_{L_\infty}}} \\
(f,b) &\mapsto (\mathbb{V}^*_\dd(\fM^\tau \otimes_{ R^\tau_{\overline{\fM}^\tau}, f} R \otimes_{\mO_L[\![u]\!]} \mO_{\mE,K'}),b).
\end{align*}
factors through $\Spf R^\tau_\rhobar$ by the proof of \cite[Corollary 5.20]{CLKisinmodules}.
We define this factoring to be $\kappa^\tau$.

\begin{thm}
The map $\kappa^\tau$ is formally smooth.
\end{thm}
\begin{proof}
This follows from \cite[\S 5.2 and 6]{LLLM} (cf. \cite[Proposition 3.4]{Le} for more details).
One first shows that the induced map $\Spf R^{\tau,\square}_{\overline{\fM}^\tau} \ra \Spf R^\square_\rhobar$ is the composition of a formally smooth map and a closed immersion (see e.g.~the proof of \cite[Theorem 3.6]{Le}).
Next, one shows that the scheme-theoretic image is $\Spf R^\tau_\rhobar$.
Since $R^{\tau,\square}_{\overline{\fM}^\tau}$ and $R^\tau_\rhobar$ are both $\mO_E$-flat with trivial nilradicals, it suffices to show that the map
\[
\Spec R^{\tau,\square}_{\overline{\fM}^\tau}[p^{-1}] \ra \Spec R^\tau_\rhobar[p^{-1}] 
\]
is surjective on $\overline{E}$-points. 
By \cite[Proposition 2.4.8]{Kisin-moduli}, it suffices to show that any Kisin module over $\mO_{\overline{E}}$ with descent data $\tau$ is isomorphic to a specialization of $\fM^\tau$.
This follows from the proof of \cite[Proposition 3.4]{Le} (there is one new case, namely $A^{(j)} = \begin{pmatrix} 1 & \\ & v\end{pmatrix}$, which is easily handled).
\end{proof}

\subsubsection{Comparing deformation rings.}\label{sec:comp}

Suppose that $\tau$ and $\tau_*$ are tame types such that $\sigma(\tau) = R_{s_\rhobar w}(\mu-s_\rhobar w' \eta)$ and $\sigma(\tau_*) = R_{s_\rhobar w_*}(\mu-s_\rhobar w'_* \eta)$ are as in Proposition \ref{pp:intersect} so that in particular $(w_j,w'_j)$ and $(w_{*,j},w'_{*,j})$ are not equal to $(s,\id)$ for all $j\in \bZ/f$.
Suppose furthermore that 
\begin{equation}\label{weightcontainment}
\text{for all $j \in \bZ/f$, $(w_j,w'_j)\neq (w_{*,j},w'_{*,j})$ implies that $w_j \neq \id$.}
\end{equation}
\begin{lemma}\label{lemma:containment}
(\ref{weightcontainment}) is equivalent to the fact that 
\[
W(\rhobar) \cap \JH(\overline \tau_*) \subseteq W(\rhobar) \cap \JH(\overline \tau).
\]
\end{lemma}
\begin{proof}
This is an easy consequence of Proposition \ref{pp:intersect}. 
\end{proof}
By Lemma \ref{lemma:containment} and~\cite[Theorem~7.1.1]{EGS}, there is a closed embedding
\[
\Spec R^{\tau_*}_\rhobar / p \to \Spec R^\tau_\rhobar / p.
\]
On the other hand, the definition of our family of Kisin modules implies that 
\[
I(\tau)+p\mO_E[\![X_j, Y_j, X_\alpha, X_{\alpha'}]\!] \subset I(\tau_*)+p\mO_E[\![X_j, Y_j, X_\alpha, X_{\alpha'}]\!].
\]
This induces a surjection 
\begin{equation}\label{eqn:compare}
R^\tau_{\overline{\fM}^\tau}/p \ra R^{\tau_*}_{\overline{\fM}^{\tau_*}}/p.
\end{equation}

\begin{pp}\label{pp:compare} 
The diagram
\[
\begin{tikzcd}
\Spf R^{\tau_*,\square}_{\overline{\fM}^{\tau_*}}/p \arrow{r}{(\ref{eqn:compare})} \arrow{d}{\kappa^\tau} & \Spf R^{\tau,\square}_{\overline{\fM}^\tau}/p \arrow{d}{\kappa^{\tau_*}} \\
\Spf R^{\tau_*}_\rhobar/p \arrow{r} & \Spf R^\tau_\rhobar/p
\end{tikzcd}
\]
commutes.

\end{pp}
\begin{proof}
It suffices to show that the map (\ref{eqn:compare}) commutes with the maps to $\Spf R_{\rhobar|_{G_{L_\infty}}}$.
This would follow from an isomorphism
\[
(\fM^\tau/p \otimes_{R^\tau_{\overline{\fM}^\tau}/p} R^{\tau_*}_{\overline{\fM}^{\tau_*}}/p\otimes_{\mO_{L'}[\![u]\!]} \mO_{\mE,K'})^{\Delta'} \cong (\fM^{\tau_*}/p \otimes_{\mO_{L_*'}[\![u]\!]} \mO_{\mE,K_*'})^{\Delta_*'}. 
\]
For $j\in \bZ/f$, let $A^{\tau,(j)}$ and $A^{\tau_*,(j)}$ be the matrices used to define $\fM^\tau$ and $\fM^{\tau_*}$, respectively, as above.
For $j\in \bZ/f$, let $B^{\tau,(j)}$ and $B^{\tau_*,(j)}$ be the matrices for $\fM^\tau \otimes_{\mO_{L'}[\![u]\!]} \mO_{\mE,K'}$ and $\fM^{\tau_*} \otimes_{\mO_{L'}[\![u]\!]} \mO_{\mE,K'_*}$, respectively, defined as in Proposition \ref{pp:kphi}.
Then by Proposition \ref{pp:kphi}, it suffices to show that 
\begin{equation}\label{eqn:matching}
B^{\tau,(j)} \equiv B^{\tau_*,(j)} \pmod{I(\tau_*)+(p)}.
\end{equation}
If $(w_{f-1-j},w'_{f-1-j}) = (w_{*,{f-1-j}},w'_{*,{f-1-j}})$, then it is easy to see that $A^{\tau,(j)}=A^{\tau_*,(j)}$ and $B^{\tau,(j)} = B^{\tau_*,(j)}$.

Now suppose that $(w_{f-1-j},w'_{f-1-j}) \neq (w_{*,{f-1-j}},w'_{*,{f-1-j}})$.
We check the (\ref{eqn:matching}) in cases.
If $x_j \neq 0$, then 
\begin{align*}
B^{\tau,(j)} &= \begin{pmatrix}
	v^{\mu_{1,f-1-j}} & -p(X_j+[x_j])^{-1}v^{\mu_{2,f-1-j}-1} \\
	(X_j+[x_j])v^{\mu_{1,f-1-j}} & v^{\mu_{2,f-1-j}} \\
 \end{pmatrix} \quad \textrm{and} \\
B^{\tau_*,(j)} &= \begin{pmatrix}
	(v+p)v^{\mu_{1,f-1-j}-1} &  \\
	(X_j+[x_j])v^{\mu_{1,f-1-j}} & v^{\mu_{2,f-1-j}} \\
 \end{pmatrix}.
\end{align*}
If $x_j = 0$, then
\[
B^{\tau,(j)} = \begin{pmatrix}
	v^{\mu_{1,f-1-j}} & -Y_jv^{\mu_{2,f-1-j}-1} \\
	X_jv^{\mu_{1,f-1-j}} & v^{\mu_{2,f-1-j}} \\
 \end{pmatrix}
\]
and $B^{\tau_*,(j)}$ is either
\[
 \begin{pmatrix}
	(v+p)v^{\mu_{1,f-1-j}-1} &  \\
	X_jv^{\mu_{1,f-1-j}} & v^{\mu_{2,f-1-j}} \\
 \end{pmatrix} \quad \textrm{or} \quad
 \begin{pmatrix}
	v^{\mu_{1,f-1-j}} & -Y_jv^{\mu_{1,f-1-j}} \\
	 & (v+p)v^{\mu_{2,f-1-j}-1} \\
 \end{pmatrix}.
\]
Then (\ref{eqn:matching}) is easy to check and the proposition follows.
\end{proof}

\subsubsection{Ideals and Serre weights.}\label{sec:iserrewts}

Recall the definition of the set $J_\rhobar$ from \S \ref{sec:serrewts}.
Also recall the definition of $\sigma_J \in W(\rhobar)$ for a subset $J$ of $J_\rhobar$.
Let $I(\sigma_J)$ be the ideal generated by $(\varpi_E,\{f_j(J)\}_j)$ in either $R^{\tau}_{\overline{\fM}^\tau}$ or $R^{\tau,\square}_{\overline{\fM}^\tau}$ where 
\[
f_j(J) = 
\begin{dcases}
X_j &\mbox{ if } f-1-j \in J\\
Y_j &\mbox{ if } f-1-j \notin J.
\end{dcases}
\]

\begin{pp}\label{pp:serrewt}
Suppose that $\sigma \in W(\rhobar)$ and $\tau$ is a tame type such that $\sigma(\tau)$ is isomorphic to $R_{s_\rhobar w}(\mu - s_\rhobar w'\eta)$ with $w$ and $w' \in W$ as in Proposition \ref{pp:intersect}.
Then $I(\tau) \subset I(\sigma)$ if and only if $\sigma \in \JH(\overline{\sigma(\tau)})$.
\end{pp}
\begin{proof}
Suppose that $\sigma(\tau)$ is as above. 
Then $\sigma_J \in \JH(\overline{\sigma(\tau)})$ is equivalent to the following statement by Proposition \ref{pp:jh}: for all $j\in \bZ/f$, 
\begin{itemize}
\item $f-1-j \in J$ implies that $(w_{f-1-j},w'_{f-1-j}) \neq (\id,\id)$; and
\item $f-1-j \notin J$ implies that $(w_{f-1-j},w'_{f-1-j}) \neq (\id,s)$.
\end{itemize}
(Compare the proof of Proposition~\ref{pp:intersect}.)
Now suppose that $\sigma = \sigma_J$ so that in particular, $f-1-j\in J$ implies that $f-1-j \in J_\rhobar$.
Then the above is equivalent to the fact that for all $j\in \bZ/f$, $f_j(\tau) \subset I(\sigma_J)$.
\end{proof}

\subsection{Local-global compatibility in families.}\label{sec:LLC}

We first recall the local Langlands correspondence for $\GL_n(L)$ from the introduction of \cite{HT}.
Fix an isomorphism $\iota: \lbar \bQ_p \ra \bC$.
Then we can define a local Langlands correspondence $\mathrm{rec}_p$ for $\GL_n(L)$ over $\lbar\bQ_p$ by requiring that $\iota \circ \mathrm{rec}_p = \mathrm{rec} \circ \iota$ where $\mathrm{rec}$ defined in \emph{loc.~cit.}~gives a bijection from isomorphism classes of smooth irreducible representations of $\GL_n(L)$ over $\bC$ to Frobenius semisimple Weil-Deligne representations of the Weil group $W_L$ over $\bC$.
(More specifically, $\rec$ is the correspondence over~$\bC$ that preserves $L$-functions and $\epsilon$-factors.)
We define $r_p(\pi)$ to be $\mathrm{rec}_p(\pi\otimes |\det|^{(n-1)/2})$.
Then $r_p$ does not depend on the choice of $\iota$.
(The definition of $r_p$ differs from that in \cite{CEGGPS} by a twist.)

Let $\sigma$ be the smooth $\sigma(\tau)$.
Let $\mH(\sigma)$ denote the Hecke algebra 
\[
\End_{\GL_2(L)}(\ind_{\GL_2(\mO_L)}^{\GL_2(L)} \sigma)
\]
for $\sigma$.
We write 
\[
\eta: \mH(\sigma) \ra R^{\tau}_\rhobar[p^{-1}]
\]
for the map from \cite[Theorem 4.1]{CEGGPS}, but normalized using $r_p$ defined above. 

We now specialize to the case when $\sigma$ is $\theta(\chi)_E$, the principal series tame type of $\GL_2(L)$ over $E$ given by induction of an $\mO_E$-character $\chi = \chi_1\otimes \chi_2$ of the Iwahori subgroup $I$ to $\GL_2(\mO_L)$.
Suppose that $\tau$ is the principal series tame inertial type such that $\sigma(\tau) \cong \theta(\chi)_E$.
Recall that $U_p$ is the element of $\mH(\theta(\chi)_E)$ associated to the $I$-double coset of $\fourmatrix{p}{0}{0}{1}$.
Suppose that $x\in \Spec R^\tau_\rhobar(\lbar E)$ gives a Galois representation $\rho_x:G_L \ra \GL_2(\lbar E)$.
We have that $\mathrm{WD}(\rho_x)^{F\text{-ss}}$ is of the form $(\tld{\chi}_1 \oplus \tld{\chi}_2) \circ \mathrm{Art}_L^{-1}$ (and $N=0$) where $\tld{\chi}_j: L^\times \ra \lbar E^\times$ is an extension of $\chi_j$ for $j =1$ and $2$. 
Let $\pi_x$ be $r_p^{-1}(\mathrm{WD}(\rho_x)^{F\text{-ss}})$ so that $\pi_x$ is isomorphic to $\Ind_{B(L)}^{\GL_2(L)} \tld{\chi}_1 \otimes \tld{\chi}_2 |\cdot|^{-1}$.
As in the proof of Lemma \ref{U_peigenvalues}, $U_p$ acts on $\pi_x^{I,\chi_1}$ by $q\tld{\chi}_1(p)$.
\cite[Theorem 4.1]{CEGGPS} implies that $\eta(U_p)$ at $x$ is $p^f$ times the $\Art_L(p)$-eigenvalue on the $\chi_1$-isotypic line.

Let $\eta'$ be the composition
\[
\mH(\sigma) \overset{\eta}{\ra} R^{\tau}_\rhobar[p^{-1}] \overset{\kappa^\tau}{\ra} R^{\tau,\square}_{\overline{\fM}^\tau}[p^{-1}].
\]

\begin{pp}\label{pp:etaup}
We have
\begin{equation}\label{eqn:etaup}
\eta'(U_p) = \prod_j p(x^{(j)}_{s_{j+1}^{-1}(1)s_{j+1}^{-1}(1)})^{-1} \pmod{v}.
\end{equation}
where $A^{\sigma, (j)}=(x^{(j)}_{ik})_{ik}$ gives $\fM^\tau$ as in \S \ref{sec:defk}.
\end{pp}
\begin{proof}
Let $\fM$ correspond to an $\overline{E}$-point of $\Spec R^{\tau,\square}_{\overline{\fM}^\tau}[p^{-1}]$ and set
\[
\rho = \mathbb{V}^*_\dd(\fM\otimes_{\mO_L[\![u]\!]} \mO_{\mE,K}).
\]
By Proposition \ref{pp:dpst} and Corollaries \ref{cor:frob} and \ref{cor:frob'}, $\varphi$ acts on the $\chi_1^{-1}$-isotypic line of $\fM/u \cong D_{\pst}(\rho^\vee)$ by $\prod_j x^{(j)}_{s_{j+1}^{-1}(1)s_{j+1}^{-1}(1)} \pmod{v}$. 
By the definition of $\mathrm{WD}(-)$ in \cite{fontaine}, $\Art_L(p)$ acts on the $\chi_1^{-1}$-isotypic line in $\mathrm{WD}(\rho^\vee)$ by $\prod_j x^{(j)}_{s_{j+1}^{-1}(1)s_{j+1}^{-1}(1)} \pmod{v}$. 
Then since $\Art_L(p)$ acts on the $\chi_1$-isotypic line in $\mathrm{WD}(\rho)$ by $\prod_j (x^{(j)}_{s_{j+1}^{-1}(1)s_{j+1}^{-1}(1)})^{-1} \pmod{v}$, $\eta'(U_p) = \prod_j p(x^{(j)}_{s_{j+1}^{-1}(1)s_{j+1}^{-1}(1)})^{-1} \pmod{v}$ for any $\overline{E}$-point of $\Spec R^{\tau,\square}_{\overline{\fM}^\tau}[p^{-1}]$.
The result now follows from the fact that $R^{\tau,\square}_{\overline{\fM}^\tau}[p^{-1}]$ is a reduced Jacobson ring whose maximal ideals have residue field of finite dimension over~$E$. (These properties are inherited from $R_\rhobar^\tau[p^{-1}]$, for example.)
\end{proof}

\begin{pp} \label{pp:x^j}
With the notation above, $p(x^{(j)}_{s_{j+1}^{-1}(1)s_{j+1}^{-1}(1)})^{-1} \pmod{v}$ is one of
\begin{equation}\label{eqn:diag}
1, \, p, \, -X_j, \, Y_j, \,-( X_j+[x_j]), \, \textrm{ and } p(X_j+[x_j])^{-1}.
\end{equation}
if $j \neq 0$ and the product of one of $(X_\alpha+[\alpha])^{-1}$ and $(X_{\alpha'}+[\alpha'])^{-1}$ and one of (\ref{eqn:diag}) if $j=0$.
In particular, $\eta'(U_p) \in R^{\tau,\square}_{\overline{\fM}^\sigma}$.
\end{pp}
\begin{proof}
This follows by inspection of~(\ref{eqn:deform}).
\end{proof}

\begin{rk}
Note that the last two possibilities $-(X_j+[x_j])$ and $p(X_j+[x_j])^{-1}$ can be removed if $\rhobar$ is semisimple.
\end{rk}

\subsubsection{Comparison with the central inertial type.} \label{sec:central}

\begin{defn}
If $w_0\in W$ is the longest element (nontrivial in all embeddings), then we define the \emph{central inertial type} of~$\rhobar$ to be the tame inertial type $\tau$ such that $\sigma(\tau) = R_{s_{\rhobar}w_0}(\mu-s_{\rhobar}w_0\eta)$.
\end{defn}

By Proposition~\ref{pp:intersect}, we have that $\JH(\overline{\sigma(\tau)})$ coincides with $W(\rhobar^\semis)$, which contains $W(\rhobar)$. 
Hence, whenever~$\sigma = \sigma_\mathrm{sm} = \theta(\chi)_E$ is a principal series tame type as above, and $\tau(\chi)$ is the principal series tame inertial type such that $\sigma(\tau(\chi)) \cong \theta(\chi)_E$, the condition~(\ref{weightcontainment}) holds and there exists a commutative diagram as in Proposition~\ref{pp:compare} (substituting~$\tau(\chi)$ for~$\tau_*$). We are going to define an element of~$R^{\tau,\square}_{\overline{\fM}^\tau}$ that specializes to a normalized version of~$\eta'(U_p)$.

\begin{defn}
Using Proposition \ref{pp:x^j}, let $\tld{U}'_p$ be the product $\prod_j \tld{x}^{(j)} \in R^{\tau,\square}_{\overline{\fM}^\tau}$, where
\begin{equation}\label{eqn:tldx}
\tld{x}^{(j)} = 1, \text{resp.~$1$}, Y_j, -X_j, (X_j+[x_j])^{-1},-(X_j+[x_j])
\end{equation}
if
\[
x^{(j)}_{s_{j+1}^{-1}(1)s_{j+1}^{-1}(1)} = 1, \text{ resp.~$p$}, X_j, -Y_j, X_j+[x_j], -p(X_j+[x_j])^{-1}
\]
if $j \neq 0$ and the product of the appropriate one of $(X_\alpha+[\alpha])^{-1}$ and $(X_{\alpha'}+[\alpha'])^{-1}$ and one of (\ref{eqn:tldx}) if $j=0$.
\end{defn}
\noindent
Again, $\tld{x}^{(j)}$ cannot be $(X_j+[x_j])^{-1}$ or $-(X_j+[x_j])$ if $\rhobar$ is semisimple.
Note that $p$ does not divide $\tld{U}'_p$.

\begin{pp}\label{pp:reduceup}
Let $e(\chi)\in \bN$ be such that $p^{-e(\chi)} \eta'(U_p)\in R^{\sigma,\square}_{\overline{\fM}^\sigma}$, but is not divisible by $p$. Then the image of $\tld{U}'_p$ in $R^{\sigma,\square}_{\overline{\fM}^\sigma}/p$ under the map (\ref{eqn:compare}) is the image of $p^{-e(\chi)} \eta'(U_p)$.
\end{pp}
\begin{proof}
This is a direct consequence of Proposition~\ref{pp:etaup}.
\end{proof}

\subsubsection{Fixed determinant deformation rings.}

In applications, it is often necessary to use deformation rings parameterizing lifts with fixed determinant.
Let $\rhobar$ be as in \S \ref{Galoisreps}.
Let $\psi: G_L \ra \mO_E^\times$ be a character lifting $\det \rhobar \overline{\varepsilon}^{-1}$.
For an inertial type $\tau$, let $R_\rhobar^{\psi,\tau}$ be the quotient of $R_\rhobar^\tau$ parameterizing lifts of $\rhobar$ of Galois type $\tau$, parallel Hodge--Tate weights $(0,1)$, and determinant $\psi \varepsilon$.
This quotient is nonzero if and only if $\psi|_{I_L} \cong \det \tau$.
If $\tau$ is a tame inertial type with $R^\tau_\rhobar$ nonzero, we similarly define $R^{\psi,\tau}_{\overline{\fM}^\tau}$ and $R^{\psi,\tau,\square}_{\overline{\fM}^\tau}$.
We let $I^\psi(\sigma)$ denote the image of $I(\sigma)$ in $R^{\psi,\tau}_{\overline{\fM}^\tau}$ and $R^{\psi,\tau,\square}_{\overline{\fM}^\tau}$.
Let $\tld{U}^{\psi,\prime}_p$ be the image of $\tld{U}_p'$ in $R^{\psi,\tau,\square}_{\overline{\fM}^\tau}$.
Then the fixed determinant analogues of Propositions \ref{pp:compare} and \ref{pp:reduceup} hold compatibly with the natural quotient maps from deformation rings to fixed determinant deformation rings.

%% file: global.tex
\section{Global applications.} \label{sec:global}

We now give examples of $\GL_2(L)$-representations $\pi_{\mathrm{glob}}(\rhobar^\vee)$ from global contexts for which there exists an $\mO_E[\GL_2(L)]$-module $M_\infty$ with an arithmetic action of $R_\infty^\psi$ (for some character $\psi: G_L \ra \mO_E^\times$ and some natural number $h$) such that $\pi_{\mathrm{glob}}(\rhobar^\vee) \cong M_\infty^\vee[\fm]$.
The construction of $M_\infty$ comes from the Taylor--Wiles--Kisin method as augmented in \cite{CEGGPS} and simplified in \cite{Scholze}.
The $\GL_2(L)$-representations $\pi_{\mathrm{glob}}(\rhobar^\vee)$ we define appear in the cohomology of Shimura curves and in spaces of algebraic modular forms on definite quaternion algebras over totally real fields.
One could also consider algebraic modular forms on definite unitary groups (see, e.g., \cite[\S 6.6]{EGS}).

\subsection{Spaces of mod $p$ automorphic forms on some quaternion algebras}\label{sec:modpaut}

Let $p$ be an odd prime.
Let $F$ be a totally real number field in which $p$ is unramified. Let $D_{/F}$ be a quaternion algebra which is unramified at all places dividing $p$ and at zero (the \emph{definite} case) or one infinite place (the \emph{indefinite} case).
If $D_{/F}$ is indefinite (unramified at exactly one infinite place) and $K = \prod_w K_w \subset (D\otimes_F \bA_F^\infty)^\times$ is an open compact subgroup, then there is a smooth projective curve $X_K$ defined over $F$ and we define $S(K,k_E)$ to be the cohomology group $H^1(X_K \times_F \lbar F, k_E)$. 
If $D_{/F}$ is definite, then we let $S(K,k_E)$ be the space of $K$-invariant smooth functions
\[f: D^\times\backslash (D\otimes_F \mathbb{A}_F^\infty)^\times \ra k_E.\]

Let $\rbar: G_F \ra \GL_2(k_E)$ be a Galois representation.
Let $S$ be the union of the set of places in $F$ where $\rbar$ is ramified, the set of places where $D$ is ramified, and the set of places dividing $p$.
Since $p>3$, we can choose a place $w_1$ of $F$ such that $\rbar$ is unramified at $w_1$ and $\rbar(\Frob_{w_1})$ has distinct eigenvalues (where $\Frob_{w_1}$ denotes a geometric Frobenius element at $w_1$).
Enlarging $E$ if necessary, we assume that these eigenvalues are in $k_E$.
Let $\mathbb{T}^{S,\mathrm{univ}}$ be the commutative polynomial algebra over $\mO_E$ generated by the formal variables $T_w$ and $S_w$ for each $w \notin S\cup \{w_1\}$ where $w_1$ is chosen as in \cite[\S 6.2]{EGS}.
Then $\mathbb{T}^{S,\mathrm{univ}}$ acts on $S(K,k_E)$ with $T_w$ and $S_w$ acting by the usual double coset action of
\[ \big[ \GL_2(\mO_{F_w}) \begin{pmatrix}
  \varpi_w &   \\
   & 1 \\
 \end{pmatrix}\GL_2(\mO_{F_w}) \big] \]
and 
\[ \big[ \GL_2(\mO_{F_w}) \begin{pmatrix}
  \varpi_w &   \\
   & \varpi_w \\
 \end{pmatrix}\GL_2(\mO_{F_w}) \big], \]
respectively.
Define a map $\mathbb{T}^{S,\mathrm{univ}}\ra k_E$ such that the image of $X^2 - (\bN w)^{-1} T_w X + (\bN w)^{-1} S_w$ in $k_E[X]$ is the characteristic polynomial of $\rbar(\Frob_w)$, where $\Frob_w$ is a geometric Frobenius element at $w$. 
Let the kernel of this map be $\fm_{\rbar}$.

\begin{rk}\label{rk:rbar}
Given a maximal ideal of $\mathbb{T}^{S,\mathrm{univ}}$, \cite[\S 9]{Breuilmodp} (resp.~\cite[\S 6.2]{EGS}) defines a Galois representation which is dual to (resp.~a twist by $\varepsilon^{-1}$ of) the representation $\rbar$ that we define. 
This is the reason for the presence of a dual in Definition \ref{defn:piglob} and the different normalization of the map $\eta$ in \S \ref{sec:LLC}.
\end{rk}

For the rest of the section, suppose that
\begin{enumerate}
\item $\rbar$ is modular, i.e. that there exists a compact open subgroup $K\subset (D\otimes_F \bA_F^\infty)^\times$ such that $S(K,k_E)_{\fm_{\rbar}}$ is nonzero;
\item $\rbar|_{G_{F(\zeta_p)}}$ is absolutely irreducible;
\item if $p=5$ then the image of $\rbar(G_{F(\zeta_p)})$ in $\PGL_2(k_E)$ is not isomorphic to $A_5$; 
\item $\rbar|_{G_{F_w}}$ is generic (see \S \ref{Galoisreps}, though the weaker \cite[Definition 2.1.1]{EGS} suffices) for all places $w|p$; and
\item $\rbar|_{G_{F_w}}$ is non-scalar at all finite places where $D$ ramifies.
\end{enumerate}

Let $v|p$ be a place of $F$, and let $L$ be $F_v$.
Let $\rhobar$ be the restriction $\rbar|_{G_{F_v}}$.
We define a compact open subgroup $K^v = \prod_{w\neq v} K_w \subset (D\otimes_F \mathbb{A}_F^{v,\infty})^\times$ as in \cite[\S 6.5]{EGS}.
The subgroup $K_w$ is $(\mO_D)_w^\times$ unless $\rbar|_{G_{F_w}}$ is ramified (which is always the case if $w$ divides $p$ by the genericity assumption), reducible, and $D$ is unramified at $w$.
If these three conditions hold and moreover $w|p$ or $w = w_1$, then $K_w$ is the standard (upper triangular) Iwahori subgroup.
If these three conditions hold, but $w\nmid p$, then $K_w$ is the subgroup of matrices which are upper triangular modulo $\varpi_w^{n_w}$ where $\varpi_w$ is a uniformizer of $\mO_{F_w}$ and $n_w$ is defined in \emph{loc. cit.}

Following \emph{loc. cit.}, we let $S$ be the set of places $w$ where $\rbar|_{G_{F_w}}$ is ramified, $D$ is ramified at $w$, or $w$ is $w_1$.
For each $w\in S$ distinct from $v$, one defines a  finite free $\mO_E$-module $V_w$ with a smooth $K_w$-action denoted $L_w$ in \emph{loc. cit.}
Rather than define $V_w$, we only note that the $\mO_E$-rank of $V_w$ is one unless $\rbar|_{G_{F_w}}$ is irreducible.
For any compact open subgroup $K_v \subset \GL_2(L)$, let $V$ be the $\mO_E[\![K^vK_v]\!]$-module $\bigotimes_{w\in S,\, w\neq v} V_w$ where $K_w$ acts trivially if $w \notin S$ or $w$ is $v$.

Fix a finite order character $\psi: G_F \ra W(k_E)^\times\subset \mO_E^\times$ such that $\det \rbar = \overline{\varepsilon\psi}$.
We extend the action of $K^vK_v$ to $K^vK_v (\mathbb{A}_F^\infty)^\times$ by letting $(\mathbb{A}_F^\infty)^\times$ act by $\psi \circ \mathrm{Art}_F$ (the compatibility of the actions here follows from the condition in the choice of $\psi$).

We let $S'$ be the set of places $w$ away from $v$ where $\rbar|_{G_{F_w}}$ is reducible and $w|p$, $D$ is ramified at $w$, or $w$ is $w_1$.
For $w\in S'$, one also introduces Hecke operators $T_w$ and scalars $\beta_w \in k_E^\times$, in the same way as in~\emph{loc. cit.}
Let $\mathbb{T}^{S,S',\mathrm{univ}}$ be the free polynomial algebra over $\mathbb{T}^{S,\mathrm{univ}}$ generated by the variables $(T_w)_{w\in S'}$.
Let $\fm'_{\rbar} \subset \mathbb{T}^{S,S',\mathrm{univ}}$ be the maximal ideal generated by $\fm_{\rbar}$ and $(T_w - \beta_w)_{w\in S'}$.

Fix a compact open subgroup $K_v \subset \GL_2(L)$.
In the definite case, we let $S(K^vK_v,V)$ be the space of functions
\[f: D^\times\backslash (D\otimes_F \mathbb{A}_F^\infty)^\times \ra V^\vee\]
such that $f(gu) = u^{-1}f(g)$ for all $g\in (D\otimes_F \mathbb{A}_F^\infty)^\times$ and $u\in K^vK_v (\mathbb{A}_F^\infty)^\times$.
In the indefinite case, $V^\vee$ defines a local system $\mathcal{F}_{V^\vee}$ on $X_{K^vK_v}$, and we let $S(K^vK_v,V)$ be the cohomology group $H^1(X_{K^vK_v} \times_F \overline{F},\mathcal{F}_{V^\vee})$.
In either case, we then let $S(K^v,V)$ be the limit 
\[
\varinjlim_{K_v} S(K^vK_v,V).
\]
Then $S(K^v,V)$ has natural commuting actions of $\mathbb{T}^{S,S',\mathrm{univ}}$ and $\GL_2(L)$.
In the indefinite case, there is also an action of $G_F$ that commutes with the two aforementioned actions.
We let $\pi_{\mathrm{glob}}(\rhobar^\vee)$ be $S(K^v,V)[\fm'_{\rbar}]$ in the definite case and $\Hom_{G_F}(\rbar,S(K^v,V)[\fm'_{\rbar}])$ in the indefinite case.
By the assumption that $\rbar$ is modular and the choices of $K^v$, $L$, and $(\beta_w)_{w\in S'}$, we know that $\pi_{\mathrm{glob}}(\rhobar^\vee)$ is nonzero.
We emphasize that despite the suggestive notation, it is far from clear that $\pi_{\mathrm{glob}}(\rhobar^\vee)$ depends only on $\rhobar^\vee$.
\begin{thm}\label{thm:patching}
There exists a natural number $h$ and an $\mO_E[\GL_2(L)]$-module $M_\infty$ with an arithmetic action of $R_\infty^\psi$ such that $M_\infty^\vee[\fm]$ is isomorphic to $\pi_{\mathrm{glob}}(\rhobar^\vee)$ (where $\psi$ denotes the restriction $\psi|_{G_{F_v}}$).
\end{thm}
The proof of this theorem uses the Taylor--Wiles method described in the next section.

\subsection{The Taylor--Wiles patching construction.}\label{sec:TW}

\subsubsection{Galois deformation rings}
For each $w\in S$, let $\rbar_w$ denote $\rbar|_{G_{F_w}}$ and $\psi_w$ denote $\psi_{G_{F_w}}$.
Let $R_{\rbar_w}^{\square,\psi_w}$ be the framed deformation ring parametrizing liftings of $\rbar_w$ with determinant $\psi_w \varepsilon$.
Let $R_S^\psi$ be $\widehat{\otimes}_{w\in S,\mO_E} R_{\rbar_w}^{\square,\psi_w}$.

Fix a finite set of places $Q$ disjoint from $S$.
For any finite set $T$ of places in $F$, let $G_{F,T}$ denote the Galois group of the maximal extension of $F$ unramified outside of $T$ over $F$.
Let $R_{F,S_Q}^\psi$ denote the deformation ring of $\rbar$ as a $G_{F,S\cup Q}$-representation with fixed determinant $\psi\varepsilon$.
Let $r(Q): G_{F,S\cup Q} \ra \GL_2(R_{S_Q}^{\psi})$ be a versal lifting of $\rbar$.
Let $R_{F,S_Q}^{\square,\psi}$ denote the complete $\mO_E$-algebra which prorepresents the functor assigning to an local Artinian $\mO_E$-algebra $A$ with residue field $k_E$ the set of equivalence classes of tuples $(r,\{\alpha_w\}_{w\in S})$ where $r$ is an $A$-lifting of $\rbar$ as a $G_{F,S\cup Q}$-representation and $\alpha_w \in \ker(\GL_2(A) \ra \GL_2(k_E))$ for each $w\in S$.

\subsubsection{Spaces of automorphic forms}
Let $K' = \prod_w K'_w \subset (D\otimes_F \mathbb{A}_F^\infty)^\times$ be a compact open subgroup such that $K'_w = K_w$ for all $w \notin Q$ and $w\neq v$.
Let $\mathbb{T}(K')$ denote the image of $\mathbb{T}^{S\cup Q,S',\mathrm{univ}}$ in $\End_{\mO_E}(S(K',V))$.
We denote the image of $\fm'_{\rbar} \cap \mathbb{T}^{S\cup Q,\mathrm{univ}}$ in $\mathbb{T}(K')$ by $\fm_Q$---it is a maximal ideal in $\mathbb{T}(K')$.
There is a map $R_{F,S_Q}^\psi \ra \mathbb{T}(K')_{\fm_Q}$ such that the image of the characteristic polynomial of $r(Q)(\Frob_w)$ is $X^2 - (\bN w)^{-1}T_w X + (\bN w)^{-1} S_w \in \mathbb{T}(K')_{\fm_Q}[X]$ for $w \notin S\cup Q \cup \{w_1\}$.
Let $r(K'): G_{F,S\cup Q} \ra \GL_2(\mathbb{T}(K')_{\fm_Q})$ be the Galois representation obtained from $r(Q)$.

\subsubsection{Auxiliary primes}
We can and do choose an integer $q \geq [F:\bQ]$ and for each integer $N\geq 1$ a tuple $(Q_N,\{\overline{\psi}_w\}_{w\in Q_N})$ where
\begin{itemize}
\item $Q_N$ is a set of $q$ places of $F$ disjoint from $S$;
\item for each $w\in Q_N$, $\bN w \equiv 1 \pmod{p^N}$;
\item for each $w\in Q_N$, $\rbar|_{G_{F_w}} \cong \overline{\psi}_w \oplus \overline{\psi}'_w$ for some $\overline{\psi}'_w \neq \overline{\psi}_w$; and
\item the ring $R_{F,S_{Q_N}}^{\square,\psi}$ can be topologically generated over $R_S^\psi$ by $q-[F:\bQ]$ elements.
\end{itemize}

Fix a compact open subgroup $K_v \subset \GL_2(L)$.
For each $N\geq 1$, let 
\[
K(N)^v = \prod_{w\neq v} K(N)_w \subset (D\otimes_F \mathbb{A}_F^{v,\infty})^\times
\]
be the compact open subgroup such that $K(N)_w = K_w$ for $w \notin Q_N$ and $w\neq v$ and $K(N)_w$ is the subgroup of matrices of $\GL_2(\mO_{F_w})$ congruent to $\fourmatrix{*}{*}{0}{1}$ modulo $\varpi_w$ if $w\in Q_N$.

Choose a lift $\varphi_w \in G_{F_w}$ of the geometric Frobenius element, and let $\varpi_w \in \mO_{F_w}$ be the uniformizer such that $\mathrm{Art}_{F_w} \varpi_w = \varphi_w|_{F_w^{\mathrm{ab}}}$.
For $w\in Q_N$, let $U_w$ be the Hecke operator corresponding to the double coset
\[
K(N)_w \fourmatrix{1}{0}{0}{\varpi_w} K(N)_w.
\]

Using that $\overline{\psi}_w(\varphi_w) \neq \overline{\psi}'_w(\varphi_w)$, the characteristic polynomial $P_w(X) \in \mathbb{T}(K(N)^vK_v)[X]$ of $r(K(N)^vK_v)(\varphi_w)$ factors as $(X-A_w)(X-B_w)$ by Hensel's lemma.
Let $\mathbb{T}(K(N)^vK_v)'$ be the $\mathbb{T}(K(N)^vK_v)$-subalgebra of $\End_{\mO_E}(S(K(N)^v K_v,V))$ generated by $U_w$ for $w\in Q_N$, and let $\fm_{Q_N}' \subset \mathbb{T}(K(N)^vK_v)'$ be the ideal generated by $\fm_{Q_N}$ and $U_w - A_w$.
For each $w\in Q_N$, let $k_w^\times(p)$ denote the $p$-Sylow subgroup of $k_w^\times$.
Let $\mO_{E,N}$ denote the group algebra of $\prod_{w\in Q_N} k_w^\times(p)$ over $\mO_E$ with augmentation ideal $\fa_N$.
Note that the natural inclusion composed with multiplication by $\prod_{w\in Q_N} (U_w-B_w)$ gives an isomorphism
\[
S(K^vK_v,V)_{\fm} \cong S(K(N)^v K_v,V)_{\fm'_Q}[\fa_N].
\]

In the definite case, let $M(K_v,N)$ be 
\[
S(K(N)^vK_v,V)_{\fm'_Q}^\vee \otimes_{R_{F,S_{Q_N}}^\psi} R_{F,S_{Q_N}}^{\square,\psi}.
\]
In the indefinite case, let $M(K_v,N)$ be 
\[
\Hom_{\mathbb{T}(K(N)^vK_v)_{\fm_{Q_N}}[G_{F,S\cup Q_N}]}(r(K(N)^vK_v),S(K(N)^vK_v,V)_{\fm'_Q})^\vee\otimes_{R_{F,S_{Q_N}}^\psi} R_{F,S_{Q_N}}^{\square,\psi}.
\]

\subsubsection{Patching}

For $w\in S$ with $w\neq v$, define $R_w^{\min}$ as in \cite[\S 6.5]{EGS}.
Let $R^{\mathrm{loc}}$ be 
\[
R_{\rhobar}^{\square,\psi_v}\widehat{\otimes}_{\mO_E}(\widehat{\otimes}_{w \in S,w\neq v,\mO_E} R_w^{\min}).
\]
Note that $R^{\mathrm{loc}}$ is formally smooth over $R_{\rhobar}^{\square,\psi_v}$.

Let $g$ be $q-[F:\bQ]$.
Let $R_\infty^\psi$ be $R^{\mathrm{loc}}[\![t_1,\ldots,t_g]\!]$, which is isomorphic to $R_{\rhobar}^{\square,\psi_v}\widehat{\otimes}_{\mO_E} \mO_E[\![x_1,\ldots,x_h]\!]$ for some integer $h\geq 0$.
We fix a versal lifting 
\[
r(\varnothing): G_{F,S} \ra \GL_2(R_{F,S}^\psi)
\]
of $\rbar$ with determinant $\psi\varepsilon$ and versal liftings 
\[
r(Q_N): G_{F,S} \ra \GL_2(R_{F,S}^\psi)
\] 
of $\rbar$ with determinant $\psi\varepsilon$ compatible with $r(\varnothing)$.
This gives an isomorphism 
\[
R_{F,S}^{\square,\psi} \cong R_{F,S}^\psi\widehat{\otimes}_{\mO_E}\mO_E[\![z_1,\ldots,z_{4\#S-1}]\!]
\]
and compatible isomorphisms
\begin{equation}\label{eqn:versal}
R_{F,S_{Q_N}}^{\square,\psi} \cong R_{F,S_{Q_N}}^\psi\widehat{\otimes}_{\mO_E}\mO_E[\![z_1,\ldots,z_{4\#S-1}]\!].
\end{equation}

For each $N\geq 1$, choose a surjection of $R^{\mathrm{loc}}$-algebras
\[
R_\infty^\psi \twoheadrightarrow R_{F,S_{Q_N}}^{\square,\psi}
\]
and a surjection $\mO_{E,\infty} := \mO_E[\![y_1,\ldots,y_q]\!] \twoheadrightarrow \mO_{E,N}$ whose kernel is contained in the ideal generated by $((1+y_i)^{p^N}-1)_{i=1}^q$.
This gives an action of the $\mO_{E,\infty}$-algebra $S_\infty := \mO_E[\![z_1,\ldots,z_{4\# S-1},y_1,\ldots,y_q]\!]$ on $M(K_v,n)$ by \eqref{eqn:versal}. 
Let $J \subset S_\infty$ be an open ideal for the adic topology defined by the maximal ideal, and let $I_J \subset \bN$ be the cofinite subset of elements $N$ such that $J$ contains the kernel of $\mO_{E,\infty} \twoheadrightarrow \mO_{E,N}$.
If $N\in I_J$, let $M(K_v,J,N)$ be $S_\infty/J \otimes_{S_\infty} M(K_v,N)$.

Let $(S_\infty/J)_{I_J}$ be the product $\prod_{i\in I_J} S_\infty/J$.
Fix a nonprincipal ultrafilter $\mathfrak{F}$ on the set $\bN = \{N\geq 1\}$.
This defines a point in $x \in \Spec (S_\infty/J)_{I_J}$ (see \cite[Lemma 2.2.2]{GN}).
Then let $M(K_v,J,\infty)$ be $(S_\infty/J)_{I_J,x} \otimes_{(S_\infty/J)_{I_J}} \prod_{N\in I_J} M(K_v,J,N)$. 
The product $\prod_{N\in I_J} M(K_v,J,N)$, and thus $M(K_v,J,\infty)$, has a diagonal $R_\infty^\psi$-action.
For an open subideal $J' \subset J$ and an open compact subgroup $K'_v \subset K_v$, there is a natural map $M(K'_v,J',\infty)\ra M(K_v,J,\infty)$ (see \cite[Lemma 3.4.11(1)]{GN}).
We let $M_\infty$ be the $R_\infty^\psi$-module
\[
\varprojlim_{J,K_v} M(K_v,J,\infty).
\]

\begin{proof}[Proof of Theorem \ref{thm:patching}]
There is a natural right $\mO_E[\GL_2(L)]$-action (by duality) on $M_\infty$ coming from the maps $M(K'_v,J,\infty)\ra M(K_v,J,\infty)$ for any inclusion $K'_v \subset K_v$ of compact open subgroups (see \cite[\S 9]{Scholze}).
Since the action of $L^\times$ on $M_\infty$ factors through $\psi_v$ on each $M(K_v,N)$, it does on $M_\infty$ as well.
Let $\fa\subset S_\infty$ be the augmentation ideal.
Then $M(K_v,N)/\fa$ is isomorphic to $S(K^vK_v,V)_{\fm}^\vee$ in the definite case and 
\[
\Hom_{\mathbb{T}(K^vK_v)_{\fm'_{\rbar}}[G_{F,S}]}(r(K^vK_v),S(K^vK_v,V)_{\fm'_{\rbar}})^\vee
\]
in the indefinite case, which is of dimension say $d(K_v)$. 
We conclude by topological Nakayama's lemma that if $U_v \subset \GL_n(\mO_L)$ is compact, open, and pro-$p$, then there is a surjection $S_\infty/J[\![U_v/K_v]\!]^{d(U_v)} \twoheadrightarrow M(K_v,J,N)$ for each $K_v \subset U_v$ and $N \in I_J$. 
Taking products and localizing, we have a surjection 
\begin{equation} \label{eqn:generate}
S_\infty/J[\![U_v/K_v]\!]^{d(U_v)} \cong (S_\infty/J)_{I_J,x} \otimes_{(S_\infty/J)_{I_J}} \prod_{N\in I_J} S_\infty/J[\![U_v/K_v]\!]^{d(U_v)} \twoheadrightarrow M(K_v,J,\infty)
\end{equation}
where the first isomorphism is given by the diagonal map.
Taking limits (which is right exact since both sides of \eqref{eqn:generate} are finite abelian groups) we see that $M_\infty$ is generated as a $S_\infty[\![U_v]\!]$-module by $d(U_v)$ elements and
in particular that $M_\infty$ is a finitely generated $S_\infty[\![\GL_2(\mO_L)]\!]$-module.
The $S_\infty$-action factors through $R_\infty^\psi$ so that $M_\infty$ is a finitely generated $R_\infty^\psi[\![\GL_2(\mO_L)]\!]$-module.

For proofs of statements in this paragraph, see \cite[Lemma 4.18]{CEGGPS}.
The finitely generated $S_\infty[\![\GL_2(\mO_L)]\!]$-module $M_\infty$ is in fact projective (see  also \cite[Proposition 3.4.16]{GN}).
If $\theta$ is a lattice in $\sigma(\tau)$, then $M_\infty(\theta)$ is supported in $R_\infty^\psi(\tau)$.
Since $\dim S_\infty = \dim R_\infty^\psi(\tau)$, $M_\infty(\theta)$ is maximal Cohen-Macaulay over $R_\infty^\psi(\tau)$.
The action of $\mathcal{H}(\theta)$ on $M_\infty(\theta)$ factors through the map $\eta$ by the proof of \cite[Theorem 4.19]{CEGGPS}.
Our different normalization of $\eta$ stems from our normalization of Galois representations (see Remark \ref{rk:rbar})
For an $\mO_E$-lattice $\theta \subset \sigma(\tau)$ in an inertial $K$-type, $M_\infty(\theta)[1/p]$ is locally free of rank at most one over $R_\infty(\tau)[1/p]$ by \cite[Proposition 3.5.1]{BD}.

Finally, in the definite case, we have that
\[
M_\infty^\vee[\fm] \cong (M_\infty/\fa)^\vee[\fm] \cong S(K^v,V)_{\fm'_{\rbar}}[\fm'_{\rbar}] = \pi_{\mathrm{glob}}(\rhobar^\vee).
\]
In the indefinite case, we have that 
\begin{align*}
(M_\infty/\fa)^\vee[\fm] &\cong \varinjlim_{K_v}\Hom_{\mathbb{T}(K^vK_v)_{\fm'_{\rbar}}[G_{F,S}]}(r(K^vK_v),S(K^vK_v,V)_{\fm'_{\rbar}})[\fm'_{\rbar}] \\
&\cong \Hom_{k_E[G_{F,S}]}(\rbar,S(K^v,V)[\fm'_{\rbar}]) \\
&= \pi_{\mathrm{glob}}(\rhobar^\vee).
\end{align*}
\end{proof}

%% file: Diagrams.bbl
\newcommand{\etalchar}[1]{$^{#1}$}
\providecommand{\bysame}{\leavevmode\hbox to3em{\hrulefill}\thinspace}
\providecommand{\MR}{\relax\ifhmode\unskip\space\fi MR }
\providecommand{\MRhref}[2]{%
  \href{http://www.ams.org/mathscinet-getitem?mr=#1}{#2}
}
\providecommand{\href}[2]{#2}
\begin{thebibliography}{LLHLM20}

\bibitem[BD14]{BD}
Christophe Breuil and Fred Diamond, \emph{Formes modulaires de {H}ilbert modulo
  {$p$} et valeurs d'extensions entre caract\`eres galoisiens}, Ann. Sci.
  \'{E}c. Norm. Sup\'{e}r. (4) \textbf{47} (2014), no.~5, 905--974.
  \MR{3294620}

\bibitem[BDJ10]{BDJ}
Kevin Buzzard, Fred Diamond, and Frazer Jarvis, \emph{On {S}erre's conjecture
  for mod {$\ell$} {G}alois representations over totally real fields}, Duke
  Math. J. \textbf{155} (2010), no.~1, 105--161. \MR{2730374}

\bibitem[BP12]{BPmodp}
Christophe Breuil and Vytautas Pa{\v s}k{\=u}nas, \emph{Towards a modulo {$p$}
  {L}anglands correspondence for {${\rm GL}_2$}}, Mem. Amer. Math. Soc.
  \textbf{216} (2012), no.~1016, vi+114. \MR{2931521}

\bibitem[Bre03]{BreuilmpLLC}
Christophe Breuil, \emph{Sur quelques repr\'{e}sentations modulaires et
  {$p$}-adiques de {${\rm GL}_2(\bold Q_p)$}. {I}}, Compositio Math.
  \textbf{138} (2003), no.~2, 165--188. \MR{2018825}

\bibitem[Bre11]{Breuilparameters}
\bysame, \emph{Diagrammes de {D}iamond et {$(\phi,\Gamma)$}-modules}, Israel J.
  Math. \textbf{182} (2011), 349--382. \MR{2783977}

\bibitem[Bre14]{Breuilmodp}
\bysame, \emph{Sur un probl{\`e}me de compatibilit{\'e} local-global modulo
  {$p$} pour {${\rm GL}_2$}}, J. Reine Angew. Math. \textbf{692} (2014), 1--76.
  \MR{3274546}

\bibitem[CDM18]{CDM}
Xavier Caruso, Agn{\`e}s David, and Ariane M{\'e}zard, \emph{Un calcul
  d'anneaux de d{\'e}formations potentiellement {B}arsotti--{T}ate}, Trans.
  Amer. Math. Soc. \textbf{370} (2018), no.~9, 6041--6096. \MR{3814324}

\bibitem[CDP14]{CDP}
Pierre Colmez, Gabriel Dospinescu, and Vytautas Pa\v{s}k\={u}nas, \emph{The
  {$p$}-adic local {L}anglands correspondence for {${\rm GL}_2(\Bbb Q_p)$}},
  Camb. J. Math. \textbf{2} (2014), no.~1, 1--47. \MR{3272011}

\bibitem[CEG{\etalchar{+}}16]{CEGGPS}
Ana Caraiani, Matthew Emerton, Toby Gee, David Geraghty, Vytautas Pa{\v
  sk\=u}nas, and Sug~Woo Shin, \emph{Patching and the {$p$}-adic local
  {L}anglands correspondence}, Camb. J. Math. \textbf{4} (2016), no.~2,
  197--287. \MR{3529394}

\bibitem[CEG{\etalchar{+}}18]{CEGGPS2}
\bysame, \emph{Patching and the {$p$}-adic {L}anglands program for {${\rm
  GL}_2(\Bbb Q_p)$}}, Compos. Math. \textbf{154} (2018), no.~3, 503--548.
  \MR{3732208}

\bibitem[CL18]{CLKisinmodules}
Ana Caraiani and Brandon Levin, \emph{Kisin modules with descent data and
  parahoric local models}, Ann. Sci. \'Ec. Norm. Sup\'er. (4) \textbf{51}
  (2018), no.~1, 181--213. \MR{3764041}

\bibitem[Col10]{Colmez}
Pierre Colmez, \emph{Repr\'{e}sentations de {${\rm GL}_2(\bold Q_p)$} et
  {$(\phi,\Gamma)$}-modules}, Ast\'{e}risque (2010), no.~330, 281--509.
  \MR{2642409}

\bibitem[EGH13]{EGH}
Matthew Emerton, Toby Gee, and Florian Herzig, \emph{Weight cycling and
  {S}erre-type conjectures for unitary groups}, Duke Math. J. \textbf{162}
  (2013), no.~9, 1649--1722. \MR{3079258}

\bibitem[EGS15]{EGS}
Matthew Emerton, Toby Gee, and David Savitt, \emph{Lattices in the cohomology
  of {S}himura curves}, Invent. Math. \textbf{200} (2015), no.~1, 1--96.
  \MR{3323575}

\bibitem[Eme11]{Emerton}
Matthew Emerton, \emph{Local-global compatibility in the $p$-adic {L}anglands
  program for $\textrm{GL}_2/\mathbf{Q}$}, preprint (2011).

\bibitem[Enn18]{Enns}
John Enns, \emph{On mod p local-global compatibility for
  $\mathrm{GL}_3(\mathbf{Q}_{p^f})$ in the ordinary case}, in preparation,
  2018.

\bibitem[Fon90]{fontaine}
Jean-Marc Fontaine, \emph{Repr\'esentations {$p$}-adiques des corps locaux.
  {I}}, The {G}rothendieck {F}estschrift, {V}ol. {II}, Progr. Math., vol.~87,
  Birkh\"auser Boston, Boston, MA, 1990, pp.~249--309. \MR{1106901}

\bibitem[GHS18]{GHS}
Toby Gee, Florian Herzig, and David Savitt, \emph{General {S}erre weight
  conjectures}, J. Eur. Math. Soc. (JEMS) \textbf{20} (2018), no.~12,
  2859--2949. \MR{3871496}

\bibitem[GLS15]{GLS}
Toby Gee, Tong Liu, and David Savitt, \emph{The weight part of {S}erre's
  conjecture for {$\mathrm{GL}(2)$}}, Forum Math. Pi \textbf{3} (2015), e2, 52.
  \MR{3324938}

\bibitem[GN]{GN}
Toby Gee and James Newton, \emph{Patching and the completed homology of locally
  symmetric spaces}, \url{https://arxiv.org/abs/1609.06965}, to appear in
  J.~Inst.~Math.~Jussieu.

\bibitem[Gro65]{EGA}
A.~Grothendieck, \emph{\'{E}l\'{e}ments de g\'{e}om\'{e}trie alg\'{e}brique.
  {IV}. \'{E}tude locale des sch\'{e}mas et des morphismes de sch\'{e}mas.
  {II}}, Inst. Hautes \'{E}tudes Sci. Publ. Math. (1965), no.~24, 231.
  \MR{0199181}

\bibitem[Her09]{herzig}
Florian Herzig, \emph{The weight in a {S}erre-type conjecture for tame
  {$n$}-dimensional {G}alois representations}, Duke Math. J. \textbf{149}
  (2009), no.~1, 37--116. \MR{2541127}

\bibitem[HLM17]{HLM}
Florian Herzig, Daniel Le, and Stefano Morra, \emph{On {${\rm mod}\,p$}
  local-global compatibility for {${\rm GL}_3$} in the ordinary case}, Compos.
  Math. \textbf{153} (2017), no.~11, 2215--2286. \MR{3705291}

\bibitem[HT01]{HT}
Michael Harris and Richard Taylor, \emph{The geometry and cohomology of some
  simple {S}himura varieties}, Annals of Mathematics Studies, vol. 151,
  Princeton University Press, Princeton, NJ, 2001, With an appendix by Vladimir
  G. Berkovich. \MR{1876802}

\bibitem[Hu10]{Hu10}
Yongquan Hu, \emph{Sur quelques repr\'{e}sentations supersinguli\`eres de
  {${\rm GL}_2(\Bbb Q_{p^f})$}}, J. Algebra \textbf{324} (2010), no.~7,
  1577--1615. \MR{2673752}

\bibitem[Hu16]{Hu}
\bysame, \emph{Valeurs sp\'{e}ciales de param\`etres de diagrammes de
  {D}iamond}, Bull. Soc. Math. France \textbf{144} (2016), no.~1, 77--115.
  \MR{3481262}

\bibitem[HW18]{HW}
Yongquan Hu and Haoran Wang, \emph{Multiplicity one for the mod $p$ cohomology
  of {S}himura curves: the tame case}, Math. Res. Lett. \textbf{25} (2018),
  no.~3, 843--873. \MR{3847337}

\bibitem[IR90]{IrelandRosen}
Kenneth Ireland and Michael Rosen, \emph{A classical introduction to modern
  number theory}, second ed., Graduate Texts in Mathematics, vol.~84,
  Springer-Verlag, New York, 1990. \MR{1070716}

\bibitem[Kis06]{KisinFcrys}
Mark Kisin, \emph{Crystalline representations and {$F$}-crystals}, Algebraic
  geometry and number theory, Progr. Math., vol. 253, Birkh\"auser Boston,
  Boston, MA, 2006, pp.~459--496. \MR{2263197 (2007j:11163)}

\bibitem[Kis08]{Kisinsemistable}
\bysame, \emph{Potentially semi-stable deformation rings}, J. Amer. Math. Soc.
  \textbf{21} (2008), no.~2, 513--546. \MR{2373358}

\bibitem[Kis09a]{KisinFM}
\bysame, \emph{The {F}ontaine--{M}azur conjecture for {${\rm GL}_2$}}, J. Amer.
  Math. Soc. \textbf{22} (2009), no.~3, 641--690. \MR{2505297}

\bibitem[Kis09b]{Kisin-moduli}
\bysame, \emph{Moduli of finite flat group schemes, and modularity}, Ann. of
  Math. (2) \textbf{170} (2009), no.~3, 1085--1180. \MR{2600871}

\bibitem[Kis10]{Kisindef}
\bysame, \emph{Deformations of {$G_{\Bbb Q_p}$} and {${\rm GL}_2(\Bbb Q_p)$}
  representations}, Ast\'{e}risque (2010), no.~330, 511--528. \MR{2642410}

\bibitem[Le19]{Le}
Daniel Le, \emph{Multiplicity one for wildly ramified representations}, Algebra
  Number Theory \textbf{13} (2019), no.~8, 1807--1827. \MR{4017535}

\bibitem[LLHL19]{LLL}
Daniel Le, Bao~V. Le~Hung, and Brandon Levin, \emph{Weight elimination in
  {S}erre-type conjectures}, Duke Math. J. \textbf{168} (2019), no.~13,
  2433--2506. \MR{4007598}

\bibitem[LLHLM18]{LLLM}
Daniel Le, Bao~V. Le~Hung, Brandon Levin, and Stefano Morra, \emph{Potentially
  crystalline deformation rings and {S}erre weight conjectures: shapes and
  shadows}, Invent. Math. \textbf{212} (2018), no.~1, 1--107. \MR{3773788}

\bibitem[LLHLM20]{LLLM2}
\bysame, \emph{Serre weights and {B}reuil's lattice conjecture in dimension
  three}, Forum Math. Pi \textbf{8} (2020), e5, 135. \MR{4079756}

\bibitem[LMP18]{LMP}
Daniel Le, Stefano Morra, and Chol Park, \emph{On {${\rm mod}\,p$} local-global
  compatibility for {${\rm GL}_3(\Bbb Q_p)$} in the non-ordinary case}, Proc.
  Lond. Math. Soc. (3) \textbf{117} (2018), no.~4, 790--848. \MR{3873135}

\bibitem[LMS]{LMS}
Daniel Le, Stefano Morra, and Benjamin Schraen, \emph{Multiplicity one at full
  congruence level}, \url{https://arxiv.org/abs/1608.07987}, to appear in
  J.~Inst.~Math.~Jussieu.

\bibitem[Pa{\v s}13]{Paskunas}
Vytautas Pa{\v s}k{\=u}nas, \emph{The image of {C}olmez's {M}ontreal functor},
  Publ. Math. Inst. Hautes \'{E}tudes Sci. \textbf{118} (2013), 1--191.
  \MR{3150248}

\bibitem[PQ18]{PQ}
Chol Park and Zicheng Qian, \emph{On mod p local-global compatibility for
  $\mathrm{GL}_n(\mathbf{Q}_p)$ in the ordinary case},
  \url{https://arxiv.org/abs/1712.03799}, preprint, 2018.

\bibitem[Sch18]{Scholze}
Peter Scholze, \emph{On the {$p$}-adic cohomology of the {L}ubin-{T}ate tower},
  Ann. Sci. \'{E}c. Norm. Sup\'{e}r. (4) \textbf{51} (2018), no.~4, 811--863,
  With an appendix by Michael Rapoport. \MR{3861564}

\end{thebibliography}
